\documentclass[a4paper, 11pt, oneside]{amsbook}
\title{The Scott adjunction \\ \tiny{Towards formal model theory}}
\author{Ivan Di Liberti}

\numberwithin{section}{chapter}
\numberwithin{subsection}{section}
\numberwithin{subsubsection}{subsection}

\usepackage{blindtext}
\usepackage{quoting}
\usepackage{amsthm}
\usepackage[utf8]{inputenc}
\usepackage{verbatim}
\usepackage{amsmath}
\usepackage{amsfonts}
\usepackage{stackrel}
\usepackage{amssymb}
\usepackage{graphicx}
\usepackage[svgnames]{xcolor}
\usepackage{lipsum}
\usepackage{svg}
\usepackage{bbm}
\usepackage{caption}
\usepackage{tcolorbox}
\usepackage{epigraph}
\usepackage{ebproof}
\usepackage[cal=boondoxo]{mathalfa}
\usepackage[all,cmtip]{xy}
\usepackage{mathtools}
\usepackage{color}
\usepackage{mathrsfs}
\usepackage{subfig}
\usepackage{framed}
\usepackage[colorlinks]{hyperref}
\usepackage{tikz}
\usepackage{tikz-cd}
\usepackage{multirow}
\usepackage[percent]{overpic}

\hypersetup{
    colorlinks = true,
    linkbordercolor = {red},
   	linkcolor ={black},
	anchorcolor = {pink},
	citecolor =  {red},
	filecolor = {teal},
	menucolor = {teal},
	runcolor =  {teal},
	urlcolor = {teal},
}

\usepackage[all,cmtip]{xy}

\usetikzlibrary{calc}
\usetikzlibrary{matrix,arrows}
\usepackage{tikz-cd}
\usepackage[square, sort, numbers]{natbib}

\usepackage{calligra} 

    {\endMakeFramed}

    {\endMakeFramed}

\makeatletter
\g@addto@macro\bfseries{\boldmath}
\makeatother

\theoremstyle{definition}
\newtheorem{thm}{Theorem}[section]
\newtheorem{exa}[thm]{Example}
\newtheorem{rem}[thm]{Remark}
\newtheorem{con}[thm]{Construction}
\newtheorem{ach}[thm]{Achtung!}
\newtheorem*{structure*}{Structure}
\newtheorem*{thm*}{Theorem}

\newtheorem{lem}[thm]{Lemma}
\newtheorem{conj}[thm]{Conjecture}
\newtheorem{prop}[thm]{Proposition}
\newtheorem{defn}[thm]{Definition}
\newtheorem{cor}[thm]{Corollary}
\newtheorem{notation}[thm]{Notation}
\newtheorem{question}[thm]{Question}

\newcommand\Sat{\operatorname{Sat}}

\newcommand\op{\circ}

\newcommand\lan{\mathsf{lan}}
\newcommand\ran{\mathsf{ran}}
\newcommand\Set{\operatorname{\bf Set}}

\newcommand\colim{\operatorname{colim}}

\newcommand\ca{\mathcal {A}}
\newcommand\cb{\mathcal {B}}
\newcommand\cc{\mathcal {C}}

\newcommand\cf{\mathcal {F}}

\newcommand\cg{\mathcal {G}}

\newcommand\ce{\mathcal {E}}

\newcommand\ck{\mathcal {K}}
\newcommand\cl{\mathcal {L}}

\newcommand\cs{\mathcal {S}}
\newcommand\ct{\mathcal {T}}
\newcommand\cp{\mathcal {P}}

\newcommand\cx{\mathcal {X}}
\newcommand\cy{\mathcal {Y}}

\def\dist{.75em}

\DeclareFontFamily{U}{min}{}
\DeclareFontShape{U}{min}{m}{n}{<-> udmj30}{}
\newcommand\yo{\!\text{\usefont{U}{min}{m}{n}\symbol{'210}}\!}

\begin{document}




\begin{center}
\noindent
\makebox[\textwidth]{
\begin{tikzpicture}
\node {\Huge Masaryk University};
\node at (0,-.75) {\huge Faculty of Science};
\node at (0,-1.25) {------------------------------------------------------};

\node at (-5.75,-.65) {\includegraphics[scale=.26]{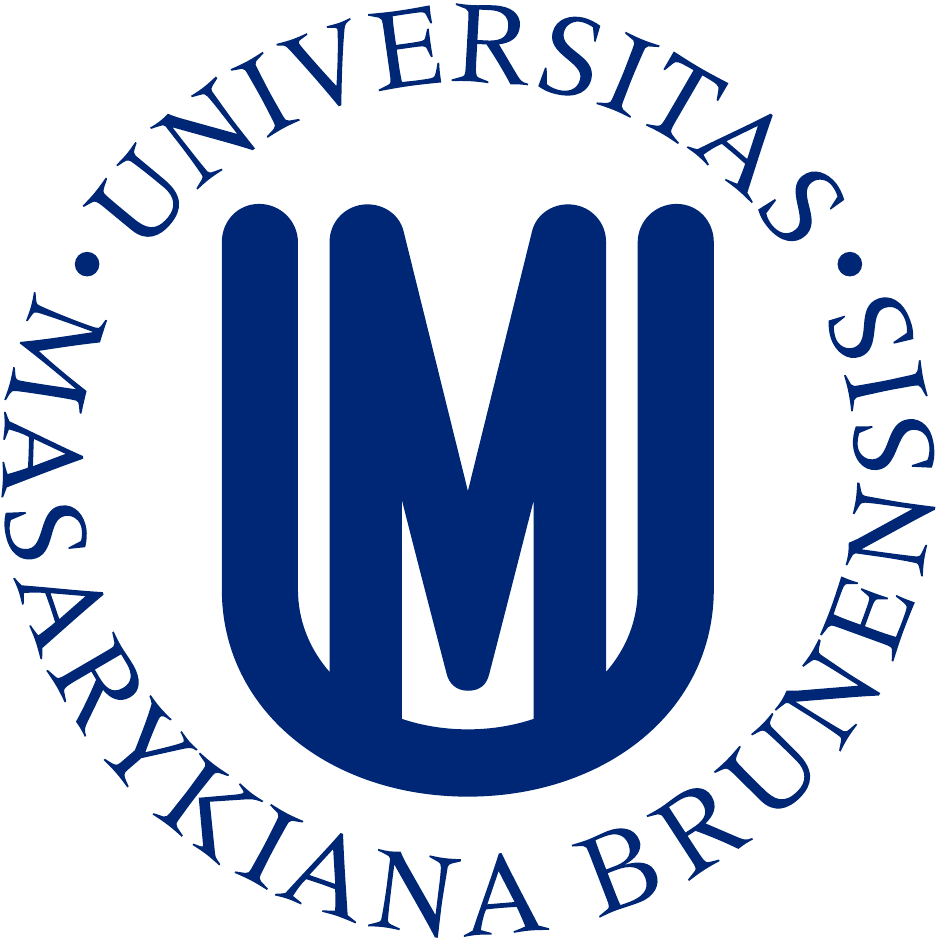}};
\node at (5.75,-.65) {\includegraphics[scale=.26]{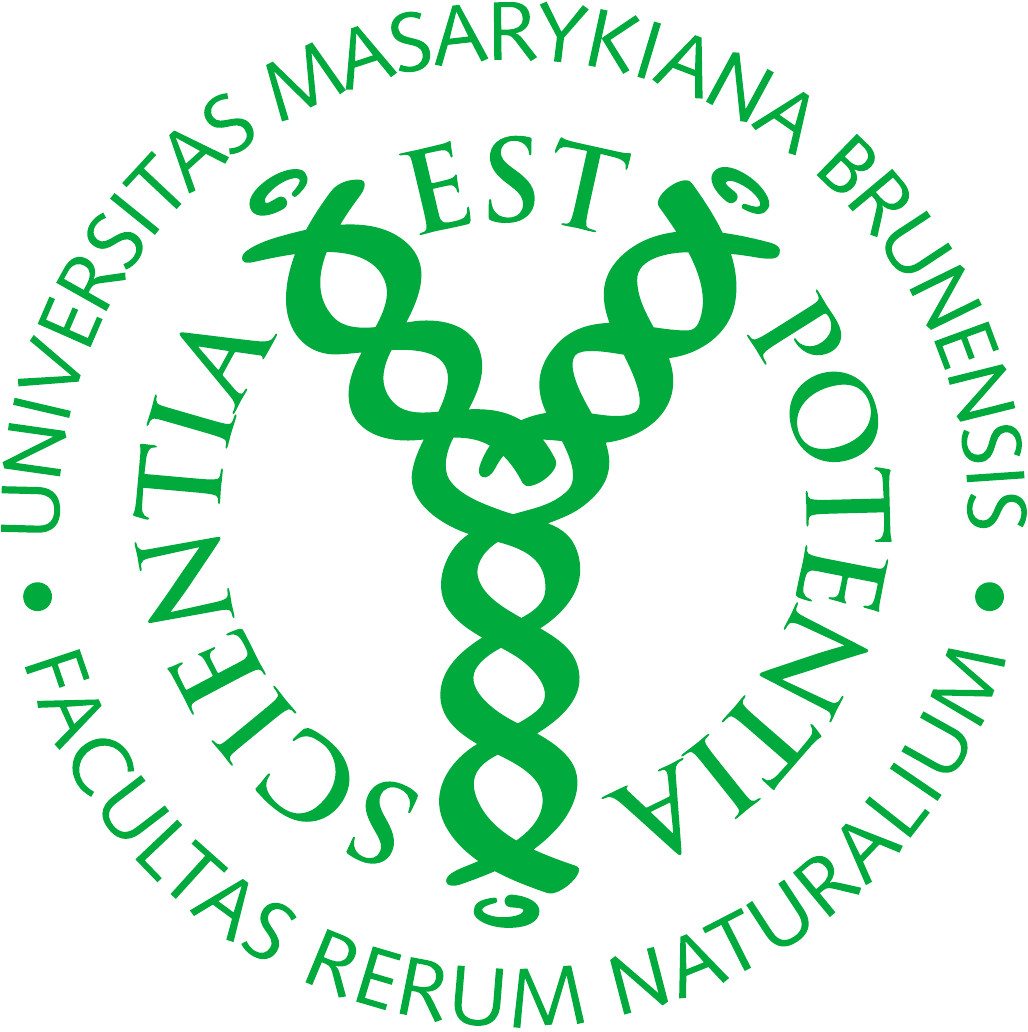}};
\end{tikzpicture}
}

\vspace{\fill}

\huge \textbf{The Scott adjunction} \\ \small{Towards formal model theory} \\[2mm]
  \Large  Ph.D thesis\\
   \Large Ivan Di Liberti\\
   \end{center}
  \vspace{\fill}
 Advisor: Prof. RNDr. Ji\v{r}í Rosick\'y, DrSc.      \\  Department of Mathematics and Statistics \hfill Brno 2020
\thispagestyle{empty}
\newpage

\frontmatter

\chapter*{Abstract}

We introduce and study the Scott adjunction, relating accessible categories with directed colimits to topoi. Our focus is twofold, we study both its applications to formal model theory and its geometric interpretation. From the geometric point of view, we introduce the categorified Isbell duality, relating bounded (possibly large) ionads to topoi. The categorified Isbell duality interacts with the Scott adjunction offering a categorification of the Scott topology over a poset (hence the name). We show that the categorified Isbell duality is idempotent, similarly to its uncategorified version. From the logical point of view, we use this machinery to provide candidate (geometric) axiomatizations of accessible categories with directed colimits. We discuss the connection between these adjunctions and the theory of classifying topoi. We relate our framework to the more classical theory of abstract elementary classes. From a more categorical perspective, we show that the $2$-category of topoi is enriched over accessible categories with directed colimits and we relate this result to the Scott adjunction.

\thispagestyle{empty}

  \chapter*{Acknowledgments}

  \section*{Scientific Acknowledgments}
During these years of work on this thesis several colleagues and more experienced researchers have contributed to my scientific growth and their mark has impacted on my style and interests in ways that I could not have imagined. In the following list I will try to convey all the people that  have been relevant in this process. Hopefully I will not forget anyone. I am indebted to:
\begin{itemize}
	\item  My advisor, \textbf{Ji\v{r}í Rosick\'y}, for the freedom and the trust he blessed me with during these years, not to mention his sharp and remarkably blunt wisdom.
	\item \textbf{Simon Henry}, for his collaboration in those days in which this thesis was nothing but an informal conversation at the whiteboard.
  \item The readers of this thesis, \textbf{Richard Garner} and \textbf{Tibor Beke}, for the time and the effort that they have dedicated to providing me their suggestions, corrections and comments. I am especially grateful to Richard for unveiling some inaccuracies in my original definition of generalized ionad. His keen comments have made the whole thesis sharper and more precise in several points that just cannot be listed.
	\item \textbf{Peter Arndt}, for his sincere interest in my research, and the hint of looking at the example of the geometric theory of fields.
	\item \textbf{Axel Osmond}, for sharing his energy and enthusiasm during my visit at the IRIF, and discussing about the non-idempotency of the Scott adjunction.
	\item \textbf{Fosco Loregian}, for introducing me to formal techniques in category theory.
	\item The \textbf{scientific committee} of the CT2019, for giving me the opportunity of presenting my work at such an important conference.
	\item \textbf{John Bourke}, for his reference storm on the topic of algebras for a $2$-monad.

\end{itemize}
The list above is  mentioning only those that have been relevant to the sole production of this thesis. Yet, my Ph.D thesis was not the only project on which I have worked during my studies, and indeed my studies have not been the only human activity that I have performed during my stay in Brno. A more complete list of acknowledgments will follow below in my mother-tongue.

\section*{Ringraziamenti}

Mi piacerebbe dedicare le primissime righe di questi ringraziamenti al formulare delle scuse: sono debitore a tutti i lettori che vorranno avventurarsi per il malfermo inglese che pervade la tesi. Questa è in effetti la ragione principale per cui questi ringraziamenti saranno scritti in una lingua che padroneggio meglio, la mia speranza è quella che almeno per questi non avrò l'impressione di non essermi saputo esprimere come avrei voluto. Sono altresì cosciente di costringere le persone nominate a fronteggiare una lingua che loro stesse non conoscono, d'altro canto, non vedo come il problema di una comunicazione efficace possa essere risolto senza spargimento di sangue alcuno.
 \subsection*{Ai maestri}

  \begin{itemize}
  	\item   ringrazio il mio relatore, \textbf{Ji\v{r}í Rosick\'y}, per l'occasione che mi ha concesso rispondendo alla mia \textit{email esplorativa} nel lontano Agosto 2016,  per il candore col quale ha puntualmente spento ognuna delle mie giovani angosce, per il sorriso colpevole con cui di tanto in tanto mi ha ripetuto \textit{it was good to make that mistake, one always has to learn}, per la fiducia paterna con cui mi ha lasciato commettere ogni errore possibile, e per l'inconfondibile accento con cui pronuncia \textit{somewhat}.
	 \item  \textbf{Wieslaw Kubiś}, per l'inspiegabile stima che mi riserva, la disponibilità con cui nel 2017 mi ha accolto nel suo gruppo di ricerca, e per la rassicurante mortificazione che provo a confrontare la mia conoscenza di topologia generale con la sua. La concretezza della sua matematica è una fonte inesauribile di (contro)esempi che a volte passano inosservati allo sguardo di un categorista.
  	\item \textbf{Simon Henry},  per la semplicità con cui ha ricoperto il ruolo di esempio. La fermezza con cui riporta la generalità più assoluta all'intuizione più concreta è stata la prova giornaliera che la matematica non è che una scienza sperimentale, in barba agli idealismi più sciocchi che la vogliono fatta d'una pasta diversa.
  \end{itemize}

  \subsection*{Ai compagni di viaggio e gli incontri felici}

  \begin{itemize}
  	\item ringrazio \textbf{la mia famiglia}, per aver sostenuto senza aver mai vacillato la mia discutibilissima volontà di inseguire la carriera accademica e per aver sopportato stoicamente i lunghi periodi in cui non ci siamo visti.
  	\item \textbf{Fosco Loregian}, per la cocciutaggine con cui difende i più arroccati idealismi, per la devozione con cui promuove la teoria dei profuntori, per le sue sconfinate conoscenze di \LaTeX \  e per avermi aspettato un po', quando ancora non conoscevo le estensioni di Kan.
  	\item \textbf{Christian Espíndola}, per i fallimentari tentativi di correggere il mio inglese, per il pessimismo che faceva da sfondo ad ogni nostra conversazione, per i caffé e i ristoranti. 
  	\item \textbf{Julia Ramos González}, per la gentilezza con cui ha risposto alla mia domanda alla fine del suo talk a \textit{Toposes in Como} e per la sinergia che accompagna la nostra collaborazione.
  	\item \textbf{Giulio Lo Monaco}, per i sensi di colpa che mi prendono quando lo sento parlare il suo ceco \textit{perfetto} e per essermi stato compagno nel nostro fallimentare tentativo di provare che ogni $\infty$-categoria bicompleta proviene da una categoria ABC.
  	\item \textbf{Sebastien Vasey}, per la cura e la disponibilità con cui mi ha aiutato durante le application per i postdoc, oltre che per la gentilezza con cui mi si è sempre rivolto.
  	\item \textbf{Axel Osmond}, per non farmi sentire troppo solo nell'avere degli interessi scientifici poco di moda, confido che in un futuro prossimo avremo modo di lavorare insieme.
  	\item \textbf{Edoardo Lanari}, per avermi rivalutato e per essersi fatto rivalutare. La stima reciproca che è inspiegabilmente germogliata fra noi è uno dei fiori più graditi del 2019.
  	\item \textbf{Peter Arndt}, per la semplicità, quasi ingenua, e l'entusiasmo con cui discute di matematica. 
  	\item \textbf{Alessio Santamaria}, per la pacatezza con cui bilancia e smorza i toni, per l'equilibrio e per le conversazioni pacificate.
	\item le segretarie del dipartimento di matematica, \textbf{Milada, Jitka, Radka e Vlaďka}, per la pazienza materna con cui hanno sopportato le mie continue goffagini burocratiche.
  	\item \textbf{Beata Kubiś}, per essere stata il mio principale punto di riferimento all'Accademia delle Scienze di Praga. Confido non me ne voglia in alcun modo Wielsaw, il senso pratico di Beata la rendeva a dir poco imbattibile. 
  	\item \textbf{Sam van Gool}, per la grande occasione che mi ha concesso invitandomi a Parigi. Confido che avremo modo di proseguire su questa strada.
  	\item  \textbf{Paul-André Melliès}, che sfortunatamente non conosco a sufficenza da avere l'ardire d'annoverarlo fra i maestri. 
	\item \textbf{Sandra Mantovani,  Beppe Metere, Davide Di Micco, \newline Emanuele Pavia, Junior Cioffo e Andrea Montoli}, per avermi fatto respirare un po' d'aria di casa durante le mie visite nel sempre più familiare dipartimento di Matematica dell'Università di Milano.
		\item  \textbf{Soichiro Fujii, Ingo Blechschmidt, Sina Hazratpour}, tre colleghi che hanno costantemente assicurato una compagnia gradita durante le conferenze.
	\item \textbf{Tim Campion}, per esserci tenuti compagnia su MathOverflow.
	\item \textbf{Ji\v{r}í Ad\'amek}, per essere stato uno \textit{zio} lontano, e per il suo irriducibile sorriso.
	\item \textbf{C\&P}, per le ciarle.
	\item le conoscenze fatte a Brno: \textbf{Betta, Jomar, Roberto, Elisa e Cristiano}.
	\item la crew del ristorante \textbf{Franz} e del caffé \textbf{Tivoli}.
	\item un ultimo ringraziamento va ad un firmamento di stelle fisse che mi hanno fatto compagnia durante questi anni moravi, ognuno dalla località in cui l'ho lasciato, o in cui s'è trasferito. Questi non me ne vogliano se non sono stati espressamente citati.

  \end{itemize}

\thispagestyle{empty}

\chapter*{Statement of originality}

\thispagestyle{empty}

Hereby I declare that this thesis is my original authorial work, which I have worked out on my own. All sources, references, and literature used or excerpted during elaboration of this work are properly cited and listed in complete reference to the due source.
\vfill
 \hfill  \\ 
Brno, 30 March 2020. \hfill \begin{overpic}[width=0.15\textwidth]{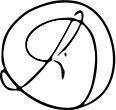}
 \put (-10,0) {Ivan Di Liberti}
\end{overpic}

\newpage

\thispagestyle{empty}
\vspace*{\fill}
\parbox{110mm}{Twas brillig, and the slithy toves \\
Did gyre and gimble in the wabe; \\
All mimsy were the borogoves, \\
And the mome raths outgrabe. \\

Jabberwocky, Lewis Carroll.}
\vfill

\thispagestyle{empty}

\newpage
\thispagestyle{empty}
\vspace*{\fill}

\hfill \textit{Ai miei nonni}, \\

\hfill \textit{da un nipote strafalario}.

\vfill
\tableofcontents

\chapter*{Introduction}

The scientific placing of this thesis is a point of contact between logic, geometry and category theory. For this reason, we have decided to provide two introductions. The  reader is free to choose which one to read, depending on their interests and background. Every introduction covers the whole content of the thesis, but it is shaped according to the expected sensibility of the reader. Also, they offer a \textit{sketch of the elephant}. After the two introductions, we proceed to an analysis of the content of the thesis.

\begin{structure*}
\begin{enumerate}
	\item[]
	\item to the geometer;
	\item to the logician;
	\item main results and structure of the thesis;
	\item intended readers;
	\item the rôle of the Toolbox.
	\end{enumerate}
\end{structure*}

\section*{To the geometer}

\subsection*{General setting}
Posets with directed suprema, topological spaces and locales provide three different approaches to geometry. Topological spaces are probably the most widespread concept and need no introduction, posets with directed suprema belong mostly to domain theory \cite{abramsky1994domain}, while locales are the main concept in formal topology and constructive approaches to geometry \cite{Vickers2007}.
\begin{center}
\begin{tikzcd}
                                                                      & \text{Loc} \arrow[lddd, "\mathbbm{pt}" description, bend left=12] \arrow[rddd, "\mathsf{pt}" description, bend left=12] &                                                                                                                 \\
                                                                      &                                                                                                                  &                                                                                                                 \\
                                                                      &                                                                                                                  &                                                                                                                 \\
\text{Top} \arrow[ruuu, "\mathcal{O}" description, dashed, bend left=12] &                                                                                                                  & \text{Pos}_{\omega} \arrow[luuu, "\mathsf{S}" description, dashed, bend left=12] \arrow[ll, "\mathsf{ST}", dashed]
\end{tikzcd}
\end{center}
These three incarnations of geometry have been studied and related (as hinted by the diagram above) in the literature, but they are far from representing the only existing approaches to geometry. Grothendieck introduced topoi as a generalization of locales of open sets of a topological space \cite{bourbaki2006theorie}. The geometric intuition has played a central rôle  since the very early days of this theory and has reached its highest peaks in the 90's. Since its introduction, topos theory has gained more and more consensus and has shaped the language in which modern algebraic geometry has been written. Garner has introduced the notion of ionad much more recently \cite{ionads}. Whereas a topos is a categorified locale, a ionad is a categorified topological space. This thesis studies the relationship between these two different approaches to higher topology and uses this geometric intuition to analyze accessible categories with directed colimits. 
We build on the analogy (and the existing literature on the topic) between \textit{plain} and higher topology and offer a generalization of the diagram above to \textbf{accessible categories with directed colimits, ionads and topoi}. 

\begin{center}
\begin{tikzcd}
                                                                      & \text{Topoi} \arrow[lddd, "\mathbbm{pt}" description, bend left=12] \arrow[rddd, "\mathsf{pt}" description, bend left=12] &                                                                                                                 \\
                                                                      &                                                                                                                  &                                                                                                                 \\
                                                                      &                                                                                                                  &                                                                                                                 \\
\text{BIon} \arrow[ruuu, "\mathbb{O}" description, dashed, bend left=12] &                                                                                                                  & \text{Acc}_{\omega} \arrow[luuu, "\mathsf{S}" description, dashed, bend left=12] \arrow[ll, "\mathsf{ST}", dashed]
\end{tikzcd}
\end{center}

The functors that regulate the interaction between these different approaches to geometry are the main characters of the thesis. These are the \textbf{Scott adjunction}, which relates accessible categories with directed colimits to topoi and the \textbf{categorified Isbell adjunction}, which relates bounded ionads to topoi. Since we have a quite good understanding of the topological case, we use this intuition to guess and infer the behavior of its categorification. In the diagram above, the only functor that has already appeared in the literature is $\mathbb{O}$, in \cite{ionads}. The adjunction $\mathsf{S} \dashv \mathsf{pt}$ was achieved in collaboration with Simon Henry and has appeared in \cite{simon} as a \textit{prequel} of this thesis. In the two following subsections, we discuss the basic characteristics of the above-mentioned adjunctions.

\subsection*{Idempotency}
The adjunction between topological spaces and locales is idempotent, and it restricts to the well known equivalence of categories between locales with enough points and sober spaces. We obtain an analogous result for the adjunction between bounded ionads and topoi, inducing an equivalence between sober (bounded) ionads and topoi with enough points. Studying the Scott adjunction is more complicated, and in general it does not behave as well as the categorified Isbell duality does, yet we manage to describe the class of those topoi which are fixed by the Scott-comonad.

\subsection*{Stone-like dualities $\&$ higher semantics}
 This topological picture fits the pattern of Stone-like dualities. As the latter is related to completeness results for propositional logic, its categorification is related to syntax-semantics dualities between categories of models and theories. Indeed a logical intuition on topoi and accessible categories has been available since many years \cite{sheavesingeometry}. A topos can be seen as a placeholder for a geometric theory, while an accessible category can be seen as a category of models of some infinitary theory. In the thesis we introduce a new point of view on ionads, defining \textbf{ionads of models} of a geometric sketch. This approach allows us to entangle the theory of classifying topoi with the Scott and Isbell adjunctions providing several comparison results in this direction.

\section*{To the logician}

\subsection*{General setting}
This thesis is mainly devoted to providing a categorical approach to reconstruction results for semantics. That means, given a category of models of some theory $\ca$, we provide a naturally associated (geometric) theory that is a candidate axiomatization of $\ca$, \[ \ca \mapsto \mathsf{S}(\ca).\] This approach has many different purposes, some practical and some very abstract. On the one hand, it provides a syntactic tool to study structural properties of $\ca$, on the other it contributes to the general study of \textbf{dualities of syntax-semantics type}. Given an accessible category $\ca$, that we would like to imagine as the category of models of some (first order) theory, we associate to it a geometric theory that is in a precise sense a best approximation of $\ca$ among geometric theories. The rôle of geometric theories is played in this thesis by Grothendieck topoi. 

\begin{center}
\begin{tikzcd}
\text{Acc}_\omega \arrow[rr, "\mathsf{S}" description, bend right=10] &                                                                                      & \text{Topoi} \arrow[ll, "\mathsf{pt}" description, bend right=10] \\
                                                   &                                                                                      &                                                \\
                                                   &                                                                                      &                                                \\
                                                   & \mathsf{Theories} \arrow[dotted, luuu, "\mathsf{Mod}(-)" description, bend left=20] \arrow[dotted, ruuu, "\gimel(-)" description, bend right=20] &                                               
\end{tikzcd}
\end{center}

Indeed (geometric) logic and topos theory has been related in the past \cite{caramello2010unification} \cite{sheavesingeometry}, and the diagram above displays the connection that we will describe in the thesis. Given an accessible category $\ca$ we build a topos $\mathsf{S}(\ca)$ (the Scott topos of $\ca$). Building on the constructions that connect topoi to geometric logic, we can use the tools of topos theory to study accessible categories. We also relate \textbf{abstract elementary classes} (AECs) to this general pattern, integrating the classical approach to axiomatic model theory into our framework.

\begin{tcolorbox}
\begin{thm*}[Thm. \ref{scottadj}, The Scott adjunction]
There is  a $2$-adjunction, $$\mathsf{S} : \text{Acc}_{\omega} \leftrightarrows \text{Topoi}: \mathsf{pt}. $$
\end{thm*}

\begin{thm*}[Thm. \ref{LDCAECs}] The Scott adjunction restricts to locally decidable topoi and AECs.
\[\mathsf{S}: \text{AECs} \leftrightarrows \text{LDTopoi}: \mathsf{pt}\]
\end{thm*}
\end{tcolorbox}

\subsection*{A Stone-like framework}
As \textbf{Stone duality} offers a valid approach to the study of semantics for propositional logic, we study the Scott adjunction from a topological perspective in order to infer logical consequences. This amounts to the categorification of the adjunction between locales and topological spaces, together with its interaction with Scott domains. 

\begin{center}
\begin{tikzcd}
                                                                      & \text{Loc} \arrow[lddd, "\mathbbm{pt}" description, bend left=12] \arrow[rddd, "\mathsf{pt}" description, bend left=12] &                                                                                                                 \\
                                                                      &                                                                                                                  &                                                                                                                 \\
                                                                      &                                                                                                                  &                                                                                                                 \\
\text{Top} \arrow[ruuu, "\mathcal{O}" description, dashed, bend left=12] &                                                                                                                  & \text{Pos}_{\omega} \arrow[luuu, "\mathsf{S}" description, dashed, bend left=12] \arrow[ll, "\mathsf{ST}", dashed]
\end{tikzcd}
 \qquad
\begin{tikzcd}
                                                                      & \text{Topoi} \arrow[lddd, "\mathbbm{pt}" description, bend left=12] \arrow[rddd, "\mathsf{pt}" description, bend left=12] &                                                                                                                 \\
                                                                      &                                                                                                                  &                                                                                                                 \\
                                                                      &                                                                                                                  &                                                                                                                 \\
\text{BIon} \arrow[ruuu, "\mathbb{O}" description, dashed, bend left=12] &                                                                                                                  & \text{Acc}_{\omega} \arrow[luuu, "\mathsf{S}" description, dashed, bend left=12] \arrow[ll, "\mathsf{ST}", dashed]
\end{tikzcd}
\end{center}

The categorical counterparts of these characters are topoi, accessible categories with directed colimits and bounded ionads, which were recently introduced by Garner. The diagram above describes the analogy between the two approaches. We call \textbf{categorified Isbell duality} the adjunction $\mathbb{O} \dashv \mathbbm{pt}$, while the adjunction $\mathsf{S}\dashv \mathsf{pt}$ is named after \textbf{Scott}. Both appear for the first time in the thesis and their study allows us to infer completeness-like results on the side of geometric logic. On the side of semantics, we try to contribute to the study of abstract elementary classes, offering a possible strategy to attack Shelah's categoricity conjecture.

\section*{Main results and structure of the thesis}
Each of the following subsections describes one of the chapters of the thesis. We briefly describe its content, recalling our main contributions.

\subsection*{Background}
The first chapter is devoted to introducing the main definitions and providing the proper references for the most relevant gadgets that will be used in the rest of the thesis. Here is the list of the treated topics:

\begin{enumerate}
	\item (general) category theory;
	\item accessible and locally presentable categories;
	\item sketches;
	\item topoi;
	\item ionads.
	\end{enumerate}

We stress that this chapter cannot give a complete presentation of the notions that it introduces and cannot be considered exhaustive even for the purpose of reading this thesis. Instead it should be seen as an \textit{auxilium} to navigate the existing literature. The chapter is mostly expository and contains new results and definition only in the section about ionads.

\subsection*{Promenade}
This chapter presents the Scott adjunction and prepares the reader for the other chapters, showing some features of the Scott adjunction.

\begin{tcolorbox}
\begin{thm*}[Thm. \ref{scottadj}, The Scott adjunction]
There is  a $2$-adjunction, $$\mathsf{S} : \text{Acc}_{\omega} \leftrightarrows \text{Topoi}: \mathsf{pt}. $$
\end{thm*}
\end{tcolorbox}

The first connections with logic and geometry are hinted. The chapter is intended to be a first encounter with the main character of the thesis.
\subsection*{Geometry}
This chapter discusses the Scott adjunction from a geometric perspective. We start with a section in general topology dealing with posets, topological spaces and locales. Then we proceed to categorify the topological constructions. We introduce the categorified Isbell adjunction and show that it is idempotent. 

\begin{tcolorbox}
\begin{thm*}[Thm. \ref{categorified isbell adj thm}, Categorified Isbell adjunction]
There is  a $2$-adjunction, $$ \mathbb{O}: \text{BIon} \leftrightarrows \text{Topoi}: \mathbbm{pt}. $$
\end{thm*}
\end{tcolorbox}

The left adjoint of this adjunction was found by Garner; in order to find a right adjoint, we must allow for large ionads. Building on the \textbf{idempotency} of the categorified Isbell adjunction \textbf{(Thm. \ref{idempotencyisbellcategorified})}, we describe those topoi for which the counit of the Scott adjunction is an equivalence of categories (\textbf{Thm. \ref{scottidempotency}}).

\subsection*{Logic}
We discuss the connection between \textbf{classifying topoi}, Scott topoi and Isbell topoi (\textbf{Thm. \ref{classificatore}} and \textbf{Thm. \ref{isbellclassificatore}}). We specialize the Scott adjunction to abstract elementary classes and locally decidable topoi.

\begin{tcolorbox}
\begin{thm*}[Thm. \ref{LDCAECs}] The Scott adjunction restricts to locally decidable topoi and AECs.
\[\mathsf{S}: \text{AECs} \leftrightarrows \text{LDTopoi}: \mathsf{pt}\]
\end{thm*}
\end{tcolorbox}
We introduce \textit{categories of saturated objects} and relate them to atomic topoi and categoricity.

\begin{tcolorbox}
\begin{thm*}[Thm. \ref{thmcategoriesofsaturated objects}] 
\begin{enumerate}
	\item[]
	\item Let $\ca$ be a category of saturated objects, then $\mathsf{S}(\ca)$ is an atomic topos.
	\item If in addition $\ca$ has  the joint embedding property, then $\mathsf{S}(\ca)$ is boolean and two valued.
	\item If in addition $\eta_\ca$ is iso-full, faithful and surjective on objects, then $\ca$ is categorical in some presentability rank.
\end{enumerate}
\end{thm*} 
\end{tcolorbox}

\subsection*{Category Theory}
We show that the $2$-category of accessible categories with directed colimits is monoidal closed and that the $2$-category of \textbf{topoi is enriched over accessible categories with directed colimits}. We show that it admits tensors and connect this fact with the Scott adjunction.

\subsection*{Toolbox}
This chapter provides technical results that are used in the previous ones. We introduce the notion of \textbf{topological embedding}, which plays an important rôle in the thesis.

\subsection*{Final remarks and open questions}
This chapter summarizes and organizes our final thoughts, offering an overview of the possible continuations of this project. Among the other things, we discuss a possible approach to \textbf{Shelah's categoricity conjecture}.

\section*{Intended readers}
The process of writing this thesis has raised a quite big issue. On the one hand we wanted to provide a mostly original text, which goes directly to the point and did not recall too many well established results. This style would have been perfect for the category theorist that is aware of the scientific literature in categorical logic, who is of course a natural candidate reader for this thesis. On the other hand one of the main purposes of this project was to provide a categorical framework to accommodate the classical theory of abstract elementary classes, and thus to use a language which can be accessible to a more classical logician (model theorist). Unfortunately, due to the scientific inclination of the author, the thesis has taken a more and more categorical shift and too many notions should have been introduced to make the thesis completely self contained. It was eventually evident that there is no satisfactory solution to this problem, and we hope that eventually a \textit{book} on categorical logic that gathers all the relevant topics will come.
Our practical solution to the issue of providing a gentle introduction to the main contents of the thesis was to start with a chapter of background that provides the most relevant definitions and guides the reader through their meaning and common uses. That chapter is thus mostly written for a reader that is a bit far from category theory and introduces many topics redirecting to the most natural references. The intention of the chapter is to introduce notions, concepts and ideas, not to explore them. The rest of the thesis is then written mostly having in mind someone that has a good understanding of the ideas exposed in the background chapter. We hope to have found a good equilibrium between these two souls of the thesis.

\section*{The rôle of the Toolbox}
One of the most important chapters of this thesis is without any doubt the Toolbox. It is designed with two main targets in mind. The first one is somehow to be the carpet under which we sweep the dust of the thesis, and thus contains all the technical results that underlie conceptual ideas. The second one is to leave room in all the other chapters for concepts and relevant proofs. The main intention behind this targets is that of providing a clean and easy to read (and browse) text. This chapter is cited everywhere in the chapters that precede it as a black-box.

\mainmatter

\chapter{Background}\label{background}

This chapter is devoted to introducing the main definitions and providing the proper references for the most relevant gadgets that will be used in the rest of the thesis. Hence, our intentions are the following:
\begin{itemize}
	\item provide the reader with an easy-to-check list of definitions of the objects that we use;
	\item provide the reader that has never met these definitions with enough material and references to familiarize themselves with the content of the thesis;
	\item meet the different necessities of the \textit{intended readers}\footnote{See the Introduction.};
	\item fix the notation.
\end{itemize}
For the reasons above, every section contained in the chapter starts with a crude list of definitions.  Then, a list of remarks and subsections will follow and is aimed to contextualize the given definitions and provide references. We stress that this chapter cannot give a complete presentation of the notions that it introduces and cannot be considered exhaustive even for the purpose of reading this thesis. Instead it should be seen as an \textit{auxilium} to navigate the existing literature. The end of the chapter fixes the notation.

\begin{ach}
If an important definition depends on some others, we proceed top-down instead of bottom-up. This means that the important definition is given first, then we proceed to define the atoms that participate in it.
\end{ach}

\section{Category theory} \label{backgroundct}

This is a thesis in category theory. We have decide not to introduce the most basic definitions of category, functor, adjoint functor and so on. Nor do we introduce those auxiliary notions that would distract the reader from the main content of the thesis. We assume complete fluency in the basic notions of category theory and some confidence with more advanced notions. Let us list a collection of references that might be useful during the reading. Several books contain a good exposition of the material that we need; our list does not have the ambition of picking the best reference for each topic.

\begin{enumerate}
 \item Basic category theory, \cite{leinster_2014}[Chap. 1-4];
\item Monads, \cite{BOR2}[Chap. 4];
\item Kan extensions, \cite{borceux_1994}[Chap. 3.7] and \cite{liberti2019codensity}[App. A];
\item (Symmetric) Monoidal (closed) categories, \cite{BOR2}[Chap. 6];
\item $2$-categories and bicategories, \cite{borceux_1994}[Chap. 7].
\end{enumerate}

\section{Accessible and locally presentable categories} \label{backgroundLPAC}
\begin{ach} In this section $\lambda$ is a regular cardinal.
\end{ach}

\begin{defn}[$\lambda$-accessible category]
A $\lambda$-accessible category $\ca$ is a category with $\lambda$-directed colimits with a set of $\lambda$-presentable objects that generate by $\lambda$-directed colimits. An accessible category is a category that is $\lambda$-accessible for some $\lambda$.
\end{defn}

\begin{defn}[Locally $\lambda$-presentable category]
A locally $\lambda$-presentable category is a cocomplete $\lambda$-accessible category. A locally presentable category is a category that is locally $\lambda$-presentable for some $\lambda$.
\end{defn}

\begin{defn}[$\lambda$-presentable object]
An object $a \in \ca$ is $\lambda$-presentable if its covariant hom-functor $\ca(a, -): \ca \to \Set$ preserves $\lambda$-directed colimits.
\end{defn}

\begin{defn}[$\lambda$-directed posets and $\lambda$-directed colimits]
A poset $P$ is $\lambda$-directed if it is non empty and for every $\lambda$-small\footnote{This means that its cardinality is strictly less then $\lambda$. For example $\aleph_0$-small means finite.} family of elements $\{p_i\} \subset P$, there exists an upper bound. A $\lambda$-directed colimit is the colimit of a diagram over a $\lambda$-directed poset (seen as a category).
\end{defn}

\begin{notation} For a category $\ca$, we will call $\ca_\lambda$ its full subcategory of $\lambda$-presentable objects.
\end{notation}

\subsection{Literature}

There are two main references for the theory of accessible and locally presentable categories, namely \cite{adamekrosicky94} and \cite{Makkaipare}. The first one is intended for a broader audience and appeared few years after the second one. The second one is mainly concerned with the logical aspects of this theory. We mainly recommend \cite{adamekrosicky94} because it appears a bit more fresh in style and definitely less demanding in general knowledge of category theory. A more experienced reader (in category theory) that is mainly interested in logic could choose \cite{Makkaipare}. Even though \cite{adamekrosicky94}  treats some $2$-categorical aspects of this topic, \cite{Makkaipare}'s exposition is much more complete in this direction.   Another good general exposition is  \cite{BOR2}[Chap. 5].

\subsection{A short comment on these definitions}

\begin{rem}
The theory of accessible and locally presentable categories has gained quite some popularity along the years because of its natural ubiquity. Most of the categories of the \textit{working mathematician} are accessible, with a few (but still extremely important) exceptions. For example, the category $\mathsf{Top}$ of topological spaces is not accessible. In general, categories of algebraic structures are locally  $\aleph_0$-presentable and many relevant categories of geometric nature are $\aleph_1$-accessible. A sound rule of thumb is that locally finitely presentable categories correspond to categories of models essentially algebraic theories, in fact this is even a theorem in a proper sense \cite{adamekrosicky94}[Chap. 3]. A similar intuition is available for accessible categories too, but some technical price must be paid \cite{adamekrosicky94}[Chap. 5]. Accessible and locally presentable categories (especially the latter) are \textit{tame} enough to make many categorical wishes come true; that's the case for example of the adjoint functor theorem, that has a very easy to check version for locally presentable categories.
\end{rem}

\begin{rem} All in all, an accessible category should be seen as a category equipped with a small set of \textit{small} objects such that every object can be obtained as a kind of directed union of them. In the category of topological spaces, these small objects are not enough to recover any other object from them.
\end{rem}

\begin{exa}
To clarify the previous remark, we give list of locally $\aleph_0$-presentable categories. On the right column we indicate the full subcategory of finitely presentable objects.

\begin{table}[!htbp]
\begin{tabular}{lllll}
\cline{2-3}
\multicolumn{1}{l|}{} & \multicolumn{1}{c|}{$\ck$} & \multicolumn{1}{c|}{$\ck_\omega$} &  &  \\ \cline{2-3}
\multicolumn{1}{l|}{} & \multicolumn{1}{c|}{$\Set$} & \multicolumn{1}{c|}{finite sets} &  &  \\ \cline{2-3}
\multicolumn{1}{l|}{} & \multicolumn{1}{c|}{$\mathsf{Grp}$} & \multicolumn{1}{c|}{finitely presentable groups} &  &  \\ \cline{2-3}
\multicolumn{1}{l|}{} & \multicolumn{1}{c|}{$\mathsf{Mod(R)}$} & \multicolumn{1}{c|}{finitely presentable modules} &  &  \\ \cline{2-3}
\end{tabular}
\end{table}

It is not surprising at all that a set $X$ is the directed union of its \textit{finite} subsets.
\end{exa}

\begin{rem} 
Accessible and locally presentable categories have a \textit{canonical representation}, in terms of \textit{free completions under $\lambda$-directed of colimits}. This theory is studied in \cite{adamekrosicky94}[Chap. 2.C]. The free completion of a category $C$ under $\lambda$-directed colimits is always indicated by $ \mathsf{Ind}_\lambda(C) $ in this thesis. 
\begin{thm}
A $\lambda$-accessible category $\ca$ is equivalent to   the free completion of $\ca_\lambda$ under $\lambda$-directed colimits, $$ \ca \simeq \mathsf{Ind}_\lambda(\ca_\lambda). $$
\end{thm}
\end{rem}

\begin{rem}
Explicit descriptions of the free completion of a category under $\lambda$-directed colimits are indeed available. To be more precise, given a category $C$ one can describe $\mathsf{Ind}(C)$ as the category of flat functors $\text{Flat}(C^\circ, \Set)$. In the special case of a category with finite colimits, we have a simpler description of flat functors. Let us state the theorem in this simpler case for the sake of simplicity.

\begin{thm}
Let $C$ be a small category with finite colimits. Its free completion under directed colimits is given by the category of functors preserving finite limits from $C^\circ$ into sets:

 $$\mathsf{Ind}(C) \simeq \mathsf{Lex}(C^\circ, \Set).$$ 
\end{thm}

\end{rem}

\subsection{Locally presentable categories and essentially algebraic theories}
The connection between locally presentable categories and essentially algebraic theories is made precise in \cite{adamekrosicky94}[Chap. 3]. While algebraic theories axiomatize operational theories, essentially algebraic theories axiomatize operational theories whose operations are only partially defined. Category theorists have an equivalent approach to essentially algebraic theories via categories with finite limits. This approach was initially due to Freyd \cite{FreydCartesianLogic}, though a seminal work of Coste \cite{coste1976approche} should be mentioned too.

\subsection{Accessible categories and (infinitary) logic}
Accessible categories have been connected to (infinitary) logic in several (partially independent) ways. This story is recounted in Chapter 5 of \cite{adamekrosicky94}. Let us recall two of the most important results of that chapter.
\begin{enumerate}
	\item As locally presentable categories, accessible categories are categories of models of theories, namely \textit{basic} theories \cite{adamekrosicky94}[Def. 5.31, Thm. 5.35].
	\item Given a theory $T$ in $L_\lambda$ the category $\mathsf{Elem}_\lambda(T)$ of models and $\lambda$-elementary embeddings is accessible \cite{adamekrosicky94}[Thm. 5.42].
\end{enumerate} 
Unfortunately, it is not true in general that the whole category of models and homomorphisms of a theory in $L_\lambda$ is accessible. It was later shown by Lieberman \cite{Lthesis} and independently by Rosický and Beke \cite{aec}  that abstract elementary classes are accessible too. We will say more about these recent developments later in the thesis. The reader that is interested in this connection might find interesting \cite{vasey2019accessible}, whose language is probably the closest to that of a model theorist.

\section{Sketches} \label{backgroundsketches}

\begin{defn}[Sketch]
A sketch is a quadruple $\mathcal{S}=(S, L,C, \sigma)$ where
\begin{enumerate}
	\item[$S$] is a small category;
	\item[$L$] is a class of diagrams in $S$, called \textit{limit} diagrams;
	\item[$C$] is a class of diagrams in $S$, called \textit{colimit} diagrams;
	\item[$\sigma$] is a function assigning to each diagram in $L$ a cone and to each diagram in $C$ a cocone.
\end{enumerate}
\end{defn}

\begin{defn}
A sketch is
\begin{itemize}

	\item \textit{limit} if $C$ is empty;
	\item \textit{colimit} if $L$ is empty;
	\item \textit{mixed} (used only in  emphatic sense) if it's not limit, nor colimit;
	\item \textit{geometric}  if each cone is finite;
	\item \textit{coherent} if it is geometric and and every cocone is either finite or discrete, or it is a regular-epi specification\footnote{See \cite{elephant2}[D2.1.2].}.
	
\end{itemize}
\end{defn}

\begin{defn}[Morphism of Sketches]
Given two sketches $\cs$ and $\ct$, a morphism of sketches $f: \cs \to \ct$ is a functor $f: S \to T$ mapping (co)limit diagrams into (co)limits diagrams and proper (co)cones into (co)cones.
\end{defn}

\begin{defn}[$2$-category of Sketches]
The $2$-category of sketches has sketches as objects, morphism of sketches as $1$-cells and natural transformations as $2$-cells.
\end{defn}

\begin{defn}[Category of models of a sketch]
For a sketch $\cs$ and a (bicomplete) category $\cc$, the category $\mathsf{Mod}_{\cc}(\cs)$ of $\cc$-models of the sketch is the full subcategory of $\cc^\cs$ of those functors that are models. If it's not specified, by $\mathsf{Mod}(\cs)$ we mean $\mathsf{Mod}_{\Set}(\cs)$.
\end{defn}

\begin{defn}[Model of a sketch]
A model of a sketch $\cs$ in a category $\cc$ is a functor $f: \cs \to \cc$ mapping each specified (co)cone to a (co)limit (co)cone.  If it's not specified a model is a $\Set$-model.
\end{defn}

\subsection{Literature}
There exists a plethora of different and yet completely equivalent approaches to the theory of sketches. We stick to the one that suits best our setting, following mainly \cite{BOR2}[Chap. 5.6] or \cite{adamekrosicky94}[Chap. 2.F]. Other authors, such as \cite{Makkaipare} and \cite{elephant2} use a different (and more classical) definition involving graphs. Sketches are normally used as generalized notion of theory. From this perspective these approaches are completely equivalent, because the underlying categories of models are the same. \cite{Makkaipare}[page 40] stresses that the graph-definition is a bit more flexible in \textit{daily practice}. Sketches were introduced by C. Ehresmann. Guitart, Lair and Burroni should definitely be mentioned among the relevant contributors. This list of references does not do justice to the French school, which has been extremely prolific on this topic, yet, for the purpose of this thesis the literature above will be more then sufficient.

\subsection{Sketches: logic and sketchable categories}
Sketches became quite common among category theorists because of their expressiveness. In fact, they can be used as a categorical analog of those theories that can be axiomatized by (co)limit properties. For example, in the previous section, essentially algebraic theories are precisely those axiomatizable by finite limits.

\subsubsection{From theories to sketches}
We have mentioned that a sketch can be seen as a kind of theory. This is much more than a motto, or a motivational presentation of sketches. In fact, given a (infinitary) first order theory $\mathbb{T}$, one can always construct in a more or less canonical way a sketch $\mathcal S_{\mathbb{T}}$ whose models are precisely the models of $\mathbb{T}$. This is very well explained in \cite{elephant2}[D2.2]; for the sake of exemplification, let us state the theorem which is most relevant to our context.

\begin{thm}
If $\mathbb{T}$ is a (geometric) (coherent) theory, there there exists a (geometric) (coherent) sketch having the same category of models of $\mathbb{T}$.
\end{thm}

Some readers might be unfamiliar with geometric and coherent theories; these are just very specific fragments of first order (infinitary) logic. For a very detailed and clean treatment we suggest \cite{elephant2}[D1.1]. Sketches are quite a handy notion of theory because we can use morphisms of sketches as a notion of translation between theories. 

\begin{prop}[{\cite{BOR2}[Ex. 5.7.14]}]
If $f: \mathcal{S} \to\mathcal{T} $ is a morphism of sketches, then composition with $f$ yields an (accessible) functor $\mathsf{Mod}(\cs) \to \mathsf{Mod}(\ct)$.
\end{prop}

\subsubsection{Sketchability}
It should not be surprising that sketches can be used to \textit{axiomatize} accessible and locally presentable categories too. The two following results appear, for example, in \cite{adamekrosicky94}[2.F].

\begin{thm}
A category is locally presentable if and only if it's equivalent to the category of models of a limit sketch.
\end{thm}

\begin{thm}
A category is accessible if and only if it's equivalent to the category of models of a mixed sketch.
\end{thm}

\newpage
\section{Topoi} \label{backgroundtopoi}

\begin{ach}
In this section by \textit{topos} we mean Grothendieck topos.
\end{ach}

\begin{defn}[Topos]
A topos $\ce$ is  lex-reflective\footnote{This means that it is a reflective subcategory and that the left adjoint preserves finite limits. Lex stands for \textit{left exact}, and was originally motivated by homological algebra.} subcategory\footnote{Up to equivalence of categories.} of a category of presheaves over a small category, $$i^*: \Set^{C^\circ} \leftrightarrows \ce : i_*.  $$
\end{defn}

\begin{defn}[Geometric morphism]
A geometric morphism of topoi $f: \ce \to \cf$ is an adjunction $f^*: \cf \leftrightarrows \ce: f_*$\footnote{Notice that $f_*$ is the right adjoint.} whose left adjoint preserves finite limits (is left exact). We will make extensive use of the following terminology:
\begin{enumerate}
	\item[$f^*$] is the inverse image functor;
	\item[$f_*$] is the direct image functor.
	\end{enumerate}
\end{defn}

\begin{defn}[$2$-category of Topoi]
The $2$-category of topoi has topoi as objects, geometric morphisms as $1$-cells and natural transformations between left adjoints as $2$-cells.
\end{defn}

\subsection{Literature}
There are several standard references for the theory of topoi. To the absolute beginner and even the experienced category theorist that does not have much confidence with the topic, we recommend \cite{leinster2010informal}. Most of the technical content of the thesis can be understood via \cite{sheavesingeometry}, a reference that we strongly suggest to start and learn topos theory. Unfortunately, the approach of \cite{sheavesingeometry} is a bit different from ours, and even though its content is sufficient for this thesis, the intuition that is provided is not $2$-categorical enough for our purposes. The reader might have to integrate with the encyclopedic \cite{elephant1,elephant2}.  A couple of constructions that are quite relevant to us are contained only in \cite{borceux_19943}, that is otherwise very much equivalent to \cite{sheavesingeometry}.

\subsection{A comment on these definitions}
Topoi were defined by Grothendieck as a \textit{natural} generalization of the category of sheaves $\mathsf{Sh}(X)$ over a topological space $X$. Their geometric nature was thus the first to be explored and exploited. Yet, with time, many other properties and facets of them have emerged, making them one of the main concepts in category theory between the 80's and 90's. Johnstone, in the preface of \cite{elephant1} gives 9 different interpretations of what a topos \textit{can be}. In fact, this multi-faced nature of the concept of topos motivates the title of his book. In this thesis we will concentrate on three main aspects of topos theory.

\begin{itemize}
	\item A topos is a (categorification of the concept of) locale;
	\item A topos is a (family of Morita-equivalent) geometric theory;
	\item A topos is an object in the $2$-category of topoi.
\end{itemize}

The first and the second aspects will be conceptual, and will allow us to infer qualitative results in geometry and logic, the last one will be our \textit{methodological point of view} on topoi, and ultimately the main reason for which \cite{sheavesingeometry} might not be a sufficient reference for this thesis.

\subsection{Site descriptions of topoi}

The first definition of topos that has been given was quite different from the one that we have introduced. As we have mentioned, topoi were introduced as category of sheaves over a space, thus the first definition was based on a generalization of this presentation. This is the theory of sites, and the reader of \cite{sheavesingeometry} will recognize this approach in \cite{sheavesingeometry}[Chap. 3]. In a nutshell, a site $(C,J)$ is the data of a category $C$ together a notion of covering families. For example, in the case of a topological space, $C$ is the locale of open sets of $X$, and $J$ is given by the open covers. Thus, a topos can be defined to be a category of sheaves over a small site, $$ \ce \simeq \mathsf{Sh}(C,J). $$
$\mathsf{Sh}(C,J)$ is defined as a full subcategory of $\Set^{C^{\circ}}$, which turn out to be lex-reflective. That's the technical bridge between the site-theoretic description of a topos and the one at the beginning of the section.
Site theory is extremely useful in order to study topoi as \textit{categories}, while our approach is much more useful in order to study them as \textit{objects}. We will never use explicitly site theory in the thesis, with the exception of a couple of proofs and a couple of examples.

\subsection{Topoi and Geometry}
It's a bit hard to convey the relationship between topos theory and geometry in a short subsection. We mainly address the reader to \cite{leinster2010informal}. Let us just mention that to every topological space $X$, one can associate its category of sheaves $\mathsf{Sh}(X)$ (and this category is a topos), moreover, this assignment is a very strong topological invariant. For this reason, the study of $\mathsf{Sh}(X)$ is equivalent to the study of $X$ from the perspective of the topologist, and is very convenient in algebraic geometry and algebraic topology. For example, the category of sets is the topos of sheaves over the one-point-space, $$ \Set \cong \mathsf{Sh}(\bullet) $$ for this reason, category-theorists sometime call $\Set$ \textit{the point}. This intuition is consistent with the fact that $\Set$ is the terminal object in the category of topoi. Moreover, as a point $p \in X$ of a topological space  $X$ is a continuous function $p: \bullet \to X$, a point of a topos $\cg$ is a geometric morphism $p: \Set \to \cg$. Parallelisms of this kind have motivated most of the definitions of topos theory and most have led to results very similar to those that were achieved in formal topology (namely the theory of locales). The class of points of a topos $\ce$ has a structure of category $\mathsf{pt}(\ce)$ in a natural way, the arrows being natural transformations between the inverse images of the geometric morphisms.

\subsection{Topoi and Logic}
Geometric logic and topos theory are tightly bound together. Indeed, for a geometric theory $\mathbb{T}$ it is possible to build a topos $\Set[\mathbb{T}]$ (the classifying topos of $\mathbb{T}$) whose category of points is precisely the category of models of $\mathbb{T}$,
$$ \mathsf{Mod}(\mathbb{T}) \cong \mathsf{pt}(\Set[\mathbb T]). $$
This amounts to the theory of classifying topoi \cite{sheavesingeometry}[Chap. X] and each topos classifies a geometric theory. This gives us a logical interpretation of a topos. Each topos is a \textit{geometric theory}, which in fact can be recovered from any of its sites of definition. Obviously, for each site that describes the same topos we obtain a different theory. Yet, these theories have the same category of models (in any topos). In this thesis we will exploit the construction of \cite{borceux_19943} to show that to each \textit{geometric} sketch (a kind of theory), one can associate a topos whose points are precisely the models of the sketch. This is another way to say that the category of topoi can internalize a geometric logic.

\subsubsection{A couple of words on elementary topoi}
Grothendieck topoi are often treated in parallel with their cousins elementary topoi. The origins of the theory of elementary topoi dates back to Lawvere's \textit{elementary theory of the category of sets} \cite{L1964}. The general idea behind his program was find those axioms that make a category a good place to found mathematics. Eventually, elementary topoi bloomed from the collaboration of Lawvere and Tirney. In a nutshell, an elementary topos is a cartesian closed category with finite limits, and a subobject classifier. It turns out that every Grothendieck topos is an elementary topos, while every cocomplete elementary topos with a generator is a Grothendieck topos. These two results establish a tight connection between the two concepts. The intuition that we have on them is yet quite different. An elementary topos is \textit{a universe of sets}, while a Grothendieck topos is \textit{a geometric theory} or a \textit{categorified locale}. The theory of elementary topoi is extremely rich, but in this thesis the elementary-topos perspective will never play any rôle; we are mentioning elementary topoi precisely to stress that this notion will not be used in the thesis. 

\subsection{Special classes of topoi}
In the thesis we will study some relevant classes of topoi. In this subsection we recall all of them and give a good reference to check further details. These references will be repeated in the relevant chapters.

\begin{table}[!htbp]
\begin{tabular}{lllll}
\cline{2-3}
\multicolumn{1}{l|}{} & \multicolumn{1}{c|}{Topoi} & \multicolumn{1}{c|}{Reference} &  &  \\ \cline{2-3}
\multicolumn{1}{l|}{} & \multicolumn{1}{c|}{connected} & \multicolumn{1}{c|}{{\cite{elephant2}[C1.5.7]}} &  &  \\ \cline{2-3}
\multicolumn{1}{l|}{} & \multicolumn{1}{c|}{compact} & \multicolumn{1}{c|}{{\cite{elephant2}[C3.2]}} &  &  \\ \cline{2-3}
\multicolumn{1}{l|}{} & \multicolumn{1}{c|}{atomic} & \multicolumn{1}{c|}{{\cite{elephant2}[C3.5]}} &  &  \\ \cline{2-3}
\multicolumn{1}{l|}{} & \multicolumn{1}{c|}{locally decidable} & \multicolumn{1}{c|}{{\cite{elephant2}[C5.4]}} &  &  \\ \cline{2-3}
\multicolumn{1}{l|}{} & \multicolumn{1}{c|}{coherent} & \multicolumn{1}{c|}{{\cite{elephant2}[D3.3]}} &  &  \\ \cline{2-3}
\multicolumn{1}{l|}{} & \multicolumn{1}{c|}{boolean} & \multicolumn{1}{c|}{{\cite{elephant2}[D3.4, D4.5], \cite{elephant1}[A4.5.22]}} &  &  \\ \cline{2-3}

\end{tabular}
\end{table}

\newpage

\section{Ionads} \label{backgroundionads}

\subsection{Garner's definitions}
\begin{defn}[Ionad]
An ionad $\cx = (X, \text{Int})$ is a set $X$ together with a comonad $\text{Int}: \Set^X \to \Set^X$ preserving finite limits.
\end{defn}

\begin{defn}[Category of opens of a ionad]
The category of opens $\mathbb{O}(\cx)$ of a ionad $\cx = (X, \text{Int})$ is the category of coalgebras of $\text{Int}$. We shall denote by $U_{\cx}$ the forgetful functor $U_\cx: \mathbb{O}(\cx) \to \Set^X$.  
\end{defn}

\begin{defn}[Morphism of Ionads]
A morphism of ionads $f: \cx \to \cy$ is a couple $(f, f^\sharp)$ where $f: X \to Y$ is a set function and $f^\sharp$ is a lift of $f^*$,
\begin{center}
\begin{tikzcd}
\mathbb{O}(\cy) \arrow[rr, "f^\sharp" description] \arrow[dd, "U_\cy" description] &  & \mathbb{O}(\cx) \arrow[dd, "U_\cx" description] \\
                                                                               &  &                                                  \\
\Set^Y \arrow[rr, "f^*" description]                                           &  & \Set^X                                          
\end{tikzcd}
\end{center}
\end{defn}

\begin{defn}[Specialization of morphism of ionads]
Given two morphism of ionads $f,g: \cx \to \cy$, a specialization of morphism of ionads $\alpha: f \Rightarrow g$ is a natural transformation between $f^\sharp$ and $g^\sharp$,
\begin{center}
\begin{tikzcd}
\mathbb{O}(\cy) \arrow[r, "f^\sharp" description, bend left=35, ""{name=U, below}]
\arrow[r,"g^\sharp" description, bend right=35, ""{name=D}]
& \mathbb{O}(\cx)
\arrow[Rightarrow, "\alpha" description, from=U, to=D] \end{tikzcd}
\end{center}
\end{defn}

\begin{defn}[$2$-category of Ionads]
The $2$-category of ionads has ionads as objects, morphism of ionads as $1$-cells and specializations as $2$-cells.
\end{defn}

\begin{defn}[Bounded Ionads]
A ionad $\cx$ is bounded if $\mathbb{O}(\cx)$ is a topos.
\end{defn}

\subsection{Ionads and topological spaces}
Ionads were defined by Garner in \cite{ionads}, and to our knowledge that's all the literature available on the topic. His definition is designed to generalize the definition of topological space. Indeed a topological space $\cx$ is the data of a set (of points) and an interior operator, $$\text{Int}: 2^X \to 2^X.$$ Garner builds on the well known analogy between powerset and presheaf categories and extends the notion of interior operator to a presheaf category. The whole theory is extremely consistent with the expectations: while the poset of (co)algebras for the interior operator is the locale of open sets of a topological space, the category of coalgebras of a ionad is a topos, a natural categorification of the concept of locale. 

\subsection{A generalization and two related propositions}
In his paper Garner mentions that in giving the definition of ionad he could have chosen a category instead of a set \cite{ionads}[Rem. 2.4], let us quote his own comment on the definition.

\begin{quote}[\cite{ionads}, Rem. 2.4]
In the definition of ionad, we have chosen to have a mere \emph{set} of points, rather
than a category of them. We do so for a number of reasons. The first is that this choice 
mirrors most closely the definition of topological space, where we have a set, and not a 
poset, of points. The second is that we would in fact obtain no extra generality by allowing 
a category of points. We may see this analogy with the topological case, where to give an 
interior operator on a poset of points $(X,\leq)$ is equally well to give a topology 
$\mathcal{O}(X)$ on $X$ such that every open set is upwards-closed with respect to $\leq$.
Similarly, to equip a small category $C$ with an interior comonad is equally well 
to give an interior comonad on $X:=\mathrm{ob}C$ together with a factorization 
of the forgetful functor $\mathbb{O}(X)\to\mathsf{Set}^{X}$ through the presheaf category 
$\mathsf{Set}^C$; this is an easy consequence of Example 2.7 below. However, the 
most compelling reason for not admitting a category of points is that, if we were to do so, 
then adjunctions such as that between the category of ionads and the category of topological 
spaces would no longer exist. Note that, although we do not allow a category of points, the 
points of any (well-behaved) ionad bear nonetheless a canonical category structure -- 
described in Definition 5.7 and Remark 5.9 below -- which may be understood as a generalization 
of the specialization ordering on the points of a space.  
\end{quote}

We have decided to allow ionads over a category, even a locally small (but possibly large) one. We will need this definition later in the text to establish a connection between ionads and topoi. While the structure of category is somewhat accessory, as Garner observes, the one of proper class will be absolutely needed.

\begin{defn}[Generalized Ionads]
A generalized ionad $\cx = (X, \text{Int})$ is a locally small (but possibly large) pre-finitely cocomplete category $X$ together with a lex comonad $\text{Int}: \P(X) \to \P(X)$.
\end{defn}

\begin{ach}
We will always omit the adjective \textit{generalized}.
\end{ach}

\begin{rem} We are well aware that the notion of generalized ionad seems a bit puzzling at first sight. \textit{Why isn't it just the data of a locally small category $X$ together with a lex comonad on $\Set^X$?} The answer to this question is a bit delicate, having both a technical and a conceptual aspect. Let us first make precise the notion above, introducing all the concepts that are mentioned, then we will discuss in what sense this is the \textit{correct} notion of generalized ionad.
\end{rem}

\begin{rem}
In a nutshell, $\P(X)$ is a well-behaved full subcategory of $\Set^X$, while the existence of finite pre-colimits will ensure us that $\P(X)$ has finite limits. Let us dedicate some remarks to make these hints more precise.
\end{rem}

\begin{rem}[On small (co)presheaves]
By $\P(X)$ we mean the full subcategory of $\Set^X$  made by small copresheaves over $X$, namely those functors $X \to \Set$ that are small colimits of corepresentables (in $\Set^X$). This is a locally small category, as opposed to $\Set^X$ which might be locally large. The study of small presheaves $X^\circ \to \Set$ over a category $X$ is quite important with respect to the topic of free completions under limits and under colimits. Obviously, when $X$ is small, every presheaf is small. Given a category $X$, its category of small presheaves is usually indicated by $\cp(X)$, while $\cp^\sharp(X)$ is $\cp(X^\circ)^\circ$. The most updated account on the property of $\cp(X)$ is given by \cite{presheaves} and \cite{DAY2007651}. $\cp(X)$ is the free completion of $X$ under colimits, while $\cp^\sharp(X)$ is the free completion of $X$ under limits. The following equation clarifies the relationship between $\cp, \P$ and $\cp^\sharp$, \[\cp^\sharp(X)^\circ = \P(X) = \cp(X^\circ).\]
This means that $\P(X)$ is the free completion of $X^\circ$ under colimits.
\end{rem}

\begin{rem}
The category of small presheaves $\cp(X)$ over a (locally small) large category $X$ is a bit pathological, especially if we keep the intuition that we have when $X$ is small. In full generality $\cp(X)$ is not complete, nor it has any limit whatsoever. Yet, under some smallness condition most of the relevant properties of $\cp(X)$ remain true. Below we recall a good example of this behavior, and we address the reader to \cite{presheaves} for a for complete account.
\end{rem}

\begin{prop}[{\cite{presheaves}[Cor. 3.8]}] \label{smallpreshcomplete}
$\cp(X)$ is (finitely) complete if and only if $X$ is (finitely) pre-complete\footnote{See {\cite{presheaves}[Def. 3.3]}.}.
\end{prop}

\begin{cor}
If $X$ is finitely pre-cocomplete, then $\P(X)$ has finite limits.
\end{cor}

A precise understanding of the notion of pre-cocomplete category is actually not needed for our purposes, the following sufficient condition will be more than enough through the thesis.

\begin{cor}[{\cite{presheaves}[Exa. 3.5 (b) and (c)]}]
If $X$ is small or it is accessible, then $\P(X)$ is complete.
\end{cor}

What must be understood is that being pre-complete, or pre-cocomplete should not be seen as a completeness-like property, instead it is much more like a smallness assumption.

\begin{exa}[Ionads are generalized ionads]
It is obvious from the previous discussion that a ionad is a generalized ionad.
\end{exa}

\begin{rem}[Small copresheaves vs copresheaves]
When $X$ is a finitely pre-cocomplete category, $\P(X)$ is an infinitary pretopos and finite limits are \textit{nice} in the sense that they can be computed in $\Set^X$. Being an infinitary pretopos, together with being the free completion under (small) colimits makes the conceptual analogy between $\P(X)$ and $2^X$ nice and tight, but there is also a technical reason to prefer small copresheaves to copresheaves.
\end{rem}

\begin{prop}
If $f^*: \cg \to \P(X)$ is a cocontinuous functor from a total category, then it has a right adjoint $f_*$.
\end{prop}

\begin{rem}
The result above allows to produce comonads on $\P(X)$ (just compose $f^*f_*$) and follows from the general theory of total categories, but needs $\P(X)$ to be locally small to stay in place. Thus the choice of $\Set^X$ would have generated size issues. A similar issue would arise with Kan extensions.
\end{rem}

\begin{ach}
$\P(X)$ is a (Grothendieck) topos if and only if $X$ is an essentially small category, thus in most of the examples of our interest $\P(X)$ will not be a Grothendieck topos. Yet, we feel free to use a part of the terminology from topos theory (geometric morphism, geometric surjection, geometric embedding), because it is an infinitary pretopos (and thus only lacks a generator to be a topos).
\end{ach}

\begin{rem} In analogy with the notion of base for a topology, Garner defines the notion of base of a ionad \cite{ionads}[Def. 3.1, Rem. 3.2]. This notion will be a handy technical tool in the thesis. Our definition is pretty much equivalent to Garner's one (up to the fact that we keep flexibility on the size of the base) and is designed to be easier to handle in our setting.
\end{rem}

\begin{defn}[Base of a ionad]
Let $\cx = (X, \text{Int})$ be a ionad. We say that a flat functor $e: B \to \P(X)$ generates\footnote{This definition is just a bit different from Garner's original definition \cite{ionads}[Def. 3.1, Rem. 3.2]. We stress that in this definition, we allow for large basis.} the ionad if $\text{Int}$ is naturally isomorphic to the density comonad of $e$, \[\text{Int} \cong \lan_e e.\]
\end{defn}

\begin{exa}
The forgetful functor $U_\cx: \mathbb{O}(\cx) \to \P(X)$ is always a basis for the ionad $\cx$. This follows from the basic theory about density comonads: when $U_\cx$ is a left adjoint, its density comonad coincides with the comonad induced by its adjunction. This observation does not appear in \cite{ionads} because he only defined small bases, and it almost never happens that $\mathbb{O}(\cx)$ is a small category.
\end{exa}

In \cite{ionads}[3.6, 3.7], the author lists three equivalent conditions for a ionad to be bounded. The conceptual one is obviously that the category of opens is a Grothendieck topos, while the other ones are more or less technical. In our treatment the equivalence between the three conditions would be false. But we have the following characterization.

\begin{prop}\label{critboundedionad}
A ionad $\cx = (X, \text{Int})$ is bounded if any of the following equivalent conditions is verified:
\begin{enumerate}
  \item $\mathbb{O}(\cx)$ is a topos.
  \item there exist a Grothendieck topos $\cg$ and a geometric surjection $f : \P(X) \twoheadrightarrow \cg$ such that $\text{Int} \cong f^*f_*$.
	\item there exist a Grothendieck topos $\cg$, a geometric surjection $f : \P(X) \twoheadrightarrow \cg$ and a flat functor $e: B \to \cg$ such that $f^*e$ generates the ionad.
\end{enumerate}
\end{prop}
\begin{proof}
Clearly $(1)$ implies $(2)$. For the implication $(2) \Rightarrow (3)$, it's enough to choose $e: B \to \cg$ to be the inclusion of any generator of $\cg$. Let us discuss the implication $(3) \Rightarrow (1)$. Let $\ce$ be the category of coalgebras for the density comonad of $e$ and call $g: \cg \to \ce$ the geometric surjection induced by the comonad, (in particular $\lan_ee \cong g^*g_*$). We claim that $\ce \simeq \mathbb{O}(\cx)$. Invoking \cite{sheavesingeometry}[VII.4 Prop. 4] and because geometric surjections compose, we have $\ce \simeq \mathsf{coAlg}(f^*g^*g_*f_*)$. The thesis follows from the observation that \[\text{Int} \cong \lan_{f^*e} (f^*e) \cong \lan_{f^*}( \lan_e (f^*e)) \cong \lan_{f^*}(f^* \lan_e e) \cong f^*g^*g_*f_*.\] 
\end{proof}

\begin{rem} Later in the thesis, we will need a practical way to induce morphism of ionads. The following proposition does not appear in \cite{ionads} and will be our main \textit{morphism generator}. From the perspective of developing technical tool in the theory of ionads, this proposition has an interest in its own right.
\end{rem}

\begin{rem} The proposition below categorifies a basic lemma in general topology: let $f: X \to Y$ be a function between topological spaces, and let $B_X$ and $B_Y$ be bases for the respective topologies. If $f^{-1}(B_Y) \subset B_X$, then $f$ is continuous. Our original proof has been simplified by Richard Garner during the reviewing process of the thesis.
\end{rem}

\begin{prop}[Generator of morphism of ionads] \label{morphism of ionads generator}
Let $\cx$ and $\cy$ be ionads, respectively generated by bases $e_X: B \to \P(X)$ and $e_Y: C \to \P(Y)$. Let $f: X \to Y$ a functor admitting a lift as in the diagram below.
\begin{center}
\begin{tikzcd}
C \arrow[dd, "e_Y" description] \arrow[rr, "f^\diamond" description, dashed] &  & B \arrow[dd, "e_X" description] \\
                                                                             &  &                                 \\
\P(Y) \arrow[rr, "f^*" description]                                         &  & \P(X)                         
\end{tikzcd}
\end{center}
 If one of the two following conditions holds,  then $f$ induces a morphism of ionads $(f, f^\sharp)$:
\end{prop}
\begin{proof}
By the discussion in \cite{ionads}[Exa. 4.6, diagram (6)], it is enough to provide a morphism as described in the diagram below.
	\begin{center}
	\begin{tikzcd}
	C \arrow[dd, "e_Y" description] \arrow[rr, "f'" description, dashed] &  & \mathbb{O}(\lan_{e_X} e_X) \arrow[dd, "\mathsf{U}_\cx" description] \\
                                                                             &  &                                 \\
	\P(Y) \arrow[rr, "f^*" description]                                         &  & \P(X)                        
	\end{tikzcd}
	\end{center}
Also, \cite{ionads}[Exa. 4.6] shows that giving a map of ionads $\cx \to \cy$ is the same of giving $f: X \to Y$ and a lift of $C \to \P(Y) {\to} \P(X)$ through $\mathbb{O}(X)$. Applying this to the identity map $\cx \to \cx$ we get a lift of $B \to \P(X)$ trough $\mathbb{O}(\lan_{e_X} e_X)$.  Now composing that with $C \to B$ gives the desired square.
  \begin{center}
  \begin{tikzcd}
  C \arrow[dd, "e_Y" description] \arrow[rr, "f^\diamond" description] \arrow[rrrr, "f'" description, dashed, bend left=15] &  & B \arrow[rr] \arrow[dd, "e_X"] &  & \mathbb{O}(\lan_{e_X} e_X) \arrow[lldd, "\mathsf{U}_\cx" description] \\
                                                                                                                       &  &                                &  &                                                                       \\
  \P(Y) \arrow[rr, "f^*" description]                                                                                   &  & \P(X)                          &  &                                                                      
  \end{tikzcd}
  \end{center}
\end{proof}

\section{Notations and conventions} \label{backgroundnotations}

Most of the notation will be introduced when needed and we will try to make it as natural and intuitive as possible, but we would like to settle some notation.
\begin{enumerate}
\item $\ca, \cb$ will always be accessible categories, very possibly with directed colimits.
\item $\cx, \cy$ will always be ionads.
\item When it appears, $\ck$ is a finitely accessible category.
\item $\mathsf{Ind}_\lambda$ is the free completion under $\lambda$-directed colimits.
\item $\ca_{\kappa}$ is the full subcategory of $\kappa$-presentable objects of $\ca$.
\item $\cg, \ct, \cf, \ce$ will be Grothendieck topoi.
\item In general, $C$ is used to indicate small categories.
\item $\eta$ is the unit of the Scott adjunction.
\item $\epsilon$ is the counit of the Scott adjunction.
\item A Scott topos is a topos of the form $\mathsf{S}(\ca)$.
\item  An Isbell topos is a topos of the form $\mathbb{O}(\cx)$.
\item  $\P(X)$ is the category of small copresheaves of $X$.
\end{enumerate}

\begin{notation}[Presentation of a topos]\label{presentation of topos}
A presentation of a topos $\cg$ is the data of a  geometric embedding into a presheaf topos $f^*: \Set^{C} \leftrightarrows \cg :  f_*$. This means precisely that there is a suitable topology $\tau_f$ on $C$ that turns $\cg$ into the category of sheaves over $\tau$; in this sense $f$ \textit{presents} the topos as the category of sheaves over the site $(C, \tau_f)$. 
\end{notation}


\chapter{Promenade}\label{promenade}

The Scott adjunction will be the main character of this thesis. This chapter introduces the reader to the essential aspects of the adjunction and hints at those features that will be developed in the following chapters. From a technical point of view we establish an adjunction between accessible categories with directed colimits and Grothendieck topoi. The qualitative content of the adjunction is twofold. On one hand it has a very clean geometric interpretation, whose roots belong to Stone-like dualities and Scott domains. On the other, it can be seen as a syntax-semantics duality between formal model theory and geometric logic. In this chapter we provide enough information to understand the crude statement of the adjunction and we touch on these contextualizations. One could say that this chapter, together with a couple of results that appear in the Toolbox, is a report of our collaboration with Simon Henry \cite{simon}.

\begin{structure*}
The exposition is organized as follows:
 \begin{enumerate}
 \item[Sec. \ref{promedadescottadj}] we introduce the constructions involved in the Scott adjunction;
  \item[Sec. \ref{promenadecommentsandsuggestions}]  we provide some comments and insights on the first section;
 \item[Sec. \ref{promenadegeneralizations}] we give a quick generalization of the adjunction and discuss its interaction with the standard theorem;
 \item[Sec. \ref{promenadeproofofscott}] we prove the Scott adjunction. 
 \end{enumerate} 
 \end{structure*}

\section{The Scott adjunction: definitions and constructions} \label{promedadescottadj}


We begin by giving the crude statement of the adjunction, then we proceed to construct and describe all the objects involved in the theorem. The actual proof of \ref{scottadj} will close the chapter.

\begin{thm}[{\citep{simon}[Prop. 2.3]} The Scott adjunction]\label{scottadj}
The $2$-functor of points $\mathsf{pt} :\text{Topoi} \to \text{Acc}_{\omega} $ has a left biadjoint $\mathsf{S}$, yielding the Scott biadjunction, $$\mathsf{S} : \text{Acc}_{\omega} \leftrightarrows \text{Topoi}: \mathsf{pt}. $$
\end{thm}

\begin{rem}[Characters on the stage] 
$\text{Acc}_{\omega}$ is the $2$-category of accessible categories with directed colimits, a $1$-cell is a functor preserving directed colimits, $2$-cells are natural transformations. $\text{Topoi}$ is the $2$-category of Grothendieck topoi. A $1$-cell is a geometric morphism and has the direction of the right adjoint. $2$-cells are natural transformation between left adjoints.
 \end{rem}
 
 \begin{rem}[$2$-categorical warnings]\label{ignorepseudo}
Both $\text{Acc}_{\omega}$ and Topoi are  $2$-categories, but most of the time our intuition and our treatment of them will be $1$-categorical, we will essentially downgrade the adjunction to a $1$-adjunction where everything works \textit{up to equivalence of categories}. We feel free to use any classical result about $1$-adjunction, paying the price of decorating any statement with the correct use of the word \textit{pseudo}. For example, right adjoints preserve pseudo-limits, and pseudo-monomorphisms.
 \end{rem}

\begin{rem}[The functor $\mathsf{pt}$]\label{pt}
The functor of points $\mathsf{pt}$ belongs to the literature since quite some time, $\mathsf{pt}$ is the covariant hom functor $\text{Topoi}(\Set, - )$. It maps a Grothendieck topos $\cg$ to its category of points (see Sec. \ref{backgroundtopoi} of the Background), \[\cg \mapsto \text{Cocontlex}(\cg, \Set).\]
Of course given a geometric morphism $f: \cg \to \ce$, we get an induced morphism $\mathsf{pt}(f): \mathsf{pt}(\cg) \to \mathsf{pt}{(\ce)}$ mapping $p^* \mapsto p^* \circ f^*$. The fact that $\text{Topoi}(\Set, \cg)$ is an accessible category with directed colimits appears in the classical reference by Borceux as \citep{borceux_19943}[Cor. 4.3.2], while the fact that $\mathsf{pt}(f)$ preserves directed colimits follows trivially from the definition.
\end{rem}

\subsection{The Scott construction}

 \begin{con}[The Scott construction]\label{defnS}
We recall the construction of $\mathsf{S}$ from \cite{simon}. Let $\ca$ be an accessible category with directed colimits.  $\mathsf{S}(\ca)$ is defined as the category the category of functors preserving directed colimits into sets. \[\mathsf{S}(\ca) = \text{Acc}_{\omega}(\ca, \Set).\]
For a functor $f: \ca \to \cb$ be a $1$-cell in $\text{Acc}_{\omega}$, the geometric morphism $\mathsf{S}f$ is defined by precomposition as described below.

\begin{center}
\begin{tikzcd}
\ca \arrow[dd, "f"] &  & \mathsf{S}\ca \arrow[dd, "f_*"{name=lower, description}, bend left=49] \\
                    &  &                                                \\
\cb                 &  & \mathsf{S}\cb \arrow[uu, "f^*"{name=upper, description}, bend left=49] 

 \ar[phantom, from=lower, to=upper, shorten >=1pc, "\dashv", shorten <=1pc]
\end{tikzcd}
\end{center}
$\mathsf{S}f = (f^* \dashv f_*)$ is defined as follows: $f^*$ is the precomposition functor $f^*(g) = g \circ f$. This is well defined because $f$ preserve directed colimits. $f^*$ is a functor preserving all colimits between locally presentable categories\footnote{This is shown in \ref{scottconstructionwelldefined}.} and thus has a right adjoint by the adjoint functor theorem\footnote{Apply the dual of \cite{borceux_1994}[Thm. 3.3.4] in combination with \cite{adamekrosicky94}[Thm 1.58].}, that we indicate with $f_*$. Observe that $f^*$ preserves finite limits because finite limits commute with directed colimits in $\Set$.
\end{con}

\begin{rem}[$\mathsf{S}(\ca)$ is a topos]\label{scottconstructionwelldefined}
 Together with \ref{defnS} this shows that the Scott construction provides a $2$-functor $\mathsf{S}: \text{Acc}_{\omega} \to \text{Topoi}$. A proof has already appeared in \citep{simon}[2.2] with a practically identical idea. The proof relies on the fact that, since finite limits commute with directed colimits, the category $\mathsf{S}(\ca)$ inherits from its inclusion in the category of all functors $\ca \to \Set$ all the relevant exactness condition prescribed by Giraud axioms. The rest of the proof is devoted to provide a generator for $\mathsf{S}(\ca)$. In the proof below we pack in categorical technology the proof-line above. 
\end{rem}
\begin{proof}[Proof of \ref{scottconstructionwelldefined}]
By definition $\ca$ must be $\lambda$-accessible for some $\lambda$. Obviously $\text{Acc}_\omega(\ca, \Set)$ sits inside $\lambda\text{-Acc}(\ca, \Set)$. Recall that $\lambda\text{-Acc}(\ca, \Set)$ is equivalent to $\Set^{\ca_\lambda}$ by the restriction-Kan extension paradigm and the universal property of $\mathsf{Ind}_\lambda$-completion.  This inclusion $i: \text{Acc}_\omega(\ca, \Set) \hookrightarrow \Set^{\ca_\lambda}$, preserves all colimits and finite limits, this is easy to show and depends on one hand on how colimits are computed in this category of functors, and on the other hand on the fact that in $\Set$ directed colimits commute with finite limits. By the adjoint functor theorem, $\text{Acc}_\omega(\ca, \Set)$ amounts to a coreflective subcategory of a topos whose associated comonad is left exact. By \cite{sheavesingeometry}[V.8 Thm.4], it is a topos.
\end{proof}

\begin{rem}[A description of $f_*$]\label{f_*}
In order to have a better understanding of the right adjoint $f_*$, which in the previous remark was shown to exist via a special version of the adjoint functor theorem, we shall fit the adjunction $(f^* \dashv f_*)$ into a more general picture. We start by introducing the diagram below.

\begin{center}
\begin{tikzcd}
\mathsf{S}\ca \arrow[ddd, "\iota_\ca"] \arrow[rr, "f_*" description, bend right] &  & \mathsf{S}\cb \arrow[ddd, "\iota_\cb"] \arrow[ll, "f^*"'] \\
 &  &  \\
 &  &  \\
{\P(\ca)} \arrow[rr, "\ran_f" description, bend right] \arrow[rr, "\lan_f" description, bend left] &  & {\P(\cb)} \arrow[ll, "f^*" description]
\end{tikzcd}
\end{center}

\begin{enumerate}
\item Recall that by $\P(\ca)$ we mean the category of small copresheaves over $\ca$. We have discussed its properties in Sec. \ref{backgroundionads} of the Background. Observe that the natural inclusion $\iota_\ca$ of $\mathsf{S}\ca$ in $\P(\ca)$ has a right adjoint\footnote{This will be shown in \ref{coreflective}.} $r_{\ca}$, namely $\mathsf{S}\ca$ is coreflective and it coincides with the algebras for the comonad $\iota_\ca \circ r_\ca$. If we ignore the evident size issue for which $\P(\ca)$ is not a topos, the adjunction $\iota_\ca \dashv r_{\ca}$ amounts to a geometric surjection $r: \P(\ca) \to \mathsf{S}\ca$.

 \item The left adjoint $\lan_f$ to $f^*$ does exist because $f$ preserve directed colimits, while in principle $\ran_f$ may not exists because it is not possible to cut down the size of the limit in the ran-limit-formula. Yet, for those functors for which it is defined, it provides a right adjoint for $f^*$. Observe that since the $f^*$ on the top is the restriction of the $f^*$ on the bottom, and $\iota_{\ca, \cb}$ are fully faithful, $f_*$ has to match with $r_\cb \circ \ran_f \circ \iota_\ca$, when this composition is well defined, \[f_* \cong r_\cb \circ \ran_f \circ \iota_\ca,\]indeed the left adjoint $f^*$ on the top coincides with $f^* \circ \iota_\cb$ and by uniqueness of the right adjoint one obtains precisely the equation above. Later in the text this formula will prove to be useful. We can already use it to have some intuition on the behavior  $f_*$, indeed $f_*(p)$ is the best approximation of $\ran_f(p)$ preserving directed colimits. In particular if it happens for some reason that $\ran_f(p)$ preserves directed colimits, then this is precisely the value of $f_*(p)$.
\end{enumerate}
\end{rem}

\begin{rem}\label{generatorscotttopos}
Let $\ca$ be a $\lambda$-accessible category, then $\mathsf{S}(\ca)$ can be described as the full subcategory of $\Set^{\ca_\lambda}$ of those functors preserving $\lambda$-small $\omega$-filtered colimits. A proof of this observation can be found in \citep{simon}[2.2], and in fact shows that $\mathsf{S}(\ca)$ has a generator.
\end{rem}

\section{The Scott adjunction: comments and suggestions} \label{promenadecommentsandsuggestions}

 \begin{rem}[Cameos in the existing literature]\label{scottname}
Despite the name, neither the adjunction nor the construction is due to Scott and was presented for the first time in \cite{simon}. It implicitly appears in special cases both in the classical literature \cite{elephant2} and in some recent developments \cite{anel}. Karazeris introduces the notion of Scott topos of a finitely accessible category $\ck$ in \cite{Karazeris2001}, this notion coincides with $\mathsf{S}(\ck)$, as the name suggests. In Chap. \ref{geometric} we will make the connection with some seminal works of Scott and clarify the reason for which this is the correct categorification of a construction which is due to him. As observed in \citep{simon}[2.4], the Scott construction is the categorification of the usual Scott topology on a poset with directed joins. This will help us to develop a geometric intuition on accessible categories with directed colimits; they will look like the underlying set of some topological space. We cannot say to be satisfied with this choice of name for the adjunction, but we did not manage to come up with a better name.
\end{rem}

\begin{rem}[The duality-pattern] \label{duality}
A duality-pattern is an adjunction that is contravariantly induced by a dualizing object. For example, the famous dual adjunction between frames and topological spaces \cite{sheavesingeometry}[Chap. IX],
\[  \mathcal{O} : \mathsf{Top}   \leftrightarrows  \mathsf{Frm}^\circ:  \mathsf{pt} \]
is induced by the Sierpinski space $\mathbb{T}$. Indeed, since it admits a natural structure of frame, and a natural structure of topological space the adjunction above can be recovered in the form \[  \mathsf{Top}(-,\mathbb{T}): \mathsf{Top} \leftrightarrows  \mathsf{Frm}^\circ   : \mathsf{Frm}(-,\mathbb{T}).\] Most of the known topological dualities are induced in this way. The interested reader might want to consult \cite{porst1991concrete}. Makkai has shown (\cite{Makkai-Pitts}, \cite{makkai88}) that relevant families of syntax-semantics dualities can be induced in this way using the category of sets as a dualizing object.
In this fashion, the content of  Rem. \ref{defnS} together with Rem. \ref{pt} acknowledges that $\mathsf{S} \dashv \mathsf{pt}$ is essentially induced by the object $\Set$ that inhabits both the $2$-categories.
\end{rem}

\begin{rem}[Generalized axiomatizations] \label{Free geometric theory}
As was suggested by Joyal, the category $\text{Logoi} = \text{Topoi}^{\op}$ can be seen as the category of geometric theories. Caramello \cite{caramello2010unification} pushes on the same idea stressing the fact that a topos is a Morita-equivalence class of geometric theories. In this perspective the Scott adjunction, which in this case is a dual adjunction \[\text{Acc}_{\omega} \leftrightarrows \text{Logoi}^{\op}, \] presents $\mathsf{S}(\ca)$ as a free geometric theory attached to the accessible category $\ca$ that is willing to axiomatize $\ca$. When $\ca$ has a faithful functor preserving directed colimits into the category of sets, $\mathsf{S}(\ca)$ axiomatizes an envelope of $\ca$, as proved in \ref{ff}. A logical understanding of the adjunction will be developed in Chap. \ref{logical}, where we connect the Scott adjunction to the theory of classifying topoi and to the seminal works of Lawvere and Linton in categorical logic. This intuition will be used also to give a topos theoretic approach to abstract elementary classes.
\end{rem}

\begin{rem}[Trivial behaviors and Diaconescu]\label{trivial}
 If $\ck$ is a finitely accessible category, $\mathsf{S}(\ck)$ coincides with the presheaf topos $\Set^{\ck_{\omega}}$, where we indicated with $\ck_{\omega}$ the full subcategory of finitely presentable objects. This follows directly from the following chain of equivalences, \[\mathsf{S}(\ck) = \text{Acc}_{\omega}(\ck, \Set)  \simeq  \text{Acc}_{\omega}(\mathsf{Ind}(\ck_{\omega}), \Set)  \simeq \Set^{\ck_{\omega}}. \] As a consequence of Diaconescu theorem \cite{elephant1}[B3.2.7] and the characterization of the $\mathsf{Ind}$-completion via flat functors, when restricted  to finitely accessible categories, the Scott adjunction yields a biequivalence of $2$-categories $\omega\text{-Acc} \simeq \text{Presh}$, with Presh the full $2$-subcategory of presheaf topoi. 
 \begin{center}
\begin{tikzcd}
\text{Acc}_{\omega} \arrow[rr, "\mathsf{S}" description, bend right] &  & \text{Topoi} \arrow[ll, "\mathsf{pt}" description, bend right] \\
 &  &  \\
\omega\text{-Acc} \arrow[uu, hook] \arrow[rr, bend right] &  & \text{Presh} \arrow[uu, hook] \arrow[ll, bend right]
\end{tikzcd}
\end{center}  
This observation is not new to literature, the proof of \citep{elephant2}[C4.3.6] gives this special case of the Scott adjunction. It is very surprising that the book does not investigate, or even mention the existence of the Scott adjunction, since it gets very close to defining it explicitly.
\end{rem}

\begin{thm}
The Scott adjunction restricts to a biequivalence of $2$-categories between the $2$-category of finitely accessible categories\footnote{With finitely accessible functors and natural transformation.} and the $2$-category of presheaf topoi\footnote{With geometric morphisms and natural transformations between left adjoints.}.

$$\mathsf{S} : \omega\text{-Acc} \leftrightarrows \text{Presh}: \mathsf{pt}. $$
\end{thm}
\begin{proof}
The previous remark has shown that  when $\ca$ is finitely accessible, $\mathsf{S}(\ca)$ is a presheaf topos and that, when $\ce$ is a presheaf topos, $\mathsf{pt}(\ce)$ is finitely accessible. To finish, we show that in this case the unit and the counit of the Scott adjunction are equivalence of categories. This is in fact essentially shown by the previous considerations. $$(\mathsf{pt}\mathsf{S}) (\mathsf{Ind}(C)) \simeq \mathsf{pt}(\Set^C) \stackrel{\text{Diac}}{\simeq} \mathsf{Ind}(C). $$

$$(\mathsf{S}\mathsf{pt}) (\Set^C) \stackrel{\text{Diac}}{\simeq} \mathsf{S}(\mathsf{Ind}(C)) \simeq \Set^C. $$
Notice that the equivalences above are precisely the unit and the counit of the Scott adjunction as described in Sec \ref{promenadeproofofscott} of this Chapter.
\end{proof}

\begin{rem} \label{trivial1}
Thus, the Scott adjunction must induce an equivalence of categories between the Cauchy completions\footnote{The free completions that adds splittings of \textit{pseudo}-idempotents.} of $\omega$-Acc and Presh. The Cauchy completion of $\omega$-Acc is the full subcategory of $\text{Acc}_\omega$ of \textit{continuous categories} \citep{cont}. Continuous categories are the categorification of the notion of continuous poset and can be characterized as split subobjects of finitely accessible categories in  $\text{Acc}_\omega$. In \citep{elephant2}[C4.3.6] Johnstone observe that if a continuous category is cocomplete, then the corresponding Scott topos is injective (with respect to geometric embeddings) and vice-versa.
\end{rem}

\begin{exa}
As a direct consequence of Rem. \ref{trivial}, we can calculate the Scott topos of $\Set$. $\mathsf{S}(\Set)$ is  $\Set^{\text{FinSet}}$. This topos is very often indicated as $\Set[\mathbb{O}]$, being the classifying topos of the theory of objects, i.e. the functor: $\text{Topoi}(-, \Set[\mathbb{O}]): \text{Topoi}^{\circ} \to \text{Cat}$ coincides with the forgetful functor. As a reminder for the reader, we state clearly the equivalence: \[\mathsf{S}(\Set) \simeq \Set[\mathbb{O}]. \]
\end{exa}

\begin{rem}[The Scott adjunction is not a biequivalence: Fields] \label{fields}
Whether the Scott adjunction is a biequivalence is a very natural question to ask. Up to this point we noticed that on the subcategory of topoi of presheaf type the counit of the adjunction is in fact an equivalence of categories. Since presheaf topoi are a quite relevant class of topoi one might think that the whole bi-adjunction amounts to a biequivalence. That's not the case: in this remark we provide a topos $\cf$ such that the counit \[\epsilon_\cf : \mathsf{Spt} \cf \to \cf\] is not an equivalence of categories. Let $\cf$ be the classifying topos of the theory of geometric fields \cite{elephant2}[D3.1.11(b)]. Its category of points is the category of fields $\mathsf{Fld}$, since this category is finitely accessible the Scott topos $\mathsf{Spt}(\cf)$ is of presheaf type by Rem. \ref{trivial}, \[\mathsf{Spt}(\cf) = \mathsf{S} (\mathsf{Fld} ) \stackrel{\ref{trivial}}{\cong} \Set^{\mathsf{Fld}_\omega}.\]
It was shown in \cite{10.2307/30041767}[Cor 2.2] that $\cf$ cannot be of presheaf type, and thus $\epsilon_\cf$ cannot be an equivalence of categories.

\end{rem}

\begin{rem}[Classifying topoi for categories of diagrams]
Let us give a proof in our framework of a well known fact in topos theory, namely that a the category of diagrams over the category of points of a topos can be axiomatized by a geometric theory. This means that there exists a topos $\cf$  such that 
$$\mathsf{pt}(\ce)^C \simeq  \mathsf{pt}(\cf).$$
The proof follows from the following chain of equivalences. 
\begin{align*} 
\mathsf{pt}(\ce)^C  = &   \text{Cat}(C,\mathsf{pt}(\ce))  \\ 
\simeq & \text{Acc}_\omega(\mathsf{Ind}(C), \mathsf{pt}(\ce)) \\
\simeq & \text{Topoi}(\mathsf{S}\mathsf{Ind}(C), \ce) \\
\simeq & \text{Topoi}(\Set^C, \ce) \\
\simeq & \text{Topoi}(\Set \times \Set^C, \ce) \\
\simeq & \text{Topoi}(\Set, \ce^{\Set^C}) \\
\simeq & \mathsf{pt}(\ce^{\Set^C}).
\end{align*}

\end{rem}

\section{The $\kappa$-Scott adjunction}\label{kappa} \label{promenadegeneralizations}

The most natural generalization of the Scott adjunction is the one in which directed colimits are replaced with $\kappa$-filtered colimits and finite limits ($\omega$-small) are replaced with $\kappa$-small limits. This unveils the deepest reason for which the Scott adjunction exists: namely $\kappa$-directed colimits commute with $\kappa$-small limits in the category of sets.

\begin{thm}\citep{simon}[Prop 3.4]
There is  an $2$-adjunction $$\mathsf{S}_{\kappa} : \text{Acc}_{\kappa} \leftrightarrows \kappa\text{-Topoi}: \mathsf{pt}_{\kappa}. $$
\end{thm}

\begin{rem}
$\text{Acc}_{\kappa}$ is the $2$-category of accessible categories with $\kappa$-directed colimits, a $1$-cell is a functor preserving $\kappa$-filtered colimits, $2$-cells are natural transformations. $\text{Topoi}_{\kappa}$ is the $2$-category of Groethendieck $\kappa$-topoi. A $1$-cell is a $\kappa$-geometric morphism and has the direction of the right adjoint. $2$-cells are natural transformation between left adjoints. A $\kappa$-topos is a $\kappa$-exact localization of a presheaf category. These creatures are not completely new to the literature but they appear sporadically and a systematic study is still missing. We should reference, though, the works of Espindola \cite{ESPINDOLA2019137}. We briefly recall the relevant definitions.
 \end{rem}

 \begin{defn}
 A $\kappa$-topos is a $\kappa$-exact localization of a presheaf category.
 \end{defn}

  \begin{defn}
 A $\kappa$-geometric morphism $f: \ce \to \cf$ between $\kappa$-topoi is an adjunction $f^*: \cf \leftrightarrows \ce: f_*$ whose left adjoint preserve $\kappa$-small limits.
 \end{defn}
 
 \begin{rem}\label{generalize}
It is quite evident that every remark until this point  finds its direct $\kappa$-generalization substituting every occurrence of \textit{directed colimits }with \textit{$\kappa$-directed colimits.}
 \end{rem}

\begin{rem}\label{kdiac}
Let $\ca$ be a category in $\text{Acc}_{\omega}$.  For a large enough $\kappa$ the Scott adjunction axiomatizes $\ca$ (in the sense of Rem. \ref{Free geometric theory}), in fact if $\ca$ is $\kappa$-accessible $\mathsf{pt}_\kappa \mathsf{S}_\kappa \ca \cong \ca$, for the $\kappa$-version of Diaconescu Theorem, that in this text appears in Rem. \ref{trivial}.
\end{rem} 
 
 \begin{rem}
 It pretty evident that $\lambda$-Topoi is a locally fully faithful sub $2$-category of $\kappa$-Topoi when $\lambda \geq \kappa$. The same is true for $\text{Acc}_\lambda$ and $\text{Acc}_\kappa$. This observation leads to a filtration of the categories Topoi and $\text{Acc}_\omega$ as shown in the following diagram,

 \begin{center}
\begin{tikzcd}
\kappa             &  & \text{Acc}_{\kappa} \arrow[dd, "\iota_\kappa^\lambda" description] \arrow[rr, "\mathsf{S}_\kappa" description, bend right=15]   &  & \kappa \text{-Topoi} \arrow[dd, "i_\kappa^\lambda" description] \arrow[ll, "\mathsf{pt}_\kappa" description, bend right=15]   \\
                   &  &                                                                                                                              &  &                                                                                                                            \\
\lambda \arrow[uu] &  & \text{Acc}_{\lambda} \arrow[dd, "\iota_\lambda^\omega" description] \arrow[rr, "\mathsf{S}_\lambda" description, bend right=15] &  & \lambda \text{-Topoi} \arrow[dd, "i_\lambda^\omega" description] \arrow[ll, "\mathsf{pt}_\lambda" description, bend right=15] \\
                   &  &                                                                                                                              &  &                                                                                                                            \\
\omega \arrow[uu]  &  & \text{Acc}_{\omega} \arrow[rr, "\mathsf{S}" description, bend right=15]                                                         &  & (\omega \text{-})\text{Topoi} \arrow[ll, "\mathsf{pt}" description, bend right=15]                                           
\end{tikzcd}
\end{center} 
 \end{rem}

 \begin{rem}[The diagram does \textbf{not} commute]
 We depicted the previous diagram in order to trigger the reader's pattern recognition and conjecture its commutativity. In this remark we stress that the diagram does \textbf{not} commute, meaning that \[\mathsf{S}_\lambda \circ \iota_\kappa^\lambda \not \simeq i_\kappa^\lambda \circ \mathsf{S}_\kappa,\] at least not in general. In fact, once the definitions are spelled out, there is absolutely no reasons why one should have commutativity in general. The same is true for the right adjoint $\mathsf{pt}$.
 \end{rem}

\begin{rem}
In the following diagram we show the interaction between Rem. \ref{kdiac} and the previous remark. Recall that presheaf categories belong to $\kappa$-topoi for every $\kappa$.
\begin{center}
\begin{tikzcd}
\kappa\text{-Acc} \arrow[rr]  &  & \text{Acc}_{\kappa} \arrow[dd, "\iota_\kappa^\lambda" description] \arrow[rr, "\mathsf{S}_\kappa" description, bend right=15]   &  & \kappa \text{-Topoi} \arrow[dd, "i_\kappa^\lambda" description] \arrow[ll, "\mathsf{pt}_\kappa" description, bend right=15]   &  & \text{Presh} \arrow[ll] \\
                              &  &                                                                                                                              &  &                                                                                                                            &  &                         \\
\lambda\text{-Acc} \arrow[rr] &  & \text{Acc}_{\lambda} \arrow[dd, "\iota_\lambda^\omega" description] \arrow[rr, "\mathsf{S}_\lambda" description, bend right=15] &  & \lambda \text{-Topoi} \arrow[dd, "i_\lambda^\omega" description] \arrow[ll, "\mathsf{pt}_\lambda" description, bend right=15] &  & \text{Presh} \arrow[ll] \\
                              &  &                                                                                                                              &  &                                                                                                                            &  &                         \\
\omega\text{-Acc} \arrow[rr]  &  & \text{Acc}_{\omega} \arrow[rr, "\mathsf{S}" description, bend right=15]                                                         &  & (\omega \text{-})\text{Topoi} \arrow[ll, "\mathsf{pt}" description, bend right=15]                                            &  & \text{Presh} \arrow[ll]
\end{tikzcd}
\end{center}
\end{rem}

\begin{rem} It might be natural to conjecture that Presh happens to be $\bigcap_\kappa \kappa \text{-Topoi}$.  Simon Henry, has provided a counterexample to this superficial conjecture, namely $\mathsf{Sh}([0,1])$. \cite{kelly1989complete}[Rem. 4.4, 4.5 and 4.6] gives a theoretical reason for which many other counterexamples do exist and then provides a collection of them in Sec. 5 of the same paper.
\end{rem}

\section{Proof of the Scott Adjunction} \label{promenadeproofofscott}

We end this chapter including a full proof of the Scott adjunction.

\begin{proof}[Proof of Thm. \ref{scottadj}] 
We prove that there exists an equivalence of categories, $$\text{Topoi}(\mathsf{S}(\ca), \cf) \cong \text{Acc}_\omega(\ca, \mathsf{pt}(\cf)). $$ The proof makes this equivalence evidently natural. This proof strategy is similar to that appearing in \cite{simon}, even thought it might look different at first sight.

\begin{align*} 
\text{Topoi}(\mathsf{S}(\ca), \cf) \cong &  \text{Cocontlex}(\cf, \mathsf{S}(\ca))  \\ 
 \cong & \text{Cocontlex}(\cf, \text{Acc}_\omega(\ca, \Set))  \\
\cong & \text{Cat}_{\text{cocontlex,}\text{acc}_\omega}(\cf \times \ca, \Set) \\
\cong &  \text{Acc}_\omega(\ca, \text{Cocontlex}(\cf,\Set))  \\
 \cong &  \text{Acc}_\omega(\ca, \text{Topoi}(\Set,\cf)). \\
  \cong & \text{Acc}_\omega(\ca, \mathsf{pt}(\cf)).
\end{align*}

\end{proof}

\begin{proof}[A description of the (co)unit]
We spell out the unit and the counit of the adjunction. 
\begin{itemize}
  \item[$\eta$] For an accessible category with directed colimits $\ca$ we must provide a functor $\eta_\ca: \ca \to \mathsf{pt}\mathsf{S}(\ca)$. Define, $$\eta_\ca(a)(-) := (-)(a).$$  $\eta_\ca(a): \mathsf{S}(\ca) \to \Set$ defined in this way  is a functor preserving colimits and finite limits and thus defines a point of $\mathsf{S}(\ca)$.
  \item[$\epsilon$] The idea is very similar, for a topos $\ce$, we must provide a geometric morphism $\epsilon_\ce: \mathsf{S}\mathsf{pt}(\ce) \to \ce$. Being a geometric morphism, it's equivalent to provide a cocontinuous and finite limits preserving functor $\epsilon_\ce^*: \ce \to \mathsf{S}\mathsf{pt}(\ce)$. Define, $$\epsilon_\ce^*(e)(-) = (-)^*(e). $$
\end{itemize}
\end{proof}

Our first encounter with the Scott adjunction is over, the following two chapters will try to give a more precise intuition to the reader, depending on the background.

\chapter{Geometry}\label{geometric}

This chapter is dedicated to unveiling the geometric flavor of the Scott adjunction. We build on a well understood analogy between topoi and locales to bring the geometric intuition on the Scott adjunction. We show that this intuition is well founded and fruitful both in formulating the correct guesses and directing the line of thought. This perspective is not new; as anticipated, the analogy between the notion of locale and that of topos was known since the very introduction of the latter. Our contribution is thus a step towards the categorification process of poset theory into actual category theory. The relation with the existing literature will be discussed during the chapter.

\begin{structure*}
The chapter is divided in six sections,
\begin{enumerate}
	\item[Sec. \ref{topology}] we recall the concrete topology on which the analogy is built: the Isbell duality, relating topological spaces to locales. We also relate the Isbell duality to Scott's seminal work on the Scott topology, this first part is completely motivational and expository. This section contains the posetal version of sections \ref{categorification}, \ref{soberenough} and \ref{isbelltoscottcategorified}.
	\item[Sec. \ref{categorification}]
	We introduce the higher dimensional analogs of topological spaces, locales and posets with directed colimits: ionads, topoi and accessible categories with directed colimits. We categorify the Isbell duality building on Garner's work on ionads, and we relate the Scott adjunction to its posetal analog.
	\item[Sec. \ref{soberenough}] We study the the categorified version of Isbell duality. This amounts to the notion of sober ionad and topos with enough points. We show that the categorified Isbell adjunction is idempotent.
	\item[Sec. \ref{isbelltoscottcategorified}] We use the previous section to derive properties of the Scott adjunction.
	\item[Sec. \ref{interaction}] We show that the analogy on which the chapter is built is deeply motivated and we show how to recover the content of the first section from the following ones.
	\item[Sec. \ref{kappaIsbell}] In the last section we provide an expected generalization of the second one to $\kappa$-topoi and $\kappa$-ionads.
	\end{enumerate}
\end{structure*}

\section{Spaces, locales and posets} \label{topology}

Our topological safari will start from the celebrated adjunction between locales and topological spaces. This was first observed by Isbell, whence the name Isbell adjunction/duality. Unfortunately this name is already taken by the dual adjunction between presheaves and co-presheaves; this sometimes leads to some terminological confusion. The two adjunctions are similar in spirit, but do not generalize, at least not apparently, in any common framework. This first subsection is mainly expository and we encourage the interested reader to consult \cite{johnstone1986stone} and \cite{sheavesingeometry}[Chap. IX] for a proper introduction. The aim of the subsection is not to introduce the reader to these results and objects, it is instead to organize them in a way that is useful to our purposes. More precise references will be given within the section.

\subsection{Spaces, Locales and Posets} \label{topologyintro}
This subsection tells the story of the diagram below. Let us bring every character to the stage.

\begin{center}
\begin{tikzcd}
                                                              & \text{Loc} &                                                                                                 \\
                                                              &                                                                                                                  &                                                                                                 \\
                                                              &                                                                                                                  &                                                                                                 \\
\text{Top} \arrow[ruuu, "\mathcal{O}" description, ] &                                                                                                                  & \text{Pos}_{\omega} \arrow[luuu, "\mathsf{S}" description, ] \arrow[ll, "\mathsf{ST}"]
\end{tikzcd}
\end{center}

\begin{rem}[The categories] 
\begin{enumerate}
\item[]
\item[Loc] is the category of locales. It is defined to be the opposite category of frames, where objects are frames and morphisms are morphisms of frames.
\item[Top] is the category of topological spaces and continuous mappings between them.
\item[$\text{Pos}_{\omega}$] is the category of posets with directed suprema and functions preserving directed suprema.
\end{enumerate}
\end{rem}

\begin{rem}[The functors]
The functors in the diagram above are well known to the general topologist; we proceed to a short description of them.
\begin{enumerate}
	\item[$\mathcal{O}$] associates to a topological space its frame of open sets and to each continuous function its inverse image.
	\item[$\mathsf{ST}$] equips a poset with directed suprema with the Scott topology \cite{johnstone1986stone}[Chap. II, 1.9].  This functor is fully faithful, i.e. a function is continuous with respect to the Scott topologies if and only if preserves suprema.
	\item[$\mathsf{S}$] is the composition of the previous two; in particular the diagram above commutes.
\end{enumerate}
\end{rem}

\begin{rem}[The Isbell duality and a posetal version of the Scott adjunction] \label{scottisbell}
Both the functors $\mathcal{O}$ and $\mathsf{S}$ have a right adjoint; we indicate them by $\mathsf{pt}$ and $\mathbbm{pt}$, which in both cases stands for \textit{points}. 

\begin{center}
\begin{tikzcd}
                                                                      & \text{Loc} \arrow[lddd, "\mathbbm{pt}" description, bend left=12] \arrow[rddd, "\mathsf{pt}" description, bend left=12] &                                                                                                                 \\
                                                                      &                                                                                                                  &                                                                                                                 \\
                                                                      &                                                                                                                  &                                                                                                                 \\
\text{Top} \arrow[ruuu, "\mathcal{O}" description, dashed, bend left=12] &                                                                                                                  & \text{Pos}_{\omega} \arrow[luuu, "\mathsf{S}" description, dashed, bend left=12] \arrow[ll, "\mathsf{ST}", dashed]
\end{tikzcd}
\end{center}
In the forthcoming remarks the reason for this clash of names will be motivated; indeed the two functors look alike. The adjunction on the left is \textit{Isbell duality}, while the one on the right was not named yet to our knowledge and we will refer to it as the \textit{(posetal) Scott adjunction}. Let us mention that there exists a natural transformation, $$\iota: \mathsf{ST} \circ \mathsf{pt} \Rightarrow \mathbbm{pt}$$ which will be completely evident from the description of the functors. We will say more about $\iota$; for the moment we just want to clarify that there is no reason to believe (and indeed it would be a false belief) that $\iota$ is an isomorphism.
\end{rem}

\begin{rem}[The Isbell duality] \label{isbellposetal}
An account of the adjunction $\mathcal{O} \dashv \mathbbm{pt}$ can be found in the very classical reference \citep{sheavesingeometry}[IX.1-3]. While the description of $\mathcal{O}$ is relatively simple, (it associates to a topological space $X$ its frame of opens $\mathcal{O}(X))$, $\mathsf{pt}$ is more complicated to define. It associates to a locale $L$ its sets of \textit{formal points} $\text{Loc}(\mathbb{T}, L)$\footnote{$\mathbb{T}$ is the boolean algebra $\{0 < 1\}$.} equipped with the topology whose open sets are defined as follows: for every $l$ in $L$ we pose, $$V(l):=\{p \in \text{Frm}(L,\mathbb{T}): p(l)=1\}.$$ Further details may be found in \cite{sheavesingeometry}[IX.2]. In the literature, such an adjunction that is induced by a contravariant hom-functor is called schizophrenic. The special object in this case is $2$ that inhabits both the categories under the alias of Sierpiński space in the case of Top and the terminal object in the case of Frm. A detailed account on this topic can be found in \citep{sheavesingeometry}[IX.3, last paragraph].
\end{rem}

\begin{rem}[The (posetal) Scott adjunction] \label{scottposetal}
To our knowledge, the adjunction $\mathsf{S} \dashv \mathsf{pt}$ does not appear explicitly in the literature. Let us give a description of both the functors involved.
\begin{itemize}
	\item[$\mathsf{pt}$] The underlying set of $\mathsf{pt}(L)$ is the same as for Isbell duality. Its posetal structure is inherited from $\mathbb{T}$; in fact $\text{Frm}(L,\mathbb{T})$ has a natural poset structure with directed unions given by pointwise evaluation \cite{Vickers2007}[1.11].
	\item[$\mathsf{S}$] Given a poset $P$, its Scott locale $\mathsf{S}(P)$ is defined by the frame $\text{Pos}_{\omega}(P, \mathbb{T})$. It's quite easy to show that this poset is a locale.
	\end{itemize}
Observe that also this adjunction is a dual one, and is induced by precisely the same object as for Isbell duality.	
\end{rem}


\begin{rem}[An unfortunate naming]
There are several reasons for which we are not satisfied with the naming choices in this thesis. Already in the topological case, Isbell duality (or adjunction) is a very overloaded name, and later in the text we will categorify this adjunction, calling it \textit{categorified Isbell duality}, which propagates an unfortunate choice. Also the name of Scott for the Scott adjunction is not completely proper, because he introduced the Scott topology on a set, providing the functor $\mathsf{ST}$. In the previous chapter, we called Scott adjunction the categorification of the posetal Scott adjunction in this section, propagating this incorrect attribution. Yet, we did not manage to find better options, and thus we will stick to these choices.
\end{rem}

\subsection{Sober spaces and spatial locales}\label{topologysoberandspatial}

The Isbell adjunction is a fascinating construction that in principle could be an equivalence of categories. \textit{Shouldn't a space be precisely the space of formal points of its locale of open sets?} It turns out that the answer is in general negative; this subsection describes how far this adjunction is from being an equivalence of categories.

\begin{rem}[Unit and counit]\label{topologyisbellunitcounit}
Given a space $X$ the unit of the Isbell adjunction $\eta_X: X \to (\mathsf{pt} \circ \mathcal{O})(X)$ might not be injective. This is due to the fact that if two points $x,y$ in $X$ are such that $\mathsf{cl}(x) = \mathsf{cl}(y)$, then $\eta_X$ will confuse them.
\end{rem}

\begin{rem}[Sober spaces and spatial locales] 
In the classical literature about this adjunction people have introduced the notion of sober space and locale with enough points, these are precisely those objects on which the (co)unit is an iso. It turns out that even if $\eta_X$ is not always an iso $\eta_{\mathsf{pt}(L)}$ is always an iso and this characterizes those $\eta$ that are isomorphisms. An analogue result is true for the counit.
\end{rem}

\begin{rem}[Idempotency]\label{topologyidempotency}
The technical content of the previous remark is summarized in the fact that the Isbell adjunction $\mathcal{O} \dashv \mathsf{pt}$ is idempotent; this is proved in \citep{sheavesingeometry}[IX.3][Prop. 2, Prop. 3 and Cor. 4.]. It might look like this result has no qualitative meaning. Instead it means that given a local $L$, the locale of opens sets of its points $\mathcal{O}\mathsf{pt}(L)$ is the best approximation of $L$ among spatial locales, namely those that are the locale of opens of a space. The same observation is true for a space $X$ and the formal points of its locale of open sets $\mathsf{pt}\mathcal{O}(X)$. In the next two proposition we give a more categorical and more concrete incarnation of this remark.
\end{rem}

\begin{thm}[{\citep{sheavesingeometry}[IX.3.3]}] The following are equivalent:
\begin{enumerate}
	\item $L$ has enough points;
	\item the counit $\epsilon_L : (\mathcal{O} \circ \mathsf{pt})(L) \to L$ is an isomorphism of locales;
	\item $L$ is the locale of open sets $\mathcal{O}(X)$ of some topological space $X$.
\end{enumerate}
\end{thm}

\begin{thm}
The subcategory of locales with enough points is coreflective in the category of locales.
\end{thm}

\begin{rem}\label{onesideoftheadj}
Similarly, the subcategory of sober spaces is reflective in the category of spaces. This is not surprising, it's far form being true that any adjunction is idempotent, but it is easy to check that given an adjunction whose induced comonad is idempotent, so must be the induced monad (and vice-versa).
\end{rem}

\subsection{From Isbell to Scott: topology}\label{topologyscottfromisbell}

\begin{rem}[Relating the Scott construction to the Isbell duality]
Going back to the (non-)commutativity of the diagram in Rem. \ref{scottisbell}, we observe that there exists a natural transformation $\iota: \mathsf{ST} \circ \mathsf{pt} \Rightarrow \mathbbm{pt}$.

\begin{center}
\begin{tikzcd}
           & \text{Loc} \arrow[lddd, "\mathbbm{pt}" description] \arrow[rddd, "\mathsf{pt}"{name=pt}, description] &                                               \\
           &                                                                                             &                                               \\
           &                                                                                             &                                               \\
\text{Top} \ar[Rightarrow, from=pt, shorten >=1pc, "\iota", shorten <=1pc]&                                                                                             & \text{Pos}_{\omega} \arrow[ll, "\mathsf{ST}"]
\end{tikzcd}
\end{center}
The natural transformation is pointwise given by the identity (for the underlying set is indeed the same), and witnessed by the fact that every Isbell-open (Rem. \ref{isbellposetal}) is a Scott-open (Rem. \ref{scottposetal}). This observation is implicitly written in \cite{johnstone1986stone}[II, 1.8].
\end{rem}

\begin{rem}[Scott is not always sober] \label{topologyscottnotsober}
In principle $\iota$ might be an isomorphism. Unfortunately it was shown by Johnstone in \cite{10.1007/BFb0089911} that some Scott-spaces are not sober. Since every space in the image of $\mathbbm{pt}$ is sober, $\iota$ cannot be an isomorphism at least in those cases. Yet, Johnstone says in \cite{10.1007/BFb0089911} that he does not know any example of a complete lattice whose Scott topology is not sober. Thus it is natural to conjecture that when $\mathsf{pt}(L)$ is complete, then $\iota_L$ is an isomorphism. We will not only show that this is true, but even provide a generalization of this result later.
\end{rem}

\begin{rem}[Scott from Isbell] \label{topologyscottfromisbellthm}
Let us conclude with a version of {\citep{sheavesingeometry}[IX.3.3]} for the Scott adjunction. This has guided us in understanding the correct idempotency of the Scott adjunction.
\begin{thm}[Consequence of {\citep{sheavesingeometry}[IX.3.3]}] The following are equivalent:
\begin{enumerate}
	\item $L$ has enough points and $\iota$ is an isomorphism at $L$;
	\item the counit of the Scott adjunction is an isomorphism of locales at $L$.
\end{enumerate}
\end{thm}
\begin{proof}
It follows directly from {\citep{sheavesingeometry}[IX.3.3]} and the fact that $\mathsf{ST}$ is fully faithful.
\end{proof}
\end{rem}

\section{Ionads, Topoi and Accessible categories } \label{categorification}
Now we come to the $2$-dimensional counterpart of the previous section. As in the previous one, this section is dedicated to describing the properties of a diagram.
\begin{center}
\begin{tikzcd}
                                                              & \text{Topoi} &                                                                                                 \\
                                                              &                                                                                                                  &                                                                                                 \\
                                                              &                                                                                                                  &                                                                                                 \\
\text{BIon} \arrow[ruuu, "\mathbb{O}" description, ] &                                                                                                                  & \text{Acc}_{\omega} \arrow[luuu, "\mathsf{S}" description, ] \arrow[ll, "\mathsf{ST}"]
\end{tikzcd}
\end{center}

\subsection{Motivations} \label{categoriesmotivations}

There are no doubts that we drew a triangle which is quite similar to the one in the previous section, but \textit{in what sense are these two triangles related?} There is a long tradition behind this question and too many papers should be cited. In this very short subsection we provide an intuitive account regarding this question.

\begin{rem}[Replacing posets with categories]
 There is a well known analogy\footnote{It is much more than an analogy, but this remark is designed to be short, motivational and inspirational. Being more precise and mentioning enriched categories over truth values would not give a more accessible description to the generic reader.} between the category of posets $\mathsf{Pos}$ and the category of categories $\mathsf{Cat}$. A part of this analogy is very natural: joins and colimits, meets and limits, monotone functions and functors. Another part might appear a bit counter intuitive at first sight. The poset of truth values $\mathbb{T}=\{0 < 1\}$ plays the rôle of the category of sets. The inclusion of a poset $i: P \to \mathbb{T}^{P^\circ}$ in its poset of ideals plays the same rôle of the Yoneda embedding.
\end{rem}

\begin{table}[!htbp]
\begin{tabular}{lllll}
\cline{2-3}
\multicolumn{1}{l|}{} & \multicolumn{1}{c|}{Pos} & \multicolumn{1}{c|}{Cat} &  &  \\ \cline{2-3}
\multicolumn{1}{l|}{} & \multicolumn{1}{c|}{$P \to \mathbb{T}^{P^\circ}$} & \multicolumn{1}{c|}{$C \to \Set^{C^\circ}$} &  &  \\ \cline{2-3}
\multicolumn{1}{l|}{} & \multicolumn{1}{c|}{joins} & \multicolumn{1}{c|}{colimits} &  &  \\ \cline{2-3}
\multicolumn{1}{l|}{} & \multicolumn{1}{c|}{meets} & \multicolumn{1}{c|}{limits} &  &  \\ \cline{2-3}
\multicolumn{1}{l|}{} & \multicolumn{1}{c|}{monotone function} & \multicolumn{1}{c|}{functor} &  &  \\ \cline{2-3}
\end{tabular}
\end{table}

\begin{rem}[$\text{Pos}_{\omega}$ and $\text{Acc}_{\omega}$]
Following the previous remark one would be tempted to say that posets with directed colimits correspond to categories with directed colimits. Thus, \textit{why put the accessibility condition on categories?} The reason is that in the case of posets the accessibility condition is even stronger, even if hidden. In fact a poset is a small (!) poclass, and poclasses with directed joins would be the correct analog of categories with directed colimits. Being accessible is a way to have control on the category without requesting smallness, which would be too strong an assumption.
\end{rem}

\begin{rem}[Locales and Topoi]
 Quite surprisingly the infinitary distributivity rules which characterizes locales has a description in term of the posetal \textit{Yoneda} inclusion. Locales can be described as those posets whose (Yoneda) inclusion has a left adjoint preserving finite joins, $$L: \mathbb{T}^{P^\circ} \leftrightarrows P: i.$$ In the same fashion, Street \cite{street_1981} proved that Topoi can be described as those categories (with a generator) whose Yoneda embedding $\cg \to \Set^{\cg^\circ}$ has a left adjoint preserving finite limits. In analogy with the case of locales this property is reflected by a kind of interaction been limits and colimits which is called \textit{descent}. In this sense a topos is a kind of $\Set$-locale, while a locale is a $\mathbb{T}$-locales. In the next section we will show that there is an interplay between these two \textit{notions of locale}. 
\end{rem}

\begin{rem}[Spaces and Ionads]
While topoi are the categorification of locales, ionads are the categorification of topological spaces. Recall that a topological space, after all, is just a set $X$ together with the data of its interior operator \[ \mathsf{int}: 2^X \to 2^X .\] This is an idempotent operator preserving finite meets. We will see that ionads are defined in a very similar way, following the pattern of the previous remarks.
\end{rem}

\subsection{Categorification} \label{subsectioncategorification}

Now that we have given some motivation for this to be the correct categorification of the Isbell duality, we have to present in more mathematical detail all the ingredients that are involved. A part of the triangle in this section is just the Scott adjunction, that we understood quite well in the previous section. Here we have to introduce ionads and all the functors in which they are involved.

\begin{center}
\begin{tikzcd}
                                                              & \text{Topoi} &                                                                                                 \\
                                                              &                                                                                                                  &                                                                                                 \\
                                                              &                                                                                                                  &                                                                                                 \\
\text{BIon} \arrow[ruuu, "\mathbb{O}" description, ] &                                                                                                                  & \text{Acc}_{\omega} \arrow[luuu, "\mathsf{S}" description, ] \arrow[ll, "\mathsf{ST}"]
\end{tikzcd}
\end{center}

\begin{rem}[The ($2$-)categories] 
\begin{enumerate}
\item[]
\item[Topoi] is the $2$-category of topoi and geometric morphisms.
\item[BIon] is the $2$-category of (generalized) bounded ionads.
\item[$\text{Acc}_{\omega}$] is the $2$-category of accessible categories with directed colimits and functors preserving them.
\end{enumerate}
\end{rem}

\subsubsection{The functors}

\begin{rem}[$\mathbb{O}$] \label{opensionad}
 Let us briefly recall the relevant definitions of Chap. \ref{background}. A generalized bounded ionad $\cx = (C, \text{Int})$ is a (possibly large, locally small and finitely pre-cocomplete) category $C$ together with a comonad $\text{Int} : \P(C) \to \P(C)$ preserving finite limits whose category of coalgebras is a Grothendieck topos. $\mathbb{O}$ was described by Garner in \cite{ionads}[Rem. 5.2], it maps a bounded ionad to its category of opens, that is the category of coalgebras for the interior operator.
\end{rem}

\begin{con}[$\mathsf{ST}$] \label{categorifiedscotttopology}
 The construction is based on the Scott adjunction, we map $\ca$ to the bounded ionad $(\ca, r_\ca i_\ca)$, as described in Rem. \ref{f_*}. Unfortunately, we still have to show that $\mathsf{S}(\ca)$ is coreflective in $\P(\ca)$, this will be done in the remark below. A functor $f: \ca \to \cb$ is sent to the morphism of ionads $(f,f^\sharp)$  below, where $f^\sharp$ coincides with the inverse image of $\mathsf{S}(f)$.

 \begin{center}
\begin{tikzcd}
\mathsf{S}(\cb) \arrow[dd] \arrow[rr, "f^\sharp" description, dashed] &  & \mathsf{S}(\ca) \arrow[dd] \\
                                                                      &  &                            \\
\P(\cb) \arrow[rr, "f^*" description]                                &  & \P(\ca)                  
\end{tikzcd}
 \end{center}
\end{con}

\begin{rem}[$\mathsf{S}(\ca)$ is coreflective in $\P(\ca)$]\label{coreflective} We would have liked to have a one-line-motivation of the fact that the inclusion $i_\ca: \mathsf{S}(\ca) \to \P(\ca)$ has a right adjoint $r_\ca$, unfortunately this result is true for a rather technical argument. By a general result of Kelly, $i_\ca$ has a right adjoint if and only if $\lan_{i_\ca}(1_{\mathsf{S}(\ca)})$ exists and $i_\ca$ preserves it. Since $\mathsf{S}(\ca)$ is small cocomplete, if $\lan_{i_\ca}(1_{\mathsf{S}(\ca)})$ exists, it must be pointwise and thus $i$ will preserve it because it is a cocontinuous functor. Thus it is enough to prove that $\lan_{i_\ca}(1_{\mathsf{S}(\ca)})$ exists. Anyone would be tempted to apply \citep{borceux_1994}[3.7.2], unfortunately $\mathsf{S}(\ca)$ is not a small category. In order to cut down this size issue, we use the fact that $\mathsf{S}(\ca)$ is a topos and thus have a dense generator $j: G \to \mathsf{S}(\ca)$. Now, we observe that \[\lan_{i_\ca}(1_{\mathsf{S}(\ca)}) = \lan_{i_\ca}(\lan_{j}(j))= \lan_{i_\ca \circ j} j.\] Finally, on the latter left Kan extension we can apply \citep{borceux_1994}[3.7.2], because $G$ is a small category.
\end{rem}

\begin{rem}[$\mathsf{S}$]
This was introduced in the previous chapter in detail. Coherently with the previous section, it is quite easy to notice that $\mathsf{S} \cong  \mathbb{O} \circ \mathsf{ST}$. Let us show it, \[  \mathbb{O} \circ \mathsf{ST}(\ca) = \mathbb{O} ( \mathsf{ST}(\ca)) \stackrel{\ref{f_*}}{=} \mathsf{coAlg}(r_\ca i_\ca)   \cong \mathsf{S}(\ca).  \]
\end{rem}

\subsubsection{Points} \label{categorifiedpoints}

\begin{rem}[Categorified Isbell duality and the Scott Adjunction] \label{categorifiedisbellscott}
As in the previous section, both the functors $\mathbb{O}$ and $\mathsf{S}$ have a right (bi)adjoint. We indicate them by $\mathsf{pt}$ and $\mathbbm{pt}$, which in both cases stands for points. $\mathsf{pt}$ has of course been introduced in the previous chapter and correspond to the right adjoint in the Scott adjunction. The other one will be more delicate to describe.

\begin{center}
\begin{tikzcd}
                                                                      & \text{Topoi} \arrow[lddd, "\mathbbm{pt}" description, bend left=12] \arrow[rddd, "\mathsf{pt}" description, bend left=12] &                                                                                                                 \\
                                                                      &                                                                                                                  &                                                                                                                 \\
                                                                      &                                                                                                                  &                                                                                                                 \\
\text{BIon} \arrow[ruuu, "\mathbb{O}" description, dashed, bend left=12] &                                                                                                                  & \text{Acc}_{\omega} \arrow[luuu, "\mathsf{S}" description, dashed, bend left=12] \arrow[ll, "\mathsf{ST}", dashed]
\end{tikzcd}
\end{center}
This sub-subsection will be mostly dedicated to the construction of $\mathbbm{pt}$ and to show that it is a right (bi)adjoint for $\mathbb{O}$. Let us mention though that there exists a natural functor $$ \iota: \mathsf{ST} \circ \mathsf{pt} \Rightarrow \mathbbm{pt} $$ which is not in general an equivalence of categories.
\end{rem}

\begin{con}[$\text{Topoi} \rightsquigarrow \text{Bion}$: every topos induces a generalized bounded ionad over its points]\label{categorified isbell construction of pt} \label{fromtopoitobion}
For a topos $\ce$, there exists a natural evaluation pairing $$\mathsf{ev}: \ce \times \mathsf{pt}(\ce) \to \Set,$$ mapping the couple $(e, p)$ to its evaluation $p^*(e)$. This construction preserves colimits and finite limits in the first coordinate, because $p^*$ is an inverse image functor. This means that its mate functor $ ev^*: \ce  \to \Set^{\mathsf{pt}(\ce)},$ preserves colimits and finite limits. Moreover $ev^*(e)$ preserves directed colimits for every $e \in \ce$. Indeed, \[ ev^*(e)(\colim p_i^*) \cong (\colim p_i^*)(e) \stackrel{(\star)}{\cong} \colim ((p_i^*)(e)) \cong  \colim (ev^*(e)(p_i^*)). \] where $(\star)$ is true because directed colimits of points are computed pointwise. Thus, since the category of points of a topos is always accessible (say $\lambda$-accessible) and $ev^*(e)$ preserves directed colimits, the value of $ev^*(e)$ is uniquely individuated by its restriction to $\mathsf{pt}(\ce)_\lambda$. Thus, $ev^*$ takes values in $\P(\mathsf{pt}(\ce))$.  Since a topos is a total category, $ev^*$ must have a right adjoint $ev_*$\footnote{For a total category the adjoint functor theorem reduces to check that the candidate left adjoint preserves colimits.}, and we get an adjunction, $$ev^* :\ce \leftrightarrows \P(\mathsf{pt}(\ce)): ev_*.$$ Since the left adjoint preserves finite limits, the comonad $ev^*ev_*$ is lex and thus induces a ionad over $\mathsf{pt}(\ce)$. This ionad is bounded, this follows from a careful analysis of the discussion above. Indeed $\Set^{\mathsf{pt}(\ce)_\lambda}$ is lex-coreflective $\P(\mathsf{pt}(\ce))$ and  $S \to \ce \to \Set^{\mathsf{pt}(\ce)_\lambda} \leftrightarrows \P(\mathsf{pt}(\ce))$, where $S$ is a site of presentation of $\ce$, satisfies the hypotheses of Prop. \ref{critboundedionad}(3).
\end{con}

\begin{con}[$\text{Topoi} \rightsquigarrow \text{Bion}$: every geometric morphism induces a morphism of ionads]\label{categorified isbell construction of pt} 

Observe that given a geometric morphism $f: \ce \to \cf$, $\mathsf{pt}(f): \mathsf{pt}(\ce) \to \mathsf{pt}(\cf) $ induces a morphism of ionads $(\mathsf{pt}(f), \mathsf{pt}(f)^\sharp)$ between $\mathbbm{pt}(\ce)$ and $\mathbbm{pt}(\cf)$. In order to describe $\mathsf{pt}(f)^\sharp$, we invoke Prop. \ref{morphism of ionads generator}[(a)]. Thus, it is enough to provide  a functor making the diagram below commutative (up to natural isomorphism).
\begin{center}
\begin{tikzcd}
\cf \arrow[dd, "ev^*_{\cf}" description] \arrow[rr, "f^\diamond" description, dashed] &  & \ce \arrow[dd, "ev^*_{\ce}" description] \\
                                                                             &  &                                 \\
\P({\mathsf{pt}(\cf)}) \arrow[rr, "\mathsf{pt}(f)^*" description]                                         &  & \P({\mathsf{pt}(\ce))}                         
\end{tikzcd}
\end{center}
Indeed such a functor exists and coincides with the inverse image $f^*$ of the geometric morphism $f$.
\end{con}

\begin{rem}[The $2$-functor $\mathbbm{pt}$] \label{categorified isbell defn of pt}
$\mathbbm{pt}(\ce)$ is defined to be the ionad $(\mathsf{pt}(\ce), ev^*ev_*)$, as described in the two previous remarks. 
\end{rem}

\begin{thm}[Categorified Isbell adjunction, $\mathbb{O} \dashv \mathbbm{pt} $]\label{categorified isbell adj thm}

$$ \mathbb{O}: \text{BIon} \leftrightarrows \text{Topoi}: \mathbbm{pt} $$
\end{thm}
\begin{proof}
We provide the unit and the counit of this adjunction. This means that we need to provide geometric morphisms $\rho: \mathbb{O}\mathbbm{pt}(\ce) \to \ce$ and morphisms of ionads $\lambda: \cx \to \mathbbm{pt}\mathbb{O} \cx $. Let's study the two problems separately.
\begin{itemize}
	\item[($\rho$)] As in the case of any geometric morphism, it is enough to provide the inverse image functor $\rho^*: \ce \to \mathbb{O}\mathbbm{pt}(\ce)$. Now, recall that the interior operator of $\mathbbm{pt}(\ce)$ is induced by the adjunction $ev^*: \ce \leftrightarrows \P({\mathsf{pt}(\ce)}): ev_*$ as described in the remark above. By the universal property of the category of coalgebras, the adjunction $\mathsf{U}: \mathbb{O}\mathbbm{pt}(\ce) \leftrightarrows \P({\mathsf{pt}}(\ce)): \mathsf{F}$ is terminal among those adjunctions that induce the comonad $ev^*ev_*$. This means that there exists a functor $\rho^*$ lifting $e^*$ along $\mathsf{U}$ as in the diagram below.

	\begin{center}
	\begin{tikzcd}
	\ce \arrow[rrdd, "ev^*" description] \arrow[rr, "\rho^*" description] &  & \mathbb{O}\mathbbm{pt}(\ce) \arrow[dd, "\mathsf{U}" description] \\
                                                                        &  &                                                                  \\
                                                                        &  & \P({\mathbbm{pt}(\ce)})                                        
	\end{tikzcd}
	\end{center}
	It is easy to show that $\rho^*$ is cocontinuous and preserves finite limits and is thus the inverse image functor of a geometric morphism $\rho: \mathbb{O}\mathbbm{pt}(\ce) \to \ce$ as desired.

	\item[($\lambda$)] Recall that a morphism of ionads $\lambda: \cx \to \mathbbm{pt}\mathbb{O} \cx $ is the data of a functor $\lambda: X \to \mathsf{pt}\mathbb{O}\cx$ together with a lifting $\lambda^\sharp: \mathbb{O}\cx\to \mathbb{O}\mathsf{pt}\mathbb{O}\cx.$ We only provide  $\lambda: X \to \mathsf{pt}\mathbb{O}\cx$, $\lambda^{\sharp}$ is induced by Prop. \ref{morphism of ionads generator}. Indeed such a functor is the same of a functor $\lambda: X \to \text{Cocontlex}(\mathbb{O}\cx, \Set)$. Define, $$\lambda(x)(s) = (\mathsf{U}(s))(x).$$ From a more conceptual point of view, $\lambda$ is just given by the composition of the  functors, $$X \stackrel{\mathsf{eval}}{\longrightarrow} \text{Cocontlex}(\P(X), \Set) \stackrel{- \circ \mathsf{U}}{\longrightarrow} \text{Cocontlex}(\mathbb{O}\cx, \Set).$$

	\end{itemize}
\end{proof}

 \section{Sober ionads and topoi with enough points} \label{soberenough}

In this section we show that the categorified Isbell adjunction is idempotent, providing a categorification of Subsec. \ref{topologysoberandspatial}. The notion of sober ionad is a bit unsatisfactory and lacks an intrinsic description. Topoi with enough points have been studied very much in the literature. Let us give (or recall) the two definitions.

\begin{defn}[Sober ionad]
A ionad is sober if $\lambda$ is an equivalence of ionads.
\end{defn}

\begin{defn}[Topos with enough points]
A topos has enough points if the inverse image functors from all of its points are jointly conservative.
\end{defn}

 \begin{thm}[Idempotency of the categorified Isbell duality] \label{idempotencyisbellcategorified}
 The following are equivalent:
 \begin{enumerate}
 	\item $\ce$ has enough points;
 	\item $\rho: \mathbb{O}\mathbbm{pt}(\ce) \to \ce$ is an equivalence of categories;
 	\item $\ce$ is equivalent to a topos of the form $\mathbb{O}(\cx)$ for some bounded ionad $\cx$.
 \end{enumerate}
 \end{thm}
 \begin{proof}
  \begin{enumerate}
  	\item[] 
 	\item[$(1) \Rightarrow (2)$] Going back to the definition of $\rho$ in Thm. \ref{categorified isbell adj thm}, it's enough to show that $ev^*$ is comonadic. Since it preserves finite limits, it's enough to show that it is conservative to apply Beck's (co)monadicity theorem. Yet, that is just a reformulation of having \textit{enough points}.
 	\item[$(2) \Rightarrow (3)$] Trivial.
 	\item[$(3) \Rightarrow (1)$] \cite{ionads}[Rem. 2.5].
 \end{enumerate}
 \end{proof}

\begin{thm}
 The following are equivalent:
 \begin{enumerate}
 	\item $\cx$ is sober;
 	\item $\cx$ is of the form $\mathbbm{pt}(\ce)$ for some topos $\ce$.
 \end{enumerate}
 \end{thm}
 \begin{proof}
 For any adjunction, it is enough to show that either the monad or the comonad is idempotent, to obtain the same result for the other one.
 \end{proof}

 \section{From Isbell to Scott: categorified} \label{isbelltoscottcategorified}
 
 This section is a categorification of its analog Subsec. \ref{topologyscottfromisbell} and shows how to infer results about the tightness of the Scott adjunction from the Isbell adjunction. We mentioned in Rem. \ref{categorifiedisbellscott} that there exists a natural transformation as described by the diagram below.

 \begin{center}
\begin{tikzcd}
           & \text{Topoi} \arrow[lddd, "\mathbbm{pt}" description] \arrow[rddd, "\mathsf{pt}"{name=pt}, description] &                                               \\
           &                                                                                             &                                               \\
           &                                                                                             &                                               \\
\text{BIon} \ar[Rightarrow, from=pt, shorten >=1pc, "\iota", shorten <=1pc]&                                                                                             & \text{Acc}_{\omega} \arrow[ll, "\mathsf{ST}"]
\end{tikzcd}
\end{center}

Let us describe it. Spelling out the content of the diagram, $\iota$ should be a morphism of ionads $$\iota:  \mathsf{ST} \mathsf{pt} (\ce) \to \mathbbm{pt}(\ce). $$
Recall that the underling category of these two ionads is $ \mathsf{pt} (\ce)$ in both cases. 
We define $\iota$ to be the identity on the underlying categories, $\iota = (1_{\mathbbm{pt}(\ce)}, \iota^\sharp)$ while $\iota^\sharp$ is induced by the following assignment defined on the basis of the ionad $\ce \to \mathsf{Spt}(\ce)$, $$\iota^\sharp(x)(p) = p^*(x). $$

\begin{rem}[ $\iota^\sharp$ is the counit of the Scott adjunction]
The reader might have noticed that $\iota^\sharp$ is precisely the counit of the Scott adjunction.
\end{rem}

\begin{thm}[From Isbell to Scott, cheap version] \label{cheapidempotencyscott} The following are equivalent:
\begin{enumerate}
\item $\ce$ has enough points and $\iota$ is an equivalence of ionads.
\item The counit of the Scott adjunction is an equivalence of categories.
\end{enumerate}
\end{thm}
\begin{proof}
This is completely obvious from the previous discussion.
\end{proof}

This result is quite disappointing in practice and we cannot accept it as it is. Yet, having understood that the Scott adjunction is not  the same as Isbell one was very important conceptual progress in order to guess the correct statement for the Scott adjunction. In the next section we provide a more useful version of the previous theorem. 

\begin{rem} Going back to the case of fields (Rem. \ref{fields}), we now understand why the Scott adjunction could not recover the classifying topos of the geometric theory of fields. 
\end{rem}

 \subsection{Covers}
 In order to provide a satisfying version of Thm. \ref{cheapidempotencyscott}, we need to introduce a tiny bit of technology, namely finitely accessible covers. For an accessible category with directed colimits $\ca$ a (finitely accessible) \textit{cover} $$\cl : \mathsf{Ind}(C) \to \ca $$ will be one of a class of cocontinuous (pseudo)epimorphisms (in Cat) having many good properties. They will be helpful for us in the discussion. Covers were introduced for the first time in \cite{aec}[4.5] and later used in \cite{LB2014}[2.5].

 \begin{rem}[Generating covers] 
 Every object $\ca$ in $\text{Acc}_{\omega}$ has a (proper class of) finitely accessible cover. Let $\ca$ be $\kappa$-accessible. Let us focus on the following diagram.
 \begin{center}

\begin{tikzcd}
 &  \ca_\kappa \arrow[ld, "\alpha" description] \arrow[rd, "\iota" description] &  \\
\text{Ind } \ca_\kappa \arrow[rr, "\lan_{\alpha}\iota" description, dotted, bend right=10] &  & \ca \arrow[ll, "{\ca(\iota-, -)}" description, dotted, bend right=10]
\end{tikzcd}
 \end{center}

$\ca(\iota-,-)$, also known as \textit{the nerve of f}, is fully faithful because $ \ca_\kappa$ is a (dense) generator in $\ca$. This functor does not lie in $\text{Acc}_{\omega}$ because it is just $\kappa$-accessible in general.  $\lan_{\alpha}(\iota)$ exists by the universal property of of the Ind completion, indeed $\ca$ has directed colimits by definition. For a concrete perspective $\lan_{\alpha}(\iota)$ is evaluating a formal directed colimit on the actual directed colimit in $\ca$. These two maps yield an adjunction \[\lan_{\alpha}(\iota) \dashv \ca(\iota-,-)\] that establishes $\ca$ as a reflective embedded subcategory of $\text{Ind } \ca_\kappa$. Since $\lan_{\alpha}(\iota)$ is a left adjoint, it lies in $\text{Acc}_{\omega}$.
 \end{rem}

\begin{defn}\label{cover} When $\ca$ is a $\kappa$-accessible category we will indicate with $\mathcal{L}_{\ca}^{\kappa}$ the map that here we indicated with  $\lan_{\alpha}(\iota)$ in the previous remark and we call it a \textit{cover} of $\ca$. \[\mathcal{L}_{\ca}^{\kappa} : \text{Ind } \ca_\kappa\to \ca .\]
\end{defn}
 
 \begin{notation}
 When it's evident from the context, or it is not relevant at all we will not specify the cardinal on the top, thus we will just write $\mathcal{L}_{\ca}$ instead of $\mathcal{L}^{\kappa}_{\ca}$.
 \end{notation}
 
 \begin{rem} When $\lambda \geq \kappa$, $\mathcal{L}_{\ca}^{\lambda}\cong  \mathcal{L}_{\ca}^{\kappa} \mathcal{L}_{\kappa}^{\lambda}$ for some transition map $\mathcal{L}_{\kappa}^{\lambda}$. We did not find any application for this, thus we decided not to go in the details of this construction.
 \end{rem}
  


\begin{rem} This construction appeared for the first time in \cite{aec}[4.5] and later in \cite{LB2014}[2.5], where it is presented as the analogue of Shelah's presentation theorem for AECs \cite{LB2014}[2.6]. The reader that is not familiar with formal category theory might find the original presentation more down to earth. In \cite{LB2014}[2.5] it is also shown that under interesting circumstances the cover is faithful.
\end{rem}

\subsection{On the (non) idempotency of the Scott adjunction}\label{scottidempotency}

This subsection provides a better version of Thm. \ref{cheapidempotencyscott}. It is based on a technical notion (topological embedding) that we define and study in the Toolbox chapter (see Chap. \ref{toolbox}).

\begin{thm}  \label{enoughpoints}
A Scott topos $\cg \cong \mathsf{S}(\ca)$ has enough points.
\end{thm}
\begin{proof}
It is enough to show that $\cg$ admits a geometric surjection from a presheaf topos \cite{elephant2}[2.2.12]. Let $\kappa$ be a cardinal such that $\ca$ is $\kappa$-accessible. We claim that the $1$-cell $\mathsf{S}(\mathcal{L}^{\kappa}_{\ca})$ described in Rem. \ref{cover} is the desired geometric surjection for the Scott topos $\mathsf{S}(\ca)$. By Rem. \ref{trivial}, the domain of $\mathsf{S}(\mathcal{L}^{\kappa}_{\ca})$ is indeed a presheaf topos. By Prop. \ref{surjections}, it is enough to prove that $\mathcal{L}^{\kappa}_{\ca}$ is a pseudo epimorphism in Cat. But that is obvious, because it has even a section in Cat, namely $\ca(\iota,1)$.
\end{proof}

\begin{thm} The following are equivalent. \label{scottidempotency}
\begin{enumerate} 
    \item The counit $\epsilon: \mathsf{Spt}(\ce) \to \ce$ is an equivalence of categories.
    \item  $\ce$ has enough points and for all presentations $i^*: \Set^X \leftrightarrows \ce: i_*$, $\mathsf{pt}(i)$ is a topological embedding.
   \item $\ce$ has enough points and there exists a presentation $i^*: \Set^X \leftrightarrows \ce: i_*$ such that $\mathsf{pt}(i)$ is a topological embedding.
  
\end{enumerate} 
\end{thm}
\begin{proof}
As expected, we follow the proof strategy hinted at by the enumeration.
\begin{itemize}

		\item[$1) \Rightarrow 2)$] By Thm. \ref{enoughpoints}, $\ce$ has enough points. We only need to show that for all presentations $i^*: \Set^X \leftrightarrows \ce: i_*$, $\mathsf{pt}(i)$ is a topological embedding. In order to do so, consider the following diagram,

		\begin{center}
		\begin{tikzcd}
\mathsf{Spt}\ce \arrow[dd, "\epsilon_\ce" description] \arrow[rr, "\mathsf{Spt}(i)" description] &  & \mathsf{Spt}\Set^X \arrow[dd, "\epsilon_{\Set^X}" description] \\
                                                                                                 &  &                                                                \\
\ce \arrow[rr, "i" description]                                                                  &  & \Set^X                                                        
\end{tikzcd}
		\end{center}
		By Rem. \ref{trivial} and the hypotheses of the theorem, one obtains that $\mathsf{Spt}(i)$ is naturally isomorphic to a composition of geometric embedding\[\mathsf{Spt}(i) \cong \epsilon_{\Set^X}^{-1} \ \circ \  i  \ \circ \  \epsilon_\ce,\]  and thus is a geometric embedding. This shows precisely that $\mathsf{pt}(i)$ is a topological embedding.

				\item[$2) \Rightarrow 3)$] Obvious.
	\item[$3) \Rightarrow 1)$] It is enough to prove that $\epsilon_\ce$ is both a surjection and a geometric embedding of topoi. $\epsilon_\ce$ is a surjection, indeed since $\ce$ has enough points, there exist a surjection $q: \Set^X \twoheadrightarrow \ce$, now we apply the comonad $\mathsf{S}\mathsf{pt}$ and we look at the following diagram,

		\begin{center}
	\begin{tikzcd}
\mathsf{S}\mathsf{pt}\Set^X \arrow[dd, "\mathsf{S}\mathsf{pt}(q)" description] \arrow[rr, "\epsilon_{\Set^X}" description] &  & \Set^X \arrow[dd, "q" description, two heads] \\
                                                                                                                           &  &                                               \\
\mathsf{S}\mathsf{pt}\ce \arrow[rr, "\epsilon_\ce" description]                                                            &  & \ce                                          
\end{tikzcd}
	\end{center}
	Now, the counit arrow on the top is an isomorphism, because $\Set^X$ is a presheaf topos. Thus $\epsilon_{\ce} \circ (\mathsf{S}\mathsf{pt})(q) $ is (essentially) a factorization of $q$. Since $q$ is a geometric surjection so must be $\epsilon_\ce$.  In order to show that $\epsilon_\ce$ is a geometric embedding, we use again the following diagram over the given presentation $i$.

	\begin{center}
	\begin{center}
		\begin{tikzcd}
\mathsf{Spt}\ce \arrow[dd, "\epsilon_\ce" description] \arrow[rr, "\mathsf{Spt}(i)" description] &  & \mathsf{Spt}\Set^X \arrow[dd, "\epsilon_{\Set^X}" description] \\
                                                                                                 &  &                                                                \\
\ce \arrow[rr, "i" description]                                                                  &  & \Set^X                                                        
\end{tikzcd}
		\end{center}
	\end{center}
This time we know that $\mathsf{Spt}(i)$ and $i$ are geometric embeddings, and thus $\epsilon_\ce$ has to be so.
\end{itemize}
\end{proof}

\begin{rem}
The version above might look like a technical but not very useful improvement of Thm. \ref{cheapidempotencyscott}. Instead, in the following Corollary we prove a non-trivial result based on a characterization (partial but useful) of topological embeddings contained in the Toolbox.
\end{rem}

\begin{cor} \label{enoughpointscocomplete}
Let $\ce$ be a topos with enough points, together with a presentation $i: \ce \to \Set^C$. If  $\mathsf{pt}(\ce)$ is complete and $\mathsf{pt}(i)$ preserves limits. Then $\epsilon_\ce$ is an equivalence of categories.
\end{cor}
\begin{proof}
We verify the condition (3) of the previous theorem. Since $\mathsf{pt}(i)$ preserve all limits, fully faithful, and $\mathsf{pt}(\ce)$ must be complete, $\mathsf{pt}(i)$ has a left adjoint by the adjoint functor theorem. This establishes $\mathsf{pt}\ce$ as a reflective subcategory of $\mathsf{pt}\Set^C$. By Prop. \ref{reflective}, $\mathsf{pt}(i)$ must be a topological embedding.
\end{proof}

\begin{rem}[Scott is not aways sober]
The previous corollary is coherent with Johnstone's observation that when a poset is (co)complete its Scott topology is generically sober, and thus the Scott adjunction should reduce to Isbell's duality.
\end{rem}

\begin{rem} Going back to the case of fields (Rem. \ref{fields}), we cannot indeed apply the corollary above, because the category of fields is not complete.
\end{rem}

\section{Interaction} \label{interaction}

In this section we shall convince the reader that the posetal version of the Scott-Isbell story \textit{embeds} in the categorical one.

\begin{center}
\begin{tikzcd}
                              & \text{Loc} \arrow[dd] \arrow[ld] \arrow[rr, "\mathsf{Sh}" description, dashed] &             & \text{Topoi} \arrow[dd] \arrow[ld] \\
\text{Top} \arrow[rr, dashed] &                                                                                & \text{BIon} &                                    \\
                              & \text{Pos}_\omega \arrow[lu] \arrow[rr, "\iota" description, dashed]           &             & \text{Acc}_\omega \arrow[lu]      
\end{tikzcd}
\end{center}
We have no applications for this observation, thus we do not provide all the details that would amount to an enormous amount of functors relating all the categories that we have mentioned. Yet, we show the easiest aspect of this phenomenon. Let us introduce and describe the following diagram,

\begin{center}

\begin{tikzcd}
\text{Loc} \arrow[rr, "\mathsf{Sh}" description] \arrow[dd, "\mathsf{pt}" description] &  & \text{Topoi} \arrow[dd, "\mathsf{pt}" description] \\
                                                                                       &  &                                                    \\
\text{Pos}_\omega \arrow[rr, "\text{i}" description]                                   &  & \text{Acc}_\omega                                 
\end{tikzcd}
\end{center}

\begin{rem}[$\mathsf{Sh}$ and $i$]
\begin{itemize}
	\item[]
	\item[$\mathsf{Sh}$] It is well known that the sheafification functor $\mathsf{Sh}: \text{Loc} \to \text{Topoi}$ establishes \text{Loc} as a full subcategory of $\text{Topoi}$ in a sense made precise in \cite{sheavesingeometry}[IX.5, Prop. 2 and 3].
	\item[$i$] This is very easy to describe. Indeed any poset with directed suprema is an accessible category with directed colimits and a function preserving directed suprema is precisely a functor preserving directed colimits.
	\end{itemize}
\end{rem}

\begin{prop}
The diagram above commutes.
\end{prop}
\begin{proof}
This is more or less tautological from the point of view of \cite{sheavesingeometry}[IX.5, Prop. 2 and 3]. In fact, $$\mathsf{pt}(L)=\text{Loc}(\mathbb{T}, L) \cong \text{Topoi}(\mathsf{Sh}(\mathbb{T}), \mathsf{Sh}(L)) \cong \mathsf{pt}(\mathsf{Sh}(L)).$$ 
\end{proof}

\section{The $\kappa$-case} \label{kappaIsbell}

Unsurprisingly, it is possible to generalize the whole content of this chapter to the $\kappa$-case. The notion of $\kappa$-ionad is completely straightforward and every construction lifts to infinite cardinals without any effort. For the sake of completeness, we report the $\kappa$-version of the main theorems that we saw in the chapter, but we omit the proof, which would be identical to the finitary case.

\begin{rem}
The following diagram describes the $\kappa$-version of the relevant adjunctions.
\begin{center}
\begin{tikzcd}
                                                                 & \kappa\text{-Topoi} \arrow[lddd, "\mathbbm{pt}_\kappa" description, bend left= 12] \arrow[rddd, "\mathsf{pt}_\kappa" description, bend left= 12] &                                                                                                 \\
                                                                 &                                                                                                                    &                                                                                                 \\
                                                                 &                                                                                                                    &                                                                                                 \\
\kappa\text{-BIon} \arrow[ruuu, "\mathbb{O_\kappa}" description, bend left= 12] &                                                                                                                    & \text{Acc}_{\kappa} \arrow[luuu, "\mathsf{S}_\kappa" description, bend left= 12] \arrow[ll, "\mathsf{ST}_\kappa"]
\end{tikzcd}
\end{center}
\end{rem}

 \begin{thm}[Idempotency of the categorified $\kappa$-Isbell duality]
 Let $\ce$ be a $\kappa$-topos. The following are equivalent:
 \begin{enumerate}
 	\item $\ce$ has enough $\kappa$-points;
 	\item $\rho: \mathbb{O}_\kappa\mathbbm{pt}\kappa(\ce) \to \ce$ is an equivalence of categories;
 	\item $\ce$ is of the form $\mathbb{O}_\kappa(\cx)$ for some bounded $\kappa$-ionad $\cx$.
 \end{enumerate}
 \end{thm}

\begin{thm}  Let $\ce$ be a $\kappa$-topos. The following are equivalent.
\begin{enumerate} 
    \item The counit $\epsilon: \mathsf{S}_\kappa\mathsf{pt}_\kappa(\ce) \to \ce$ is an equivalence of categories.
    \item  $\ce$ has enough $\kappa$-points and for all presentations $i^*: \Set^X \leftrightarrows \ce: i_*$, $\mathsf{pt}_\kappa(i)$ is a topological embedding.
   \item $\ce$ has enough $\kappa$-points and there exists a presentation $i^*: \Set^X \leftrightarrows \ce: i_*$ such that $\mathsf{pt}(i)$ is a topological embedding.
  
\end{enumerate} 
\end{thm}

\chapter{Logic}\label{logical}
The aim of this chapter is to give a logical account on the Scott adjunction. The reader will notice that once properly formulated, the statements of this chapter follow directly from the previous chapter. This should not be surprising and echoes the fact that once Stone-like dualities are proven, completeness-like theorems for propositional logic follow almost on the spot. 

\begin{notation}[Isbell and Scott topoi]
\begin{itemize}
	\item[]
	\item By an Isbell topos we mean a topos of the form $\mathbb{O}(\cx)$, for some bounded ionad $\cx$;
	\item By a Scott topos we mean a topos of the form $\mathsf{S}(\ca)$ for some accessible category $\ca$ with directed colimits.
\end{itemize}
\end{notation}

\begin{structure*}
The exposition is organized as follows:
\begin{enumerate}
	\item[Sec. \ref{logicgeneralizedaxiom}]  The first section will push the claim that the Scott topos $\mathsf{S}(\ca)$ is a kind of very weak notion of theory naturally attached to the accessible category which is a candidate geometric axiomatization of $\ca$. We will see how this traces back to the seminal works of Linton and Lawvere on algebraic theories and algebraic varieties.
	\item[Sec. \ref{logicclassifyingtopoi}] The second section inspects a very natural guess that might pop up in the mind of the topos theorist: \textit{is there any relation between Scott topoi} and \textit{classifying topoi}? The question will have a partially affirmative answer in the first subsection.  The second one subsumes these partial results. Indeed every theory $\cs$ has a category of models $\mathsf{Mod}(\cs)$, but this category does not retain enough information to recover the theory, even when the theory has enough points. That's why the Scott adjunction is not sharp enough. Nevertheless, every theory has a ionad of models $\mathbb{M}\mathbbm{od}(\cs)$, the category of opens of such a ionad $\mathbb{O}\mathbb{M}\mathbbm{od}(\cs)$ recovers theories with enough points.
	\item[Sec. \ref{logicaec}] This section describes the relation between the Scott adjunction and abstract elementary classes, providing a restriction of the Scott adjunction to one between accessible categories where every map is a monomorphism and locally decidable topoi.
	\item[Sec. \ref{logicsaturatedobjects}] In this section we give the definition of \textit{category of saturated objects} (CSO) and show that the Scott adjunction restricts to an adjunction between CSO and atomic topoi. This section can be understood as an attempt to conceptualize the main result in \cite{simon}.
	\end{enumerate}
\end{structure*}

\section{Generalized axiomatizations} \label{logicgeneralizedaxiom}

\begin{ach} The content of this section is substantially inspired by some private conversations with Jiří Rosický and he should be credited for it.
\end{ach}

\begin{rem}\label{groups}
Let $\mathsf{Grp}$ be the category of groups and  $\mathsf{U}: \mathsf{Grp} \to \Set$ be the forgetful functor. The historical starting point of a categorical understanding of universal algebra was precisely that one can recover the (a maximal presentation of) the algebraic theory of groups from $\mathsf{U}$. Consider all the natural transformations of the form \[\mu: \mathsf{U}^n \Rightarrow \mathsf{U}^m, \]
these can be seen as implicitly defined operations of groups. If we gather these operations in an equational theory $\mathbb{T}_\mathsf{U}$, we see that the functor $\mathsf{U}$ lifts to the category of models $\mathsf{Mod}(\mathbb{T}_\mathsf{U})$ as indicated by the diagram below.

\begin{center}
\begin{tikzcd}
\mathsf{Grp} \arrow[rdd, "\mathsf{U}" description] \arrow[r, dotted] & \mathsf{Mod}(\mathbb{T}_\mathsf{U}) \arrow[dd, "|-|" description] \\
                                                                     &                                                          \\
                                                                     & \Set                                                    
\end{tikzcd}
\end{center}
It is a quite classical result that the comparison functor above is fully faithful and essentially surjective, thus we have axiomatized the category of groups (probably with a non minimal family of operations).
\end{rem}

\begin{rem}\label{lawvereliterature}
The idea above was introduced in Lawvere's PhD thesis \cite{lawvere1963functorial} and later developed in great generality by Linton \cite{10.1007/978-3-642-99902-4_3,10.1007/BFb0083080}. The interested reader might find interesting \cite{adamekrosicky94}[Chap. 3] and the expository paper \cite{HYLAND2007437}.   Nowadays this is a standard technique in categorical logic and some generalizations of it were presented in \citep{infinitarylang} by Rosický and later again in \citep{LB2014}[Rem. 3.5].
\end{rem}

\begin{rem}[Rosický's remark] 
Rem. \ref{groups} ascertains that the collection of functor $\{ \mathsf{U}^n\}_{n \in \mathbb{N}}$, together with all the natural transformations between them, retains all the informations about the category of groups. Observe that in this specific case, the functors $ \mathsf{U}^n$ all preserve directed colimits, because finite limits commute with directed colimits. This means that this small category $\{ \mathsf{U}^n\}_{n \in \mathbb{N}}$ is a full subcategory of the Scott topos of the category of groups. In fact the vocabulary of the theory that we used to axiomatize the category of groups is made up of symbols coming from a full subcategory of the Scott topos.
\end{rem}

\begin{rem}[Lieberman-Rosický construction]
In  \citep{LB2014}[Rem. 3.5] given a couple $(\ca, \mathsf{U})$ where $\ca$ is an a accessible category with directed colimits together with a faithful functor $\mathsf{U}: \ca \to \Set$ preserving directed colimits, the authors form a category $\mathbb{U}$ whose objects are finitely accessible sub-functors of $\mathsf{U}^n$ and whose arrows are natural transformations between them. Of course there is a naturally attached signature $\Sigma_U$ and a naturally attached first order theory $\mathbb{T}_\mathsf{U}$.  In the same fashion as the previous remarks one finds a comparison functor $\ca \to \Sigma_\mathsf{U}\text{-Str}$. In \citep{LB2014}[Rem. 3.5] the authors stress that is the most natural candidate to axiomatize $\ca$. A model of $\mathbb{T}_\mathsf{U}$ is the same as a functor $\mathbb{U} \to \Set$ preserving products and subobjects. Of course the functor $\ca \to \Sigma_\mathsf{U}\text{-Str}$ factors through $\text{Mod}(\mathbb{U})$ (seen as a sketch) \[l: \ca \to \text{Mod}(\mathbb{U}),\] but in \citep{LB2014}[Rem. 3.5] this was not the main concern of the authors.
\end{rem}

\begin{rem}[Generalized axiomatizations]
The generalized axiomatization of Lieberman and Rosický amounts to a sketch $\mathbb{U}$. As we mentioned, there exists an obvious inclusion of $\mathbb{U}$ in the Scott topos of $\ca$, $$i: \mathbb{U} \to \mathsf{S}(\ca)$$ which is a flat functor because finite limits in $\mathsf{S}(\ca)$ are computed pointwise in $\Set^\ca$. Thus, every point $p: \Set \to \mathsf{S}(\ca)$ induces a model of the sketch $\mathbb{U}$ by  composition,
$$i^*: \mathsf{pt}(\mathsf{S}\ca) \to \text{Mod}(\mathbb{U})$$
$$p \mapsto  p^* \circ i.$$
In particular this shows that the unit of the Scott adjunction lifts the comparison functor between $\ca$ and $\text{Mod}(\mathbb{U})$ along $i^*$ and thus the Scott topos provides a \textit{sharper} axiomatization of $\mathbb{T}_\mathsf{U}$.

\begin{center}
\begin{tikzcd}
                                                         & \ca \arrow[ldd, "\eta_\ca" description, bend right] \arrow[rdd, "l" description, bend left] &                        \\
                                                         &                                                                                             &                        \\
\mathsf{pt}\mathsf{S}(\ca) \arrow[rr, "i^*" description] &                                                                                             & \text{Mod}(\mathbb{U})
\end{tikzcd}
\end{center}
\end{rem}

\begin{rem}[Faithful functors are likely to generate the Scott topos] Yet, it should be noticed that when $\mathbb{U}$ is a generator in $\mathsf{S}(\ca)$, the functor $i^*$ is an equivalence of categories. As unlikely as it may sound, in all the examples that we can think of, a generator of the Scott topos is always given by a faithful forgetful functor $\mathsf{U}: \ca \to \Set$. This phenomenon is so pervasive that the author has believed for quite some time that an object in the Scott topos $\mathsf{S}(\ca)$ is a generator if and only if it is faithful and conservative. We still lack a counterexample, or a theorem proving such a statement.
\end{rem}
 
\section{Classifying topoi} \label{logicclassifyingtopoi}

\begin{notation}[Isbell and Scott topoi]
\begin{itemize}
	\item[]
	\item By an Isbell topos we mean a topos of the form $\mathbb{O}(\cx)$, for some bounded ionad $\cx$;
	\item By a Scott topos we mean a topos of the form $\mathsf{S}(\ca)$ for some accessible category $\ca$ with directed colimits.
\end{itemize}
\end{notation}

This section is devoted to specifying the connection between Scott topoi, Isbell topoi and classifying topoi. Recall that for a geometric theory $\mathbb{T}$, a classifying topos $\Set[\mathbb{T}]$ is a topos representing the functor of models in topoi, \[ \mathsf{Mod}_{(-)}(\mathbb{T}) \cong \text{Topoi}(-, \Set[\mathbb{T}]).\]  The theory of classifying topoi allows us to internalize geometric logic in the internal logic of the $2$-category of topoi.

\subsection{Categories of models, Scott topoi and classifying topoi} 

The Scott topos $\mathsf{S}(\mathsf{Grp})$ of the category of groups is $\Set^{\mathsf{Grp}_{\omega}}$, this follow from Rem. \ref{trivial} and applies to $\mathsf{Mod}(\mathbb{T})$ for every Lawvere theory $\mathbb{T}$. It is well known that $\Set^{\mathsf{Grp}_{\omega}}$ is also the classifying topos of the theory of groups. This section is devoted to understating if this is just a coincidence, or if the Scott topos is actually related to the classifying topos.

\begin{rem}
Let $\ca$ be an accessible category with directed colimits. In order to properly ask the question \textit{is $\mathsf{S}(\ca)$ the classifying topos?}, we should answer the question \textit{the classifying topos of what?} Indeed $\ca$ is just a category, while one can compute classifying topoi of theories. Our strategy is to introduce a quite general notion of theory that fits in the following diagram,

\begin{center}
\begin{tikzcd}
\text{Acc}_\omega \arrow[rr, "\mathsf{S}" description, bend right=10] &                                                                                      & \text{Topoi} \arrow[ll, "\mathsf{pt}" description, bend right=10] \\
                                                   &                                                                                      &                                                \\
                                                   &                                                                                      &                                                \\
                                                   & \mathsf{Theories} \arrow[dotted, luuu, "\mathsf{Mod}(-)" description, bend left=20] \arrow[dotted, ruuu, "\gimel(-)" description, bend right=20] &                                               
\end{tikzcd}
\end{center}

in such a way that:

\begin{enumerate}
	\item $\gimel(\mathbb{T})$ gives the classifying topos of $\mathbb{T}$;
	\item $\mathsf{Mod}(-) \cong \mathsf{pt} \gimel (-)$.
\end{enumerate} 

In this new setting we can reformulate our previous discussion in the following mathematical question: \[\gimel(-) \stackrel{?}{\cong} \mathsf{S} \mathsf{Mod}(-).\]

\end{rem}

\begin{rem}[Geometric Sketches]
The notion of theory that we plan to use is that of geometric sketch. The category of (small) sketches was described in \cite{Makkaipare}[3.1], while a detailed study of geometric sketches was conducted in \cite{adamek_johnstone_makowsky_rosicky_1997,Admek1996OnGA}.

\begin{center}
\begin{tikzcd}
\text{Acc}_\omega \arrow[rr, "\mathsf{S}" description, bend right=10] &                                                                                      & \text{Topoi} \arrow[ll, "\mathsf{pt}" description, bend right=10] \\
                                                   &                                                                                      &                                                \\
                                                   &                                                                                      &                                                \\
                                                   & \mathsf{GSketches} \arrow[luuu, "\mathsf{Mod}(-)" description, bend left=20] \arrow[ruuu, "\gimel(-)" description, bend right=20] &                                               
\end{tikzcd}
\end{center}
\end{rem}

\begin{rem}
Following \cite{Makkaipare}, there exists a natural way to generate a sketch from any accessible category. This construction, in principle, gives even  a left adjoint for the functor $\mathsf{Mod}(-)$, but does land in large sketches. Thus it is indeed true that for each accessible category there exist a sketch (a theory) canonically associated to it. We do not follow this line because the notion of large sketch, from a philosophical perspective, is a bit unnatural. Syntax should always be very frugal. From an operational perspective, presentations should always be as small as possible. It is possible to cut down the size of the sketch, but this construction cannot be defined functorially on the whole category of accessible categories with directed colimits. Since elegance and naturality is one of the main motivations for this treatment of syntax-semantics dualities, we decided to avoid any kind of non-natural construction.
\end{rem}

\begin{rem}
Geometric sketches contain  coherent  sketches. In the dictionary between logic and geometry that is well motivated in the indicated papers (\cite{adamek_johnstone_makowsky_rosicky_1997,Admek1996OnGA}) these two classes correspond respectively to geometric and coherent theories. The latter essentially contain all first order theories via the process of Morleyzation. These observations make our choice of geometric sketches a very general notion of theory and makes us confident that it's a good notion to look at.
\end{rem}

 We now proceed to describe the two functors labeled with the name of $\mathsf{Mod}$ and $\gimel$.

\begin{rem}[Mod]
This 2-functor is very easy to describe. To each sketch $\cs$ we associate its category of Set-models, while it is quite evident that a morphism of sketches induces by composition a functor preserving directed colimits (see Sec. \ref{backgroundsketches} in the Background chapter).
\end{rem}

\begin{con}[$\gimel$]
The topos completion of a geometric sketch is a highly nontrivial object to describe. Among the possible constructions that appear in the literature, we refer to \citep{borceux_19943}[4.3]. Briefly, the idea behind this construction is the following. 

\begin{enumerate} 
	\item By {\citep{borceux_19943}[4.3.3]}, every sketch $\cs$ can be completed to a sketch $\bar{\cs}$ whose underlying category is cartesian.
	\item By {\citep{borceux_19943}[4.3.6]}, this construction is functorial and does not change the model of the sketch in any Grothendieck topos.
	\item By {\citep{borceux_19943}[4.3.8]}, the completion of the sketch has a natural topology $\bar{J}$.
	\item The correspondence $\cs \mapsto \bar{\cs} \mapsto (\bar{S}, \bar{J})$ transforms geometric sketches into sites and morphism of sketches into morphism of sites.
	\item We compute sheaves over the site $(\bar{S}, \bar{J})$. 
	\item Define $\gimel$ to be $\cs \mapsto \bar{\cs} \mapsto (\bar{S}, \bar{J}) \mapsto \mathsf{Sh}(\bar{S}, \bar{J})$.
\end{enumerate}
\end{con}

\begin{rem}
While \citep{borceux_19943}[4.3.6] proves that  $\mathsf{Mod}(-) \simeq \mathsf{pt} \gimel (-)$, and \citep{borceux_19943}[4.3.8] prove that $\gimel(\cs)$ is the classifying topos of $\cs$ among Grothendieck topoi, the main question of this section remains completely open, is $\gimel(\cs)$ isomorphic to the Scott topos $\mathsf{S} \mathsf{Mod}(-)$ of the category of Set models of $\cs$? We answer this question with the following theorem.
\end{rem}

\begin{thm}\label{classificatore}
If the counit $\epsilon_{\gimel(\cs)}$ of the Scott adjunction is an equivalence of categories on $\gimel(\cs)$, then $\gimel(\cs)$ coincides with $\mathsf{S} \mathsf{Mod}(\cs)$.
\end{thm}
\begin{proof}
We introduced enough technology to make this proof incredibly slick. Recall the counit \[\mathsf{S} \mathsf{pt} (\gimel(\cs)) \to \gimel(\cs) \] and assume that it is an equivalence of categories. Now, since $\mathsf{Mod}(-) \simeq \mathsf{pt} \gimel (-)$, we obtain that \[\gimel(\cs) \simeq \mathsf{S}\mathsf{Mod}(\cs),\] which indeed it our thesis.
\end{proof}

\begin{rem}
Thm. \ref{scottidempotency} characterizes those topoi for which the counit is an equivalence of categories, providing a full description of those geometric sketches for which $\gimel(\cs)$ coincides with $\mathsf{S} \mathsf{Mod}(\cs)$. Since Thm. \ref{classificatore} might not look satisfactory, in the following comment we use Cor. \ref{enoughpointscocomplete} to derive a nice looking statement.
\end{rem}

\begin{cor}
Assume $\gimel(\cs)$ has enough points and $\mathsf{Mod}(\cs)$ is complete. Let $i: \gimel(\cs) \to \Set^C$ be a presentation such that $\mathsf{pt}(i)$ preserve limits. then $\gimel(S)$ coincides with $\mathsf{S} \mathsf{Mod}(\cs)$.
\end{cor}
\begin{proof}
Apply Cor. \ref{enoughpointscocomplete} to Thm. \ref{classificatore}.
\end{proof}

\subsection{Ionads of models, Isbell topoi and classifying topoi}

Indeed the main result of this section up to this point has been partially unsatisfactory. As happens sometimes, the answer is not as nice as expected because the question in the first place did not take in consideration some relevant factors. The category of models of a sketch does not retain enough information on the sketch. Fortunately, we will show that every sketch has a ionad of models (not just a category) and the category of opens of this ionad is a much better approximation of the classifying topos.
In this subsection, we switch diagram of study to the one below.

\begin{center}
\begin{tikzcd}
\text{BIon} \arrow[rr, "\mathbb{O}" description, bend right=10] &                                                                                      & \text{Topoi} \arrow[ll, "\mathbbm{pt}" description, bend right=10] \\
                                                   &                                                                                      &                                                \\
                                                   &                                                                                      &                                                \\
                                                   & \mathsf{LGSketches} \arrow[luuu, "\mathbb{M}\mathbbm{od}(-)" description, bend left=20] \arrow[ruuu, "\gimel(-)" description, bend right=20] &                                               
\end{tikzcd}
\end{center}

Of corse, in order to study it, we need to introduce all its nodes and legs. We should say what we mean by $\mathsf{LGSketches}$ and $\mathbb{M}\mathbbm{od}(-)$. Whatever they will be, the main point of the section is to show that this diagram fixes the one of the previous section, in the sense that we will obtain the following result.

 \begin{thm*} The following are equivalent:
 \begin{itemize} 
 	\item $\gimel(\cs)$ has enough points;
 	\item $\gimel(\cs)$ coincides with $\mathbb{O}\mathbb{M}\mathbbm{od}(\cs)$.
 \end{itemize}
 \end{thm*}

 We decided to present this theorem separately from the previous one because indeed a ionad of models is a much more complex object to study than a category of models, thus the results of the previous section are indeed very interesting, because easier to handle.

\begin{exa}[Motivating ionads of models: Ultracategories]
We are not completely used to thinking about ionads of models. Indeed a (bounded) ionad is quite complex data, and we do not completely have a logical intuition on its interior operator. \textit{In which sense does the interior operator equip a category of models with a topology?} One very interesting example, that hasn't appeared in the literature to our knowledge is the case of ultracategories. Ultracategories where introduced by Makkai in \cite{AWODEY2013319} and later simplified by Lurie in \cite{lurieultracategories}. These objects are the data of a category $\ca$ together with an ultrastructure, that is a family of functors \[\int_X:\beta(X) \times \ca^X \to \ca. \] We redirect to \cite{lurieultracategories} for the precise definition. In a nutshell, each of these functors $\int_X$ defines a way to compute the ultraproduct of an $X$-indexed family of objects along some ultrafilter. Of course there is a notion of morphism of ultracategories, namely a functor $\ca \to \cb$ which is compatible with the ultrastructure  \cite{lurieultracategories}[Def. 1.41]. Since the category of sets has a natural ultrastructure, for every ultracategory $\ca$ one can define $\text{Ult}(\ca, \Set)$ which obviously sits inside $\Set^\ca$. Lurie observes that the inclusion \[\iota: \text{Ult}(\ca, \Set) \to \Set^\ca\] preserves all colimits \cite{lurieultracategories}[War. 1.4.4], and in fact also finite limits (the proof is the same). In particular, when $\ca$ is accessible and every ultrafunctor is accessible, the inclusion $\iota: \text{Ult}(\ca, \Set) \to \Set^\ca$ factors through $\P(\ca)$ and thus the ultrastructure over $\ca$ defines a idempotent lex comonad over $\P(\ca)$ by the adjoint functor theorem. This shows that every (good enough) accessible ultracategory yields a ionad, which is also \text{compact} in the sense that its category of opens is a compact (coherent) topos. This example is really a step towards a categorified Stone duality involving compact ionads and boolean topoi.
\end{exa} 

\subsubsection{$\mathsf{LGSketches}$ and $\mathbb{M}\mathbbm{od}(-)$}

\begin{defn}
A geometric sketch $\mathcal{S}$ is lex if its underlying category has finite limits and every limiting cone is in the limit class.
\end{defn}

\begin{rem}[Lex sketches are \textit{enough}]
\cite{borceux_19943}[4.3.3] shows that every geometric sketch can be replaced with a lex geometric sketch in such a way that the underlying category of models, and even the classifying topos, does not change. In this sense this full subcategory of geometric sketches is as expressive as the whole category of geometric sketches.
\end{rem}

\begin{prop}[$\mathbb{M}\mathbbm{od}(-)$ on objects]

Every lex geometric sketch $\mathcal{S}$ induces a ionad $\mathbb{M}\mathbbm{od}(\mathcal{S})$ over its category of models $\mathsf{Mod}(\mathcal{S})$.
\end{prop}
\begin{proof}
The underlying category of the ionad $\mathbb{M}\mathbbm{od}(\mathcal{S})$ is $\mathsf{Mod}(\mathcal{S})$. We must provide an interior operator (a lex comonad), $$\text{Int}_{\cs}: \P({\mathsf{Mod}(\mathcal{S})}) \to \P({\mathsf{Mod}(\mathcal{S})}). $$ In order to do so, we consider  the evaluation pairing $\mathsf{eval}: S \times \mathsf{Mod}(\cs) \to \Set$ mapping $(s,p) \mapsto p(s)$. Let $\mathsf{ev}: \cs \to \Set^{\mathsf{Mod}(S)}$ be its mate. Similarly to Con. \ref{fromtopoitobion}, such functor takes values in $\P(\mathsf{Mod}(S))$. Because $\mathcal{S}$ is a lex sketch, this functor must preserve finite limits. Indeed, \[\mathsf{ev}(\lim s_i)(-) \cong (-)(\lim s_i) \cong \lim ((-)(s_i)) \cong  \lim \mathsf{ev}(s_i)(-).\]
Now, the left Kan extension $\lan_y \mathsf{ev}$ (see diagram below) is left exact because $\P(\mathsf{Mod}(S))$ is an infinitary pretopos and $\mathsf{ev}$ preserves finite limits. 
\begin{center}
\begin{tikzcd}
S \arrow[rr, "\mathsf{ev}" description] \arrow[ddd, "y" description]               &  & \P(\mathsf{Mod}(S)) \\
                                                                                     &  &                          \\
                                                                                     &  &                          \\
\Set^{S^\circ} \arrow[rruuu, "\lan_y \mathsf{ev}" description, dashed, bend right] &  &                         
\end{tikzcd}
\end{center}
Moreover it is cocontinuous because of the universal property of the presheaf construction. Because $\Set^{S^\circ}$ is a total category,  $\lan_y \mathsf{ev}$ must have a right adjoint (and it must coincide with $\lan_{\mathsf{ev}} y$). The induced comonad must be left exact, because the left adjoint is left exact. Define \[\text{Int}_{\cs}:=\lan_y \mathsf{ev} \circ \lan_{\mathsf{ev}} y. \]
Observe that $\text{Int}_{\cs}$ coincides with the density comonad of $\mathsf{ev}$ by \cite{liberti2019codensity}[A.7]. Such result dates back to \cite{appelgate1969categories}.
\end{proof} 

\begin{rem}[$\mathbb{M}\mathbbm{od}(-)$ on morphism of sketches]
This definition will not be given explicitly: in fact we will use the following remark to show that the ionad above is isomorphic to the one induced by $\gimel(\cs)$, and thus there exists a natural way to define $\mathbb{M}\mathbbm{od}(-)$ on morphisms. 
\end{rem}
 
\subsubsection{Ionads of models and theories with enough points}

\begin{rem}
In the main result of the previous section, a relevant rôle was played by the fact that $\mathsf{pt}\gimel \simeq \mathsf{Mod}.$ The same must be true in this one. Thus we should show that $\mathbbm{pt}\gimel \simeq \mathbb{M}\mathbbm{od}.$ Indeed we only need to show that the interior operator is the same, because the underlying category is the same by the discussion in the previous section. 
\end{rem}

\begin{prop}
\[\mathbbm{pt} \circ \gimel \simeq \mathbb{M}\mathbbm{od}.\]
\end{prop}
\begin{proof}

Let $\cs$ be a lex geometric sketch. Of course there is a map $j: S \to \gimel{\cs}$, because $S$ is a site of definition of $\gimel{\cs}$. Moreover, $j$ is obviously dense. In particular the evaluation functor that defines the ionad $\mathbbm{pt} \circ \gimel$ given by $ev^*: \gimel(\cs) \to \P({\mathsf{pt} \circ \gimel(\cs)})$ is uniquely determined by its composition with $j$. This means that the comonad $ev^*ev_*$ is isomorphic to the density comonad of the composition $ev^* \circ j$. Indeed, \[ev^*ev_* \cong \lan_{ev^*} ev^* \cong \lan_{ev^*} (\lan_j(ev^*j)) \cong \lan_{ev^*j}(ev^*j).\] Yet, $ev^*j$ is evidently $\mathsf{ev}$, and thus $ev^* ev_* \cong \text{Int}_\cs$ as desired.

\end{proof}

 \begin{thm} \label{isbellclassificatore} The following are equivalent:
 \begin{itemize} 
 	\item $\gimel(\cs)$ has enough points;
 	\item $\gimel(\cs)$ coincides with $\mathbb{O}\mathbb{M}\mathbbm{od}(\cs)$.
 \end{itemize}
 \end{thm}
 \begin{proof}
 By Thm. \ref{idempotencyisbellcategorified}, $\gimel(S)$ has enough points if and only if the counit of the categorified Isbell duality   $\rho: \mathbb{O}\mathbbm{pt}(\gimel)(\cs) \to \cs$ is an equivalence of topoi. Now, since $\mathbbm{pt} \circ \gimel \cong \mathbb{M}\mathbbm{od}$, we obtain the thesis.
 \end{proof}

 \begin{rem} Going back to the case of fields (Rem. \ref{fields}), we now understand that the correct notion of semantics to study, in order to recover the full geometric theory, is the ionad of models $\mathbb{M}\mathbbm{od}(\cs)$ for a lex geometric sketch $\cs$ such that $\gimel(\cs) \simeq \cf$. Indeed, the category of opens of this ionad is $\cf$ itself, because $\cf$ has enough points.
\end{rem}
 \section{Abstract elementary classes and locally decidable topoi} \label{logicaec}

\subsection{A general discussion}
This section is dedicated to the interaction between Abstract elementary classes and the Scott adjunction. Abstract elementary classes were introduced in the 70's by Shelah as a framework to encompass infinitary logics within the language of model theorist. In principle, an abstract elementary class $\ca$ should look like the category of models of a first order infinitary theory whose morphisms are elementary embeddings. The problem of relating abstract elementary classes and accessible categories has been tackled by Lieberman \citep{L2011}, and Beke and Rosický \cite{aec}, and lately has attracted the interest of model theorists such as Vasey, Boney and Grossberg \citep{everybody}.  There are many partial, even very convincing results, in this characterization. Let us recall at least one of them. For us, this characterization will be the definition of abstract elementary class.

 \begin{thm}[\citep{aec}(5.7)]\label{AEC}  A category $\ca$ is equivalent to an abstract elementary class if and only if it is an accessible category with directed colimits, whose morphisms are monomorphisms and which admits a full with respect to isomorphisms and nearly full embedding $U$ into a finitely accessible category preserving directed colimits and monomorphisms.
 \end{thm}

 \begin{defn} 
 A functor $\mathsf{U} : \ca \to \cb$ is nearly full if, given a commutative diagram,
 \begin{center}
 \begin{tikzcd}
\mathsf{U}(a) \arrow[rrd, "\mathsf{U}(f)" description] \arrow[dd, "h" description] &  &  \\
 &  & \mathsf{U}(c) \\
\mathsf{U}(b) \arrow[rru, "\mathsf{U}(g)" description] &  & 
\end{tikzcd}
 \end{center}
 in $\cb$, there is a map $\bar{h}$ in $\ca$ such that $h = \mathsf{U}(\bar{h})$ and $g\bar{h}=f$. Observe that when $\mathsf{U}$ is faithful such a filling has to be unique.
 \end{defn}
 
 \begin{rem}
In some reference the notion of nearly-full functor was called coherent, referring directly to the \textit{coherence axiom} of AECs that it incarnates. The word coherent is overloaded in category theory, and thus we do not adopt this terminology, but nowadays it is getting more and more common. 
\end{rem}

 \begin{exa}[$\mathsf{pt}(\ce)$ is likely to be an AEC]\label{esempiobase}
 Let $\ce$ be a Grothendieck topos and $f^*: \Set^C \leftrightarrows \ce :f_*$ a presentation of $\ce$. Applying the functor $\mathsf{pt}$ we get a fully faithful functor  \[\mathsf{pt}(\ce) \stackrel{\ref{embeddings}}{\to} \mathsf{pt}(\Set^C) \stackrel{\ref{trivial}}{\cong} \mathsf{Ind}(C) \]
 into a finitely accessibly category. Thus when every map in $\mathsf{pt}(\ce)$ is a monomorphism we obtain that $\mathsf{pt}(\ce)$ is an AEC via Thm. \ref{AEC}. We will see in the next section (Thm. \ref{LDCAECs}) that this happens when $\ce$ is locally decidable; thus the category of points of a locally decidable topos is always an AEC.
 \end{exa}

 \begin{exa}[$\eta_\ca$ behaves nicely on AECs]
When $\ca$ is an abstract elementary class, the unit of the Scott adjunction $\eta_\ca: \ca \to \mathsf{pt}\mathsf{S}(\ca)$ is faithful and iso-full. This follows directly from Prop. \ref{ff}.
 \end{exa}

 \begin{rem}
 Even if this is the sharpest (available)  categorical characterization of AECs it is not hard to see how unsatisfactory it is. Among the most evident problems, one can see that it is hard to provide a categorical \textit{understanding} of nearly full and full with respect to isomorphisms. Of course, an other problem is that the list of requirements is pretty long and very hard to check: \textit{when does such a $U$ exist?}
 \end{rem}

It is very hard to understand when such a pseudo monomorphism exists. That is why it is very useful to have a testing lemma for its existence.

\begin{thm}[Testing lemma]
Let $\ca$ be an object in $\text{Acc}_{\omega}$ where every morphism is a monomorphism. If $\eta_\ca$ is a nearly-full pseudo monomorphism, then $\ca$ is an AEC.
\end{thm}
\begin{proof}
The proof is relatively easy, choose a presentation $f^*: \Set^C \leftrightarrows \mathsf{S}(\ca): f_*$ of $\mathsf{S}(\ca)$. Now in

\[\ca \stackrel{\eta_\ca}{\to} \textsf{pt}\textsf{S}(\ca) \stackrel{\ref{embeddings}}{\to} \textsf{pt}(\Set^{C}) \stackrel{\ref{trivial}}{\cong} \mathsf{Ind}(C),\]

the composition is a faithful and nearly full functor preserving directed colimits from an accessible category to a finitely accessible category, and thus $\ca$ is an AEC because of Thm. \ref{AEC}.
\end{proof}

\subsection{Locally decidable topoi and AECs}

The main result of this subsection relates locally decidable topoi to AECs. The full subcategory of $\text{Acc}_\omega$ whose objects are AECs will be indicated by $\text{AECs}$. As in the previous chapters, let us give the precise statement and then discuss it in better detail.

\begin{thm}\label{LDCAECs} The Scott adjunction restricts to locally decidable topoi and AECs.

\[\mathsf{S}: \text{AECs} \leftrightarrows \text{LDTopoi}: \mathsf{pt}\]
\end{thm}

\subsubsection{Locally decidable topoi}

The definition of locally decidable topos will appear obscure at first sight.

\begin{defn}[Decidable object]
An object $e$ in a topos $\ce$ is decidable if the diagonal map $e \to e \times e$ is a complemented subobject.
\end{defn}

\begin{defn}[Locally decidable topos]
An object $e$ in a topos $\ce$ is called locally decidable iff there is an epimorphism $e' \twoheadrightarrow e$ such that $e'$ is a decidable object.  $\ce$ is locally decidable if every object is locally decidable.
\end{defn}

In order to make the definition above clear we should really define decidable objects and discuss their meaning. This is carried out in the literature and it is not our intention to recall the whole theory of locally decidable topoi. Let us instead give the following characterization, that we may take as a definition.

\begin{thm}[{\citep{elephant2}[C5.4.4]}, Characterization of loc. dec. topoi] The following are equivalent:
\begin{enumerate}
	\item $\ce$ is locally decidable;
	\item there exists a site $(C,J)$ of presentation where every map is epic;
	\item there exists a localic geometric morphism into a Boolean topos.
\end{enumerate}
\end{thm}

\begin{rem}
Recall that a localic topos $\ce$ is a topos of sheaves over a locale. The theorem above (which is due to Freyd \cite{Aspects}) shows that a locally decidable topos is still a topos of sheaves over a locale, but the locale is not in $\Set$. It is instead in some boolean topos. A boolean topos is the closest kind of topos we can think of to the category of sets itself. For more details, we redirect the reader to the Background chapter, where we give references to the literature.
\end{rem}

\subsubsection{Proof of Thm. \ref{LDCAECs}}

\begin{proof}[Proof of Thm. \ref{LDCAECs}]
\begin{itemize}
	\item[]
	\item Let $\ce$ be a locally decidable topos. By Exa. \ref{esempiobase}, it is enough to show that every map in $\mathsf{pt}(\ce)$ is a monomorphism. This is more or less a folklore result, let us give the shortest path to it given our technology. Recall that one of the possible characterization of a locally decidable topos is that it has a localic geometric morphism into a boolean topos $\ce \to \cb$. If $\cb$ is a boolean topos, then every map in $\mathsf{pt}(\cg)$ is a monomorphism \cite{elephant2}[D1.2.10, last paragraph]. Now, the induce morphism below,

\[ \textsf{pt}(\ce) \to \textsf{pt}(\cb),\]
is faithful by Prop. \ref{localicmaps}. Thus every map in $\textsf{pt}\ce$ must be a monomorphism.
	\item Let's show that for an accessible category with directed colimits $\ca$, its Scott topos is locally decidable. By \citep{elephant2}[C5.4.4], it's enough to prove that $\mathsf{S}\ca$ has a site where every map is an epimorphism. Using Rem. \ref{generatorscotttopos}, $\ca_\kappa^{\circ}$ is a site of definition of $\mathsf{S}\ca$, and since every map in $\ca$ is a monomorphism, every map in $\ca_\kappa^{\circ}$ is epic.
	\end{itemize}
\end{proof}

The previous theorem admits an even sharper version.

\begin{thm} Let $\ca$ be an accessible category with directed colimits and a faithful functor $\mathsf{U}: \ca \to \Set$ preserving directed colimits. If $\mathsf{S}\ca$ is locally decidable, then  every map in $\ca$ is a monomorphism.
\end{thm}
\begin{proof}

\begin{enumerate}
	\item[]
\item[Step 1]  If $\cg$ is a boolean topos, then every map in $\mathsf{pt}(\cg)$ is a monomorphism \cite{elephant2}[D1.2.10, last paragraph].
\item[Step 2] Recall that one of the possible characterization of a locally decidable topos is that it has a localic geometric morphism into a boolean topos $\mathsf{S}(\ca) \to \cg$.
\item[Step 3]  In the following diagram
\[\ca \stackrel{\eta_\ca}{\to} \textsf{pt}\textsf{S}(\ca) \stackrel{\ref{localicmaps}}{\to} \textsf{pt}(\cg),\]
the composition is a faithful functor by Prop. \ref{ff}. Thus $\ca$ has a faithful functor into a category where every map is a monomorphism. As a result every map in $\ca$ is a monomorphism.
\end{enumerate}
\end{proof}

\begin{rem}The following corollary gives a complete characterization of those continuous categories that are abstract elementary classes. Recall that continuous categories were defined in \cite{cont} in analogy with continuous posets in order to study exponentiable topoi. Among the possible characterizations, a category is continuous if and only if it is a reflective subcategory of a finitely accessible category whose right adjoint preserve directed colimits. We discussed continuous categories in the Promenade \ref{trivial} and \ref{trivial1}.
\end{rem}
 
\begin{cor}[Continuous categories and AECs]
Let $\ca$ be a continuous category. The following are equivalent:
\begin{enumerate}
	\item $\ca$ is an AEC.
	\item Every map in $\ca$ is a monomorphism.
	\item $\mathsf{S}(\ca)$ is locally decidable.
	\end{enumerate}
\end{cor}
\begin{proof}
Since it's a split subobject in $\text{Acc}_{\omega}$ of a finitely accessible category, the hypotheses of \citep{aec}[5.7] are met.
\end{proof}

\section{Categories of saturated objects, atomicity and categoricity} \label{logicsaturatedobjects}

\begin{rem}
In this section we define categories of saturated objects and study their connection with atomic topoi and categoricity. The connection between atomic topoi and categoricity was pointed out in \citep{Caramelloatomic}. This section corresponds to a kind of syntax-free counterpart of \citep{Caramelloatomic}. In the definition of \textit{category of saturated objects} we axiomatize the relevant properties of the inclusion $\iota: \Set_{\kappa} \to \Set$ and we prove the following two theorems.
\begin{thm*}
\begin{enumerate}
	\item[]
	\item If $\ca$ is a category of saturated objects, then $\mathsf{S}(\ca)$ is an atomic topos.
	\item If in addition $\ca$ has  the joint embedding property, then $\mathsf{S}(\ca)$ is boolean and two valued.
	\item If in addition $\eta_\ca$ is isofull and faithful and surjective on objects, then $\ca$ is categorical in some presentability rank.
\end{enumerate}
\end{thm*} 

\begin{thm*} 
If $\ce$ is an atomic topos, then $\mathsf{pt}(\ce)$ is a \textit{candidate} category of saturated objects.
\end{thm*} 
\end{rem}


Let us recall (or introduce) the notion of $\omega$-saturated object in an accessible category and the joint embedding property.

\begin{defn}
Let $\ca$ be an accessible category. We say that $s \in \ca$ is $\omega$-saturated if it is injective with respect to maps between finitely presentable objects. That is, given a morphism between finitely presentable objects $f: p \to p'$ and a map $p \to s$, there exists a lift as in the diagram below.
\begin{center}
\begin{tikzcd}
s                     &                       \\
p \arrow[u] \arrow[r] & p' \arrow[lu, dashed]
\end{tikzcd}
\end{center}
\end{defn}

\begin{rem}
In general, when we look at accessible categories from the perspective of model theory, every map in $\ca$ is a monomorphism, and this definition is implicitly adding the hypothesis that every morphism is \textit{injective}.
\end{rem}

\begin{rem}
A very good paper to understand the categorical approach to saturation is \cite{Rsaturated}.
\end{rem}

\begin{defn}
Let $\ca$ be a category. We say that $\ca$ has the joint embedding property if given two objects $A,B$ there exist and object $C$ and two morphisms $A \to C$, $B \to C$.
\end{defn}

\begin{rem}
In \citep{simon}, Henry proves that there are AECs that cannot appear as the category of points of a topos, which means that they cannot be axiomatized in $\text{L}_{\infty, \omega}$. This answers a question initially asked by Rosický at the conference Category Theory 2014 and makes a step towards our understanding of the connection between accessible categories with directed colimits and axiomatizable classes.
The main tool that allows him to achieve this result is called in the paper the \textit{Scott construction}; he proves the Scott topos of $\Set_{\geq \kappa}$\footnote{The category of sets of cardinality at least $\kappa$ and injective functions} is atomic. Even if we developed together the relevant rudiments of the Scott construction, the reason for which this result was true appeared to the author of this thesis enigmatic and mysterious. With this motivation in mind we\footnote{The author of this thesis.} came to the conclusion that  the Scott topos of $\Set_{\geq \kappa}$ is atomic because of the fact that $\Set_{\geq \kappa}$ appears as a subcategory of saturated objects in $\Set$. 
\end{rem}

\begin{rem}
As a direct corollary of the theorems in this section one gets back the main result of \citep{simon}, but this is not the main accomplishment of this section. Our main contribution is to present a conceptual understanding of \citep{simon} and a neat technical simplification of his proofs. We also improve our poor knowledge of the Scott adjunction, trying to collect and underline its main features. We feel that the Scott adjunction might serve as a tool to have a categorical understanding of Shelah's categoricity conjecture for accessible categories with directed colimits.
\end{rem}

\begin{rem}[What is categoricity and what about the categoricity conjecture?]
Recall that a category of models of some theory is categorical in some cardinality $\kappa$ if it has precisely one model of cardinality $\kappa$. Morley has shown in 1965 that if a category of models is categorical in some cardinal $\kappa$, then it must be categorical in any cardinal above and in any cardinal below up to $\omega_1$ (\cite{chang1990model}). We will be more precise about Morley's result in the section about open problems. When Abstract elementary classes were introduced in the 1970's, Shelah chose Morley's theorem as a sanity check result for his definition. Since then, many approximations of these results has appeared in the literature. The most updated to our knowledge is contained in \cite{vasey2019categoricity}. We recommend the paper also as an introduction to this topic.
\end{rem}

\begin{defn}[(Candidate) categories of ($\omega$-)saturated objects] Let $\ca$ be a category in $\text{Acc}_{\omega}$. We say that $\ca$ is a category of (finitely) saturated objects if there a is topological embedding $j: \ca \to \ck$ in $\text{Acc}_{\omega}$ such that:
\begin{enumerate}
\item $\ck$ is a finitely accessible category.
\item $j\ca \subset \text{Sat}_\omega(\ck)$\footnote{The full subcategory of $\omega$-saturated objects.}.
\item $\ck_{\omega}$ has the amalgamation property\footnote{A category has the amalgamation property is every span can be completed to a square.}.
\end{enumerate}
We say that $\ca$ is a candidate category of (finitely) saturated objects if there exists a functor $j$ that verifies $(1)$-$(3)$.
\end{defn}

\begin{rem}
The notion of \textit{category of saturated objects} axiomatizes the properties of the inclusion $j\Sat_\omega(\ck) \hookrightarrow \ck$, our motivating example was the inclusion of $\Set_{\geq \kappa} \hookrightarrow \Set_{\geq \omega} \hookrightarrow \Set$. The fact that every object in $\Set_{\geq \kappa}$ is injective with respect to finite sets is essentially the axiom of choice. \citep{Rsaturated} describes a direct connection between saturation and amalgamation property, which was also implied in \citep{Caramelloatomic}.
\end{rem}

In \citep{Caramelloatomic}, Caramello proves - essentially - that the category of points of an atomic topos is a category of saturated objects and she observes that it is countable categorical. This shows that there is a deep connection between categoricity, saturation and atomic topoi. We recall the last notion before going on with the exposition.

\begin{defn}[Characterization of atomic topoi, {\cite{elephant2}[C3.5]}] Let $\cg$ be a Grothendieck topos, then the following are equivalent:
\begin{enumerate}
\item $\cg$ is atomic.
\item $\cg$ is the category of sheaves over an atomic site.
\item The subobject lattice of every object is a complete atomic boolean algebra.
\item Every object can be written as a disjoint union of atoms.
\end{enumerate}
\end{defn}

\begin{thm} \label{thmcategoriesofsaturated objects}
\begin{enumerate}
	\item[]
	\item If $\ca$ is a category of saturated objects, then $\mathsf{S}(\ca)$ is an atomic topos.
	\item If in addition $\ca$ has  the joint embedding property, then $\mathsf{S}(\ca)$ is boolean and two valued.
	\item If in addition $\eta_\ca$ is iso-full, faithful and surjective on objects, then $\ca$ is categorical in some presentability rank.
\end{enumerate}
\end{thm} 
\begin{proof}

\begin{enumerate}
	\item[]
	\item Let $\ca$ be a category of saturated objects $j: \ca \to \ck$. We must show that $\mathsf{S}(\ca)$ is atomic. The idea of the proof is  very simple; we will show that:
		\begin{itemize}
		\item[(a)] $\mathsf{S}j$  presents $\ca$ as  $j^*: \Set^{\ck_\omega} \leftrightarrows \mathsf{S}(\ca):  j_*$;
		\item[(b)] The induced topology on $\ck_\omega$ is atomic.
		\end{itemize}
	(a) follows directly from the definition of topological embedding and Rem. \ref{trivial}. (b) goes identically to \citep{simon}[Cor. 4.9]: note that for any map $k  \to k' \in \ck_{\omega}$, the induced map $j^*yk \to j^*yk'$ is an epimorphism: indeed any map $k \to ja$ with $a \in \ca$ can be extended along $k \to k'$ because $j$ makes $\ca$ a category of saturated objects. So the induced topology on $\ck_{\omega}$ is the atomic topology (every non-empty sieve is a cover). The fact that $\ck_{\omega}$ has the amalgamation property is needed to make the atomic topology a proper topology.
	\item Because $\ca$ has the joint embedding property, its Scott topos is connected by Cor. \ref{JEPconnected}. Then, $\mathsf{S}(\ca)$ is atomic and connected. By \cite{caramello2018theories}[4.2.17] it is boolean two-valued.
	\item This follows from Prop. \ref{ff} and \cite{Caramelloatomic}. In fact, Caramello has shown that $\mathsf{ptS}(\ca)$ must be countably categorical and the countable object is saturated (by construction). Thus, the unit of the Scott adjunction must reflect the (essential) unicity of such an object.
\end{enumerate}
\end{proof}

\begin{thm} 
If $\ce$ is an atomic topos, then $\mathsf{pt}(\ce)$ is a \textit{candidate} category of saturated objects.
\end{thm}
\begin{proof}
Let $\ce$ be an atomic topos and $i: \ce \to \Set^C$ be a presentation of $\ce$ by an atomic site. It follows from \cite{Caramelloatomic} that $\mathsf{pt}(i)$ presents $\mathsf{pt}(\ce)$ as a candidate category of saturated objects.
\end{proof}

\subsection{Categories of $\kappa$-saturated objects}

Obviously the previous definitions can be generalized to the $\kappa$-case of the Scott adjunction, obtaining analogous results. Let us boldly state them.

\begin{defn}[(Candidate) categories of ($\kappa$-)saturated objects] Let $\ca$ be a category in $\text{Acc}_{\kappa}$. We say that $\ca$ is a category of $\kappa$-saturated objects if there is topological embedding (for the $\mathsf{S}_\kappa$-adjunction) $j: \ca \to \ck$ in $\text{Acc}_{\kappa}$ such that:
\begin{enumerate}
\item $\ck$ is a $\kappa$-accessible category.
\item $j\ca \subset \text{Sat}_\kappa(\ck)$.
\item $\ck_{\kappa}$ has the amalgamation property.
\end{enumerate}
We say that $\ca$ is a candidate category of $\kappa$-saturated objects if there exists a functor $j$ that verifies $(1)$-$(3)$.
\end{defn}

\begin{thm} \label{kappasaturated}
\begin{enumerate}
	\item[]
	\item If $\ca$ is a category of $\kappa$-saturated objects, then $\mathsf{S}_\kappa(\ca)$ is an atomic $\kappa$-topos.
	\item If in addition $\ca$ has  the joint embedding property, then $\mathsf{S}_\kappa(\ca)$ is boolean and two valued.
	\item If in addition $\eta_\ca$\footnote{the unit of the $\kappa$-Scott adjunction.} is iso-full, faithful and surjective on objects, then $\ca$ is categorical in some presentability rank.
\end{enumerate}
\end{thm} 

\begin{thm} 
If $\ce$ is an atomic $\kappa$-topos, then $\mathsf{pt}_\kappa(\ce)$ is a \textit{candidate} category of $\kappa$-saturated objects.
\end{thm}


\chapter{Category theory}

This chapter is dedicated to a $2$-categorical perspective on the Scott adjunction and its main characters. We provide an overview of the categorical properties of $\text{Acc}_\omega$ and Topoi. Mainly, we show that the $2$-category of topoi is enriched over $\text{Acc}_\omega$ and has copowers. We show that this observation generalizes the Scott adjunction in a precise sense. We discuss the $2$-categorical properties of both the $2$-categories, but this work is not original. We will provide references within the discussion.

\section{$2$-categorical properties of $\text{Acc}_\omega$}
\subsection{(co)Limits in $\text{Acc}_\omega$} \label{colimitsinacc}

The literature contains a variety of results on the $2$-dimensional structure of the $2$-category $\text{Acc}$ of accessible categories and accessible functors. Among these, one should mention \citep{Makkaipare} for lax and pseudo-limits in $\text{Acc}$ and \citep{colimitacc} for colimits. Our main object of study, namely $\text{Acc}_\omega$, is a (non-full) subcategory of $\text{Acc}$, and thus it is a bit tricky to infer its properties from the existing literature. Most of the work was successfully accomplished in \cite{lieberman2015limits}. Let us list the main results of these references that are related to $\text{Acc}_\omega$.

\begin{prop}[{\cite{lieberman2015limits}[2.2]}]
$\text{Acc}_\omega$ is closed under pie-limits\footnote{These are those limits can be reduced to products, inserters and equifiers.} in $\text{Acc}$ (and thus in the illegitimate $2$-category of locally small categories).
\end{prop}

\begin{prop}[Slight refinement of {\cite{colimitacc}[2.1]}]
Every directed diagram of accessible categories and full embeddings preserving directed colimits has colimit in Cat, and is in fact the colimit in $\text{Acc}_\omega$.
\end{prop}

\subsection{$\text{Acc}_\omega$ is monoidal closed}

This subsection discusses a monoidal closed structure on $\text{Acc}_\omega$. The reader should keep in mind the monoidal product of modules over a ring, because the construction is similar in spirit, at the end of the subsection we will provide an hopefully convincing argument in order to show that the construction is similar for a quite quantitative reason. The main result of the section should be seen as a slight variation of \cite{kelly1982basic}[6.5] where the enrichment base is obviously the category of Sets and $\mathcal{F}$-cocontinuity is replaces by preservation of directed colimits. Our result doesn't technically follows from Kelly's one because of size issues, but the general idea of the proof is in that spirit. Moreover, we found it clearer to provide an explicit construction of the tensor product in our specific case. The reader is encouraged to check \cite{350869}, where Brandenburg provides a concise presentations of Kelly's construction. For a treatment of how the bilinear tensor product on categories with certain colimits gives you a monoidal bicategory we refer to \cite{bourke2017skew,HYLAND2002141,LOPEZFRANCO20112557}. 

\begin{rem}[A natural internal hom] \label{internalhom} Given two accessible categories $\ca, \cb $ in $\text{Acc}_\omega$, the category of functors preserving directed colimits $\text{Acc}_\omega(\ca, \cb)$ has directed colimits and they are computed pointwise. Moreover it is easy to show that it is sketchable and thus accessible. Indeed $\text{Acc}_\omega(\ca, \cb)$ is accessibly embedded in $\cb^{\ca_\lambda}$ and coincides with those functors $\ca_\lambda \to \cb$ preserving $\lambda$-small directed colimits, which makes it clearly sketchable. Thus we obtain a $2$-functor, \[[-,-]:\text{Acc}_\omega^\circ \times \text{Acc}_\omega \to \text{Acc}_\omega.\] In our analogy, this corresponds to the fact that the set of morphisms between two modules over a ring $\mathsf{Mod}(M,N)$ has a (pointwise) structure of module.
\end{rem}

\begin{rem}[Looking for a tensor product: the universal property] \label{univproptensor} 
Assume for a moment that the tensorial structure that we are looking for exists, then we would obtain a family of (pseudo)natural equivalences of categories, $$\text{Acc}_\omega(\ca \otimes \cb, \cc) \simeq \text{Acc}_\omega (\ca, [\cb, \cc]) \simeq \omega\text{-Bicocont}(\ca \times \cb, \cc).$$ In the display we wrote $\omega\text{-Bicocont}(\ca \times \cb, \cc)$ to mean the category of those functors that preserve directed colimits in each variable. The equation gives us the universal property that should define $\ca \otimes \cb$ up to equivalence of categories and is consistent with our ongoing analogy of modules over a ring, indeed the tensor product classifies \textit{bilinear maps}. 
\end{rem}

\begin{con}[Looking for a tensor product: the construction] \label{constructiontensor} 
Let $\yo: \ca \times \cb \to \cp(\ca \times \cb)$ be the Yoneda embedding of $\ca \times \cb$ corestricted to the full subcategory of small presheaves \cite{presheaves}. Let $B(\ca,\cb)$ be the full subcategory of $\cp(\ca \times \cb)$ of those functors that preserve cofiltered limits in both variables\footnote{i.e., send filtered colimits in $\ca$ or $\cb$ to limits in $\Set$.}. It is easy to show that $B(\ca,\cb)$ is sketched by a limit theory, and thus is locally presentable. The inclusion $i: B(\ca,\cb) \hookrightarrow \cp(\ca \times \cb)$ defines a small-orthogonality class\footnote{Here we are using that $\ca$ and $\cb$ are accessible in order to cut down the size of the orthogonality.} in $\cp(\ca \times \cb)$ and is thus reflective \cite{adamekrosicky94}[1.37, 1.38]. Let $L$ be the left adjoint of the inclusion, as a result we obtain an adjunction, $$L: \cp(\ca \times \cb) \leftrightarrows B(\ca,\cb): i. $$
Now define $\ca \otimes \cb$ to be the smallest full subcategory of $B(\ca,\cb)$ closed under directed colimits and containing the image of $L \circ \yo$. Thm. \cite{kelly1982basic}[6.23] in Kelly ensures that $\ca \otimes \cb$ has the universal property described in Rem. \ref{univproptensor} and thus is our tensor product. It might be a bit hard to see but this construction still follows our analogy, the tensor product of two modules is indeed built from free module on the product and  the \textit{bilinear} relations.
\end{con}

\begin{thm}\label{accmonoidalclosed}
$\text{Acc}_\omega$, together with the tensor product $\otimes$ defined in Con. \ref{constructiontensor} and the internal hom defined in Rem. \ref{internalhom} is a monoidal biclosed bicategory\footnote{This is not strictly true, because the definition of monoidal closed category does not allow for equivalence of categories. We did not find a precise terminology in the literature and we felt non-useful to introduce a new concept for such a small discrepancy.} in the sense that there is are pseudo-equivalences of categories $$\text{Acc}_\omega(\ca \otimes \cb, \cc) \simeq \text{Acc}_\omega (\ca, [\cb, \cc]),$$ which are natural in $\cc$.
\end{thm}
\begin{proof}
Follows directly from the discussion above.
\end{proof}

\begin{rem}[Up to iso/up to equivalence]
As in \cite{kelly1982basic}[6.5], we will not distinguish between the properties of this monoidal structure (where everything is true up to equivalence of categories) and a usual one, where everything is true up to isomorphism. In our study this distinction never plays a rôle, thus we will use the usual terminology about monoidal structures. 
\end{rem}

\begin{rem}[The unit] \label{unitmonoidalstructure}
The unit of the above-mentioned monoidal structure is the terminal category in Cat, which is also terminal in $\text{Acc}_\omega$.
\end{rem}

\begin{rem}[Looking for a tensor product: an abstract overview] \label{tensorabstract} 
In this subsection we have used the case of modules over a ring as a kind of analogy/motivating example. In this remark we shall convince the reader that the analogy can be pushed much further. Let's start by the observation that $R\text{-}\mathsf{Mod}$ is the category of algebras for the monad $R[-]: \Set \to \Set$. The monoidal closed structure of $\mathsf{Mod}$ can be recovered from the one of the category of sets $(\Set, 1, \times, [-,-])$ via a classical theorem proved by Kock in the seventies.  It would not make the tractation more readable to cite all the papers that are involved in this story, thus we mention the PhD thesis of Brandenburg \cite{brandenburg2014tensor}[Chap. 6] which provides a very coherent and elegant presentation of the literature. 

\begin{thm}(Seal, 6.5.1 in \cite{brandenburg2014tensor}) Let $T$ be a coherent (symmetric) monoidal monad on a (symmetric) monoidal category C with reflexive coequalizers. Then $\mathsf{Mod}(T)$ becomes a (symmetric) monoidal category.
\end{thm}

Now similarly to $\mathsf{Mod}(R)$ the $2$-category of categories with directed colimits and functors preserving them is the category of (pseudo)algebras for the KZ monad of the Ind-completion over locally small categories $$\mathsf{Ind}: \text{Cat} \to \text{Cat}. $$ \cite{bourke2017skew}[6.7] provides a version of Seal's theorem for monads over Cat. While it's quite easy to show that the completion under directed colimits meets many of Bourke's hypotheses, we do not believe that it meets all of them, thus we did not manage to apply a Kock-like result to derive Thm. \ref{accmonoidalclosed}. Yet, we think we have provided enough evidence that the analogy is not just motivational.
\end{rem}

\section{$2$-categorical properties of $\text{Topoi}$}
\subsection{(co)Limits in $\text{Topoi}$}

The $2$-categorical properties of the category of topoi have been studied in detail in the literature. We mention \citep{elephant2}[B3.4] and \citep{Lurie} as a main reference.

\subsection{Enrichment over $\text{Acc}_\omega$, tensor and powers} \label{topoienriched}

\subsubsection{The enrichment}
The main content of this sub-subsection will be to show that the category of topoi and geometric morphisms (in the direction of the right adjoint) is enriched over $\text{Acc}_\omega$. Notice that formally speaking, we are enriching over a monoidal bicategory, thus the usual theory of enrichment is not sufficient. As Garner pointed out to us, the theory of bicategories enriched in a monoidal bicategory is originally due to Bozapalides in the 1970s, though he was working without the appropriate technical notion; more precise definitions are in the PhD theses of Camordy and Lack; and everything is worked out in excruciating detail in \cite{garner2016enriched}.

\begin{rem} Recall that to provide such an enrichment means to

\begin{enumerate}
	\item show that given two topoi $\ce, \cf$, the set of geometric morphisms $\text{Topoi}(\ce, \cf)$ admits a structure of accessible category with directed colimits.
	\item provide, for each triple of topoi $\ce, \cf, \cg$, a functor preserving directed colimits $$\circ: \text{Topoi}(\ce, \cf) \otimes \text{Topoi}(\cf, \cg) \to  \text{Topoi}(\ce, \cg), $$ making the relevant diagrams commute.
\end{enumerate} 
(1) will be shown in Prop. \ref{topoihomacc}, while (2) will be shown in Prop. \ref{topoicomposition}.
\end{rem}

\begin{prop}\label{topoihomacc}
 Let $\ce,\cf$ be two topoi. Then the category of geometric morphisms $\text{Cocontlex}(\ce, \cf)$, whose objects are cocontinuous left exact functors and morphisms are natural transformations is an accessible category with directed colimits.
\end{prop}

\begin{proof}
The proof goes as in \cite{borceux_19943}[Cor.4.3.2], $\Set$ plays no rôle in the proof. What matters is that finite limits commute with directed colimits in a topos.
\end{proof}

\begin{prop}\label{topoicomposition}
For each triple of topoi $\ce, \cf, \cg$, there exists a functor preserving directed colimits $$\circ: \text{Topoi}(\ce, \cf) \otimes \text{Topoi}(\cf, \cg) \to  \text{Topoi}(\ce, \cg), $$ making the relevant diagrams commute.
\end{prop}
\begin{proof}
We will only provide the composition. The relevant diagrams commute trivially from the presentation of the composition. Recall that by \ref{univproptensor} a map of the form $\circ: \text{Topoi}(\ce, \cf) \otimes \text{Topoi}(\cf, \cg) \to  \text{Topoi}(\ce, \cg), $ preserving directed colimits is the same of a functor $$ \circ: \text{Topoi}(\ce, \cf) \times \text{Topoi}(\cf, \cg) \to  \text{Topoi}(\ce, \cg) $$ preserving directed colimits in each variables. Obviously, since left adjoints can be composed in Cat, we already have such a composition. It's enough to show that it preserves directed colimits in each variable. Indeed this is the case, because directed colimits in these categories are computed pointwise.
\end{proof}

\begin{thm} \label{enrichment}
The category of topoi is enriched over $\text{Acc}_\omega$.
\end{thm}
\begin{proof}
Trivial from the previous discussion.
\end{proof}

\subsection{Tensors} \label{tensored}
In this subsection we show that the $2$-category of topoi has tensors (copowers) with respect to the enrichment of the previous section.

\begin{rem}
Let us recall what means to have tensors for a category $\mathsf{K}$ enriched over sets (that is, just a locally small category). To have tensors in this case means that we can define a functor $\boxtimes: \Set \times \mathsf{K} \to \mathsf{K}$ in such a way that, $$\mathsf{K}(S \boxtimes k, h) \cong \Set(S,\mathsf{K}( k, h)). $$ For example, the category of modules over a ring has tensors given by the formula $S \boxtimes M := \oplus_{S}M$; indeed it is straightforward to observe that $$R\text{-}\mathsf{Mod}(\oplus_S M, N) \cong \Set(S,R\text{-}\mathsf{Mod}(M, N)). $$ In this case, this follows from the universal property of the coproduct.
\end{rem}

\begin{rem}[The construction of tensors] \label{boxtimesdef}
We shall define a $2$-functor $\boxtimes: \text{Acc}_\omega \times \text{Topoi} \to \text{Topoi}.$ Our construction is reminiscent of the Scott adjunction, and we will see that there is an extremely tight connection between the two. Given a topos $\ce$ and an accessible category with directed colimits $\ca$ we define, $$\ca \boxtimes \ce := \text{Acc}_\omega(\ca, \ce).$$ In order to make this construction meaningful we need to accomplish two tasks:

\begin{enumerate}
	\item show that the construction is well defined (on the level of objects), that is, show that $\text{Acc}_\omega(\ca, \ce)$ is a topos.
	\item describe the action of $\boxtimes$ on functors.
\end{enumerate}

We split these two tasks into two different remarks.
\end{rem}

\begin{rem}[$\text{Acc}_\omega(\ca, \ce)$ is a topos] \label{tensorscottconstructionwelldefined}

By definition $\ca$ must be $\lambda$-accessible for some $\lambda$. Obviously $\text{Acc}_\omega(\ca, \ce)$ sits inside $\lambda\text{-Acc}(\ca, \ce)$. Recall that $\lambda\text{-Acc}(\ca, \ce)$ is equivalent to $\ce^{\ca_\lambda}$ by the restriction-Kan extension paradigm and the universal property of $\mathsf{Ind}_\lambda$-completion.  The inclusion $i: \text{Acc}_\omega(\ca, \ce) \hookrightarrow \ce^{\ca_\lambda}$, preserves all colimits and finite limits, this is easy to show and depends on the one hand on how colimits are computed in this category of functors, and on the other hand on the fact that in a topos directed colimits commute with finite limits. Thus $\text{Acc}_\omega(\ca, \ce)$ amounts to a coreflective subcategory of a topos whose associated comonad is left exact. So by \cite{sheavesingeometry}[V.8 Thm.4], it is a topos.
\end{rem}

\begin{rem}[Action of $\boxtimes$ on functors] \label{tensor1cells}
Let $f: \ca \to \ce$ be in $\text{Acc}_\omega(\ca, \ce)$ and let $g:\ce \to \cf$ be a geometric morphism. We must define a geometric morphism $$f \boxtimes g: \text{Acc}_\omega(\ca, \ce) \to \text{Acc}_\omega(\cb, \cf).  $$
We shall describe the left adjoint $(f \boxtimes g)^*$ (which goes in the opposite direction ($f \boxtimes g)^*:  \text{Acc}_\omega(\cb, \cf) \to \text{Acc}_\omega(\ca, \ce)$) by the following equation: $$(f \boxtimes g)^*(s)= g^* \circ s \circ f.$$
\end{rem}

\begin{prop} \label{tensored}
Topoi has tensors over $\text{Acc}_\omega$.
\end{prop}
\begin{proof}
Putting together the content of \ref{boxtimesdef}, \ref{tensorscottconstructionwelldefined} and \ref{tensor1cells}, we only need to show that $\boxtimes$ has the correct universal property, that is: $$\text{Topoi}(\ca \boxtimes \ce, \cf) \cong \text{Acc}_\omega(\ca, \text{Topoi}(\ce, \cf)). $$ When we spell out the actual meaning of the equation above, we discover that we did all the relevant job in the previous remarks. Indeed the biggest obstruction was the well-posedness of the definition.

\begin{align*} 
\text{Topoi}(\ca \boxtimes \ce, \cf) \cong &  \text{Cocontlex}(\cf, \ca \boxtimes \ce)  \\ 
 \cong & \text{Cocontlex}(\cf, \text{Acc}_\omega(\ca, \ce))  \\
\cong & \text{Cat}_{\text{cocontlex,}\text{acc}_\omega}(\cf \times \ca, \ce) \\
\cong &  \text{Acc}_\omega(\ca, \text{Cocontlex}(\cf,\ce))  \\
 \cong &  \text{Acc}_\omega(\ca, \text{Topoi}(\ce,\cf)).
\end{align*}
\end{proof}

\section{The Scott adjunction revisited}

\begin{rem}[Yet another proof of the Scott adjunction]
Let us start by mentioning that we can re-obtain the Scott adjunction directly from the fact that Topoi is tensored over $\text{Acc}_\omega$. Indeed if we evaluate the equation in \ref{tensored} when $\ce$ is the terminal topos Set, $$\text{Topoi}(\ca \boxtimes \Set, \cf) \cong \text{Acc}_\omega(\ca, \text{Topoi}(\Set, \cf)) $$  we obtain precisely the statement of the Scott adjunction, $$\text{Topoi}(\mathsf{S}(\ca), \cf) \cong \text{Acc}_\omega(\ca, \mathsf{pt}(\cf)). $$ Being tensored over $\text{Acc}_\omega$ means in a way to have a relative version of the Scott adjunction.
\end{rem}

\begin{rem}\label{idiot}
Among natural numbers, we find extremely familiar the following formula, \[(30 \times 5) \times 6 = 30 \times (5 \times 6).\]
Yet, this formula yields an important property of the category of sets. Indeed $\Set$ is tensored over itself and the tensorial structure is given by the product. The formula above tells us that the tensorial structure of $\Set$ associates over its product.
\end{rem}

\begin{rem}[Associativity of $\boxtimes$ with respect to $\times$]
Recall that the category of topoi has products, but they are very far from being computed as in Cat. Pitts has shown \cite{pitts1985product} that $\ce \times \cf \cong \text{Cont}(\ce^\circ, \cf)$. This description, later rediscovered by Lurie, is crucial to get a slick proof of the statement below. 
\end{rem}

\begin{prop} Let $\ca$ be a finitely accessinble category. Then,
\[\ca \boxtimes (\ce \times \cf) \simeq ( \ca \boxtimes \ce ) \times \cf.\]
\end{prop}
\begin{proof}
We show it by direct computation.
\begin{align*} 
\ca \boxtimes (\ce \times \cf) \simeq &  \text{Acc}_\omega (\ca, \text{Cont}(\ce^\circ, \cf) )  \\ 
\simeq & \text{Cat}(\ca_\omega, \text{Cont}(\ce^\circ, \cf) ) \\
\simeq & \text{Cont}(\ce^\circ, \text{Cat}(\ca_\omega,\cf)) \\
\simeq & \text{Cont}(\ce^\circ, \text{Acc}_\omega (\ca,\cf)) \\
\simeq & ( \ca \boxtimes \cf ) \times \ce.
\end{align*}
\end{proof}

\begin{rem}
Similarly to Rem. \ref{idiot}, the following display will appear completely trivial,
\[(30 \times 1) \times 6 = 30 \times (1 \times 6).\]
Yet, we can get inspiration from it, to unveil an important simplification of the tensor $\ca \boxtimes \ce$. We will show that it is enough to know the Scott topos $\mathsf{S}(\ca)$ to compute $\ca \boxtimes \ce$, at least when $\ca$ is finitely accessible.
\end{rem}

\begin{prop}[Interaction between $\boxtimes$ and Scott]
Let $\ca$ be a finitely accessinble category. Then,
$$\ca \boxtimes (-) \cong \mathsf{S}(\ca) \times (-).$$
\end{prop}
\begin{proof}
$$\ca \boxtimes (-) \cong \ca \boxtimes (\Set \times -) \cong (\ca \boxtimes \Set) \times (-).$$
\end{proof}

\begin{prop}[Powers and exponentiable Scott topoi]
Let $\ca$ be a finitely accessible category. Then Topoi has powers with respect to $\ca$.
Moreover, $\ce^\ca$ is given by the exponential topos $\ce^{\mathsf{S}(\ca)}.$
\end{prop}
\begin{proof}
The universal property of the power object $\ce^\ca$ is expressed by the following equation,
$$\text{Topoi}(\cf, \ce^\ca) \cong \text{Acc}_\omega(\ca, \text{Topoi}(\cf, \ce)).$$ Now, because we have tensors, this is saying that $\text{Topoi}(\cf, \ce^\ca) \cong \text{Topoi}(\ca \boxtimes \cf, \ce)).$ Because of the previous proposition, we can  gather this observation in the following equation.
$$\text{Topoi}(\cf, \ce^\ca) \cong  \text{Topoi}(\mathsf{S}(\ca) \times \cf, \ce)).$$ This means that $\ce^\ca$ has the same universal property of the topos $\ce^{\mathsf{S}(\ca)}$ and thus exists if and only if the latter exists. By the well known characterization of exponentiable topoi, this happens if and only if $\mathsf{S}(\ca)$ is continuous which is of course true for presheaf categories.
\end{proof}

\chapter{Toolbox}\label{toolbox}

This chapter contains technical results on the Scott adjunction that will be extensively employed for more qualitative results. We study the behavior of $\mathsf{pt}$, $\mathsf{S}$ and $\eta$, trying to discern all their relevant properties. 
Before continuing, we briefly list the main results that we will prove in order to facilitate the consultation.
\begin{enumerate}
\item[\ref{embeddings}] $\mathsf{pt}$ transforms geometric embeddings in fully faithful functors.
\item[\ref{localicmaps}] $\mathsf{pt}$ transforms localic morphisms in faithful functors.
\item[\ref{surjections}] $\mathsf{S}$ maps pseudo-epis (of Cat) to geometric surjections.
\item[\ref{reflective}] $\mathsf{S}$ maps reflections to geometric embeddings.
\item[\ref{topological}] Introduces and studies the notion of topological embeddings between accessible categories.
\item[\ref{ff}] $\eta$ is faithful (and iso-full) if and only if $\ca$ has a faithful (and iso-full) functor into a finitely accessible category.
\end{enumerate}

\section{Embeddings \& surjections}

\begin{rem}
Observe that since $\mathsf{S}$ is a left adjoint, it preserve pseudo epimorphisms, analogously $\mathsf{pt}$ preserves pseudo monomorphisms. Props. \ref{embeddings} and \ref{surjections} might be consequences of this observation, but we lack an explicit description of pseudo monomorphisms and pseudo epimorphisms in both categories. Notice that, instead, \ref{reflective} represents a non-trivial behavior of $\mathsf{S}$.
\end{rem}

\subsection{On the behavior of $\mathsf{pt}$}
The functor $\mathsf{pt}$ behaves nicely with various notions of \textit{injective} or \textit{locally injective} geometric morphism.
\newline

\subsubsection{Geometric embeddings} 
Geometric embeddings between topoi are a key object in topos theory. Intuitively, they represent the proper notion of subtopos. The most common example of geometric embedding is the one of presentation of a topos in the sense of \ref{presentation of topos}. It is a well known fact that subtopoi of a presheaf topos $\Set^C$ correspond to Grothendieck topologies on $C$ bijectively.

\begin{prop}  \label{embeddings}
 $\mathsf{pt}$ sends geometric embeddings to fully faithful functors.
\end{prop}
\begin{proof}
This is a relatively trivial consequence of the fact that the direct image functor is fully faithful but we shall include the proof in order to show a standard way of thinking. Let $i^*: \cg \leftrightarrows \ce: i_*$ be a geometric embedding. Recall, this means precisely that the $\ce$ is reflective in $\cg$ via this adjunction, i.e. the direct image is fully faithful. Let $p,q: \Set \rightrightarrows \ce$ be two points, or equivalently let $p^*, q^*: \ce \rightrightarrows \Set$ be two cocontinuous functors preserving finite limits. And let $\mu, \nu: p^* \Rightarrow q^*$ be two natural transformation between the points. 

\begin{center}
\begin{tikzcd}[column sep=2cm]
{\bf Set} &\ar[l, shift left=\dist, "q^*"{name=q}]\ar[l, shift right=\dist, "p^*"'{name=p}] \mathcal E \ar[r, hook, shift right=\dist, "i_*"']&\ar[l, shift right=\dist, "i^*"'] \mathcal G
\ar[from=p, to=q, Rightarrow, shorten <=3pt, shorten >=3pt, shift left=1em, "\nu"]
\ar[from=p, to=q, Rightarrow, shorten <=3pt, shorten >=3pt, shift right=1em, "\eta"']
\end{tikzcd}
\end{center}

The action of $\mathsf{pt}(i)$ on this data is the following. It maps $p^*$ to $p^*i^*$ while $\mathsf{pt}(i)(\mu)$ is defined by whiskering $\mu$ with $i^*$ as pictured by the diagram below.

\begin{center}
\begin{tikzcd}[row sep=2cm]
p^*\ar[d, bend left, Rightarrow, "\nu"]\ar[d, bend right, Rightarrow, "\mu"'] & p^*i^*\ar[d, bend left, Rightarrow, "\mu_{i^*}"]\ar[d, bend right, Rightarrow, "\nu_{i^*}"]\\
q^* & q^*i^*
\end{tikzcd}
\end{center}

Now, observe that $\mu \cong \mu_{i^*i_*}$ because $i^*i_*$ is isomorphic to the identity. This proves that $\mathsf{pt}(i)$ is faithful, in fact  $\mathsf{pt}(i)(\mu) =  \mathsf{pt}(i)(\nu)$ means that $\mu_{i^*} = \nu_{i^*}$, this implies that $\mu_{i^*i_*} = \nu_{i^*i_*}$, and so $\mu = \nu$. A similar argument shows that $\mathsf{pt}(i)$ is full (using that $i_*$ is full).
\end{proof}

\subsubsection{Localic morphisms} Localic topoi are those topoi that appear as the category of sheaves over a locale. Those topoi have a clear topological meaning and represent a quite concrete notion of generalized space. Localic morphisms are used to generalize the notion of localic topos; a localic morphism $f: \cg \to \ce$ attests that there exist an internal locale $L$ in $\ce$ such that $\cg \simeq \mathsf{Sh}(L, \ce)$. In accordance with this observation, a topos $\cg$ is localic if and only if the essentially unique geometric morphism $ \cg \to \Set$ is localic.

\begin{defn} A morphism of topoi $f: \cg \to \ce$ is localic if every object in $\cg$ is a subquotient of an object in the inverse image of $f$.
\end{defn} 

\begin{prop}  \label{localicmaps}
$\mathsf{pt}$ sends localic geometric morphisms to faithful functors.
\end{prop}
\begin{proof}
Consider a localic geometric morphism $f: \cg \to \ce$. We shall prove that $\mathsf{pt}f$ is faithful on points. In order to do so,  let $p,q: \Set \rightrightarrows \ce$ be two points, or equivalently let $p^*, q^*: \ce \rightrightarrows \Set$ be two cocontinuous functors preserving finite limits. And let $\mu, \nu: p^* \Rightarrow q^*$ be two natural transformation between the points. 

\begin{center}
\begin{tikzcd}[column sep=2cm]
{\bf Set} &\ar[l, shift left=\dist, "q^*"{name=q}]\ar[l, shift right=\dist, "p^*"'{name=p}] \mathcal E \ar[r, hook, shift right=\dist, "f_*"']&\ar[l, shift right=\dist, "f^*"'] \mathcal G
\ar[from=p, to=q, Rightarrow, shorten <=3pt, shorten >=3pt, shift left=1em, "\nu"]
\ar[from=p, to=q, Rightarrow, shorten <=3pt, shorten >=3pt, shift right=1em, "\mu"']
\end{tikzcd}
\end{center}

We need to prove that if $\mu_{f^*} = \nu_{f^*}$, then $\mu = \nu$. In order to do so, let $e$ be an object in $\ce$. Since $f$ is a localic morphism, there is an object $g \in \cg$ and an epimorphism $l: f^*(g) \twoheadrightarrow e$\footnote{This is not quite true, we know that $e$ is a subquotient of $f^*(g)$, in the general case the proof gets a bit messier to follow, for this reason we will cover in detail just this case.}.

\begin{center}
\begin{tikzcd}[row sep=2cm]
p^*(e)\ar[d, bend left, Rightarrow, "\nu"]\ar[d, bend right, Rightarrow, "\mu"'] & p^*i^*(g)\ar[d, bend left, Rightarrow, "\mu_{i^*g}"]\ar[d, bend right, Rightarrow, "\nu_{i^*g}"] \ar[l, "p^*(l)"]\\
q^*(e) & q^*i^*(g) \ar[l, "q^*(l)"]
\end{tikzcd}
\end{center}

Now, we know that $\mu \circ p^*(l) = q^*(l) \circ \mu_{i^*g}$ and $\nu \circ p^*(l) = q^*(l) \circ \nu_{i^*g}$, because of the naturality of $\mu$ and $\nu$. Since  $\mu_{i^*} = \nu_{i^*}$, we get \[\mu \circ p^*(l) = q^*(l) \circ \mu_{i^*g} = q^*(l) \circ \nu_{i^*g} = \nu \circ p^*(l) .\]

Finally observe that $p^*(l)$ is an epi, because $p^*$ preserves epis, and thus we can cancel it, obtaining the thesis.

\end{proof}

\subsubsection{Geometric surjections}

\begin{prop} Let $f: \cg \to \ce$ be a geometric morphism. The following are equivalent.
\begin{itemize}
\item For every point $j: \Set \to \ce$ the pullback $\cg \times_{\ce} \Set$ has a point.
\item $\mathsf{pt}(f)$ is surjective on objects.
\end{itemize}
\end{prop}
\begin{proof}
Trivial.
\end{proof}

\subsection{On the behavior of $\mathsf{S}$}
The functor $\mathsf{S}$ behaves nicely with respect to epis, as expected. It does not behave nicely with any notion of monomorphism. In the next section we study those accessible functors $f$ such that $\mathsf{S}(f)$ is a geometric embedding.
\newline

\begin{prop}\label{surjections} $\mathsf{S}$ maps pseudo-epis (of Cat) to geometric surjections.
\end{prop}
\begin{proof}
See \citep{laxepi}[4.2].
\end{proof}


%
%
%
%

\subsection{Topological embeddings}\label{topological}

\begin{defn}
Let $f: \ca \to \cb$ be a $1$-cell in $\text{Acc}_{\omega}$. We say that $f$ is a topological embedding if $\mathsf{S}(f)$ is a geometric embedding.
\end{defn}

This subsection is devoted to describe topological embeddings between accessible categories with directed colimits. The reader should expect this description to be highly nontrivial and rather technical, because $\mathsf{S}$ is a left adjoint and is not expected to have nice behavior on any kind of monomorphism.

Fortunately we will manage to provide some useful partial results. Let us list the lemmas that we are going to prove.
\begin{enumerate}
	\item[\ref{necessary}] a necessary condition for a functor to admit a topological embedding into a finitely accessible category
	\item[\ref{sufficient}] a sufficient and quite easy to check criterion for a functor to be a topological embedding
	\item[\ref{intofinitely}]  a full description of topological embeddings into finitely accessible categories

\end{enumerate}

\begin{rem}
Topological embeddings into finitely accessible categories $i: \ca \to \mathsf{Ind}(C)$ are very important because $\mathsf{S}(i)$ will describe, by definition, a subtopos of $\Set^C$. This means that there exist a topology $J$ on $C$ such that $\mathsf{S}(\ca)$ is equivalent to $\mathsf{Sh}(C,J)$, this leads to concrete presentations of the Scott topos.
\end{rem}

\subsubsection{A necessary condition}\label{necessary}

\begin{lem}[A necessary condition]
If $\ca$ has a fully faithful topological embedding $f: \ca \to \mathsf{Ind}(C)$ into a finitely accessible category, then $\eta_A : \ca \to \mathsf{ptS}(\ca)$ is fully faithful.
\end{lem}
\begin{proof}
Assume that $\ca$ has a topological embedding $f: \ca \to \mathsf{Ind}(C)$ into a finitely accessible category. This means that $\mathsf{S}(f)$ is a geometric embedding. Now, we look at the following diagram.
	\begin{center}
	\begin{tikzcd}
\ca \arrow[rr, "f" description] \arrow[dd, "\eta_\ca"]              &  & \mathsf{Ind}(C) \arrow[dd, "\eta_{\mathsf{Ind}(C)}" description, two heads, hook] \\
                                                                    &  &                                                                                   \\
\mathsf{pt} \mathsf{S}(\ca) \arrow[rr, "\mathsf{ptS}f" description] &  & \mathsf{pt} \mathsf{S}(\mathsf{Ind}(C))                                          
\end{tikzcd}
	\end{center}
\ref{trivial} implies that $\eta_{\mathsf{Ind}(C)}$ is an equivalence of categories, while \ref{embeddings} implies that $\mathsf{ptS}(f)$ is fully faithful. Since also $f$ is fully faithful, $\eta_\ca$ is forced to be fully faithful.
\end{proof}

\subsubsection{A sufficient condition}\label{sufficient}


\begin{thm}\label{reflective}
Let $i: \ca \to \cb$ be a $1$-cell in $\text{Acc}_\omega$ exhibiting $\ca$ as a reflective subcategory of $\cb$ \[L: \cb \leftrightarrows \ca: i. \] Then $i$ is a topological embedding.
\end{thm}


 \begin{proof}
 We want to show that $\mathsf{S}(i)$ is a geometric embedding. This is equivalent to show that the counit $i^*i_*(-) \Rightarrow (-)$ is an isomorphism. Going back to \ref{f_*}, we write down the obvious computations, \[i^*i_*(-) \cong (i^* \circ r_\cb \circ \ran_i \circ \iota_\ca) (-).\] Now, observe that since $i$ has a left adjoint $L$ the operator $\ran_i$ just coincides with $(-) \circ L$, thus we can elaborate the previous equation as follows.
 \[ (i^* \circ r_\cb \circ \ran_i \circ  \iota_\ca)  (-) \cong (i^*\circ r_\cb)(- \circ L),\]
 Now, $(- \circ L)$ will preserve directed colimits because is the composition of a cocontinuous functor with a functor preserving directed colimits. This means that it is a fixed point of $r_\cb$.
 \[ (i^*\circ r_\cb)((-) \circ L) \cong i^*((-) \circ L) \cong  (-) \circ L \circ i \cong (-).\] The latter isomorphism is just the definition of reflective subcategory. This concludes the proof.
 \end{proof}

\subsubsection{Into finitely accessible categories}\label{intofinitely}

\begin{thm}
$f: \ca \to \mathsf{Ind}(C)$ is a topological embedding into a finitely accessible category if and only if, for all $p: \ca \to \Set$ preserving directed colimits, the following equation holds (whenever well defined), \[\lan_i(\ran_f(p) \circ i) \circ f \cong p.\]
\end{thm}

\begin{proof}
The result follows from the discussion below.
\end{proof}

\begin{rem}[$f_*$ and finitely accessible categories]

Given a $1$-cell $f: \ca \to \cb$ in $\text{Acc}_{\omega}$, we experienced that it can be quite painful to give an explicit formula for the direct image functor $f_*$. In this remark we improve the formula provided in \ref{f_*} in the special case that the codomain if finitely accessible. In order to do so we study the diagram of \ref{f_*}. To settle the notation, call $f: \ca \to \mathsf{Ind}(C)$ our object of study and $i: C \to \mathsf{Ind}(C)$ the obvious inclusion.

\begin{center}
\begin{tikzcd}
\mathsf{S}\ca \arrow[ddd, "\iota_\ca"] \arrow[rr, "f_*" description, bend right] &  & \mathsf{S}\mathsf{Ind}(C) \arrow[ddd, "\iota_{\mathsf{Ind}(C)}"] \arrow[ll, "f^*"'] \\
 &  &  \\
 &  &  \\
{\P(\ca)} \arrow[rr, "\ran_f" description, bend right] \arrow[rr, "\lan_f" description, bend left] &  & {\P(\mathsf{Ind}(C))} \arrow[ll, "f^*" description]
\end{tikzcd}
\end{center}

We are use to this diagram from \ref{f_*}, where we learnt also the following formula \[f_* \cong r_\cb \circ \ran_f \circ \iota_\ca.\] We now use the following diagram to give a better description of the previous equation.

\begin{center}
\begin{tikzcd}
\mathsf{S}\ca \arrow[ddd, "\iota_\ca"] \arrow[rr, "f_*" description, bend right]                       &  & \mathsf{S}\mathsf{Ind}(C) \arrow[ddd, "\iota_{\mathsf{Ind}(C)}"] \arrow[ll, "f^*"'] \arrow[rr, "i^*" description, two heads, tail, bend right] &  & \Set^C \arrow[ll, "\lan_i" description, two heads, tail] \arrow[llddd, "\lan_i" description] \\
                                                                                                       &  &                                                                                                                                                &  &                                                                                              \\
                                                                                                       &  &                                                                                                                                                &  &                                                                                              \\
{\P(\ca)} \arrow[rr, "\ran_f" description, bend right] \arrow[rr, "\lan_f" description, bend left] &  & {\P(\mathsf{Ind}(C))} \arrow[ll, "f^*" description] \arrow[rruuu, "i^*" description, bend right]                                           &  &                                                                                             
\end{tikzcd}
\end{center}

We claim that in the notations of the diagram above, we can describe the direct image $f_*$ by the following formula, \[f_* \cong \lan_i \circ  i^* \circ \ran_f \circ \iota_\ca,\]

this follows from the observation that $r_{\mathsf{Ind}(C)}$ coincides with $\lan_i \circ  i^*$ in the diagram about.
\end{rem}

\section{A study of the unit $\eta_\ca$} 

This section is devoted to a focus on the unit of the Scott adjunction. We will show that good properties of $\eta_\ca$ are related to the existence of \textit{finitely accessible representations} of $\ca$.
A weaker version of the following proposition appeared in \citep{simon}[2.6]. Here we give a different proof, that we find more elegant and provide a stronger statement.

\begin{prop}\label{ff} The following are equivalent:
\begin{enumerate}
	\item  The unit at $\ca$ of the Scott adjunction $\ca \to \mathsf{pt}  \mathsf{S} \ca$ is faithful (and iso-full);
	\item  $\ca$ admits a faithful (and iso-full) functor $f: \ca \to \mathsf{Ind}(C)$ preserving directed colimits;
\end{enumerate}
\end{prop}
\begin{proof}
\begin{itemize}
	\item[$1) \Rightarrow 2)$] Assume that $\eta_\ca$ is faithful. Recall that any topos admits a geometric embedding in a presheaf category, this is true in particular for $\mathsf{S}(\ca)$. Let us call $\iota$ some such geometric embedding $\iota: \mathsf{S}(\ca) \to \Set^X$. Following \ref{embeddings} and \ref{trivial}, $\mathsf{pt}(\iota)$ is a fully faithful functor into a finitely accessible category $\mathsf{pt}(\iota): \mathsf{ptS}(\ca) \to \mathsf{Ind}(X)$. Thus the composition $\mathsf{pt}(\iota) \circ \eta_A$ is a faithful functor into a finitely accessible category \[\ca \to \mathsf{pt}  \mathsf{S} \ca \to \mathsf{Ind}(X).\] Obverse that if $\eta_\ca$ is iso-full, so is the composition $\mathsf{pt}(\iota) \circ \eta_A$.
	\item[$2) \Rightarrow 1)$] Assume that $\ca$ admits a faithful functor $f: \ca \to \mathsf{Ind}(C)$ preserving directed colimits. Now we apply the monad $\mathsf{pt}\mathsf{S}$ obtaining the following diagram.
	\begin{center}
	\begin{tikzcd}
\ca \arrow[rr, "f" description] \arrow[dd, "\eta_\ca"]              &  & \mathsf{Ind}(C) \arrow[dd, "\eta_{\mathsf{Ind}(C)}" description, two heads, hook] \\
                                                                    &  &                                                                                   \\
\mathsf{pt} \mathsf{S}(\ca) \arrow[rr, "\mathsf{ptS}f" description] &  & \mathsf{pt} \mathsf{S}(\mathsf{Ind}(C))                                          
\end{tikzcd}
	\end{center}
  \ref{trivial} implies that $\eta_{\mathsf{Ind}(C)}$ is an equivalence of categories, thus $\mathsf{ptS}(f) \circ \eta_\ca$ is (essentially) a factorization of $f$. In particular, if $f$ is faithful, so has to be $\eta_\ca$. Moreover, if $f$ is iso-full and faithful, so must be $\eta_\ca$, because this characterizes pseudo-monomorphisms in Cat (and by direct verification also in $\text{Acc}_\omega$).
\end{itemize}
\end{proof}

 \begin{rem}
 If we remove iso-fullness from the statement we can reduce the range of $f$ from any finitely accessible category to the category of sets.

\begin{prop}\label{ff} The following are equivalent:
\begin{enumerate}
	\item $\ca$ admits a faithful functor $f: \ca \to \Set$ preserving directed colimits.
	\item  $\ca$ admits a faithful functor $f: \ca \to \mathsf{Ind}(C)$ preserving directed colimits;
\end{enumerate}
\end{prop}
\begin{proof}
The proof is very simple. $1) \Rightarrow 2)$ is completely evident. In order to prove $2) \Rightarrow 1)$, obverse that since $\mathsf{Ind}(C)$ is finitely accessible, there is a faithful functor $ \cy :  \mathsf{Ind}(C) \to \Set$ preserving directed colimits given by \[ \cy := \coprod_{p \in C} \mathsf{Ind}(C)(p, -). \]    The composition $g:= \cy \circ f$ is the desired functor into $\Set$.
\end{proof}
\end{rem}

\chapter{Final remarks and Open Problems}

\begin{rem}[The initial project]
The original intention of this thesis was to exploit the Scott adjunction to obtain a better understanding of categorical model theory, that is the abstract study of accessible categories with directed colimits from a logical perspective. We believe to have partially contributed to this general plan, yet some observations must be made. Indeed, after some time, we came to the conclusion that a lot of foundational issue about our approach needed to be discussed in order to understand what kind of tools we were developing. For example, in the early days of our work, we naively confused the Scott adjunction with the categorified Isbell duality.
\end{rem}
\begin{rem}
 As a result, we mainly devoted our project to the creation of a foundational framework, where it would be possible to organize categorical model theory. In particular, this means that a large part of our original plan is still to be pursued, and of course we tried to do so. Unfortunately, a family of technical results needs to be improved in order to approach sharp and sophisticated problems like \textit{Shelah's categoricity conjecture}.
\end{rem}

\begin{rem}[The rôle of geometry]
 Also, moving to a more foundational framework has raised the importance of the geometric intuition on our work. The interplay between the geometric and logic aspects of our quest has proven to be unavoidable. Ionads, for example, offer a perfect ground to study semantics with a syntax-free approach, and yet the intuition that we had studying them was completely geometric.
 \end{rem}

 \begin{rem}[A new point of view: ionads of models]
 This might be, in a way, one of our main contribution: to shift (at least partially) the object of study from accessible categories with directed colimits to ionad-like objects. They offer a new ground where categorical model theory finds its natural environment. There were already evidences that a similar framework could be relevant. An accessible category with directed colimits comes very often equipped with a forgetful functor $U: \ca \to \Set$ preserving directed colimits. $U$ can be seen as a point \[\bullet \stackrel{U}{\to} \P(\ca)\] in the category of small copresheaves over $\ca$, and its density comonad offer a ionad-like object that might be worthy of study in its own right. 
 \end{rem}

\begin{rem}
 For these reasons, we think that it is time to disclose our work to the community and offer to everybody the chance of improving our results. This chapter is dedicated to listing some of the possible further directions of our work. Obviously, some open problems are just curiosities that haven't found an answers, some others instead are quite relevant and rather technical open problems that would lead to a better understanding of categorical model theory if properly solved.
\end{rem}

\section{Geometry}

\subsection{Connected topoi}

\begin{defn}
A topos $\ce$ is connected if the inverse image of the terminal geometric morphism $\ce \to \Set$ is fully faithful.
\end{defn}

\begin{rem} 
 \cite{elephant2}[C1.5.7] describes the general and relevant properties of connected topoi and connected geometric morphisms. For the sake of this subsection, we can think of a connected topos as the locale of opens of a connected topological space. Indeed, the definition can be reduced to this intuition.
\end{rem}

\begin{question}
\begin{enumerate}
\item[]
\item What kind of Scott topoi are connected?
\item What kind of Isbell topoi are connected?
\end{enumerate}
\end{question}

In the direction of the first question, we can offer a first approximation of the result.

\begin{thm} If $\ca$ is connected then its Scott topos $\mathsf{S}(\ca)$ is connected.
\end{thm}
\begin{proof}
The terminal map $t: \mathsf{S}(\ca) \to \Set$ appears as $\mathsf{S}(\tau)$, where $\tau$ is the terminal map $\tau: \ca \to \cdot$. When $\ca$ is connected $\tau$ is a lax-epi, and $\tau^*$ is fully faithful, \cite{laxepi}.
\end{proof}

\begin{cor}[The JEP implies connectedness of the Scott topos] \label{JEPconnected}
Let $\ca$ be an accessible category where every map is a monomorphism. If $\ca$ has the joint embedding property, then $\mathsf{S}(\ca)$ is connected.
\end{cor}
\begin{proof}
Obviously if $\ca$ has the JEP, it is connected.
\end{proof}

\subsection{Closed sets}

\begin{rem}
A fascinating aspect of general topology is the duality between closed and open sets. Indeed, given a topological space $X$, its poset of opens is dually equivalent to its poset of closed sets \[(-)^c: \mathcal{O}(X) \leftrightarrows \cc(X)^\circ :(-)^c.\]
This resembles the fact that, given a topos $\ce$, its opposite category is monadic over $\ce$ via the subobject classifier, \[\Omega^{(-)}: \ce \leftrightarrows \ce^\circ : \Omega^{(-)}.\]
Indeed, this is compatible with the theory of $n$-topoi, where the generalized universal bundle unifies the role played by $\Omega$ in a Grothendieck topos and by $1$ in a locale. Indeed,  $(-)^c = (-) \Rightarrow 1$ and thus the two results are somewhat the same. Of course, in poset theory the monadicity is much tighter.
\end{rem}

\begin{question}
If the theory of (bounded, generalized) ionads offers us a clear way to picture topoi as categories of open sets with respect to an interior operator. Is there a notion of closure operator that nicely interacts with this theory?
\end{question}

\begin{rem}
Of course, there are natural attempts to answer this question. One could say that a closure operator over a category $X$ is a colex monad $\text{cl} : \P(X) \to \P(X)$. In the classical case of topological spaces there is a bijection between closure and interior operator given by the formula \[ \text{cl} = \text{Int}((-)^c)^c.\]
As discussed in the previous remark, in this case one has that the adjunction $\mathsf{Alg}(\text{cl})^\circ \leftrightarrows \mathsf{coAlg}(\text{Int})$ yields an equivalence of categories.
Assuming that $X$ is a small category, (so that $\P(X)$ has a subobject classifier) one has a similar correspondence between monads and comonads, \[\Omega^{\Omega^{(-)}}: \mathsf{Mnd}(\Set^X) \leftrightarrows \mathsf{coMnd}(\Set^X) : \Omega^{\Omega^{(-)}} \] but such correspondence might not send interior operator to closure ones. Since $\Omega$ is injective, the double dualization monad is somewhat exact, but such an exactness property is not enough for our purposes. We still get an adjunction, \[ \mathsf{Alg}(\Omega^{\Omega^{\text{Int}}})^\circ \leftrightarrows \mathsf{coAlg}(\text{Int}).\] 
\end{rem}

\section{Logic}

\subsection{Categoricity spectra and Shelah's conjecture}

We have already mentioned in Chap. \ref{logical} that we believe it is possible to use our technology to approach Shelah's categoricity conjecture. In the 60's Morley proved the following two theorems \cite{chang1990model}.

\begin{thm}[Morley $(\downarrow)$] Let $\mathbb{T}$ be a complete first order theory in a countable language with infinite models. If $\mathbb{T}$ is categorical in some cardinal $\kappa$, then $\mathbb{T}$ is categorical in any cardinal $\omega_1 \leq \mu \leq \kappa.$
\end{thm}

\begin{thm}[Morley $(\uparrow)$]
Let $\mathbb{T}$ be a complete first order theory in a countable language with infinite models. If $\mathbb{T}$ is categorical in some cardinal $\kappa$, then $\mathbb{T}$ is categorical in any cardinal above $\kappa$.
\end{thm}

Adapting the theory of  Abstract elementary classes \cite{baldwin2009categoricity} to accessible categories, Shelah's categoricity conjecture has generically the following form.

\begin{conj}
Let $\ca$ be a \textit{nice} accessible category with directed colimits having just one model of presentability rank $\kappa$; then $\ca$ must have at most one model of presentability rank $\lambda$ for every regular cardinal $\lambda \geq \kappa$.
\end{conj}

Indeed it is very non-trivial to prove this statement as such and one can provide weaker versions of the statement above as follows.

\begin{conj}[Weaker conjecture]
Let $\ca$ be a \textit{nice} accessible category with directed colimits having just one model of presentability rank $\kappa$; then $\ca$ must be categorical in unboundedly many cardinals higher than $\kappa$.
\end{conj}

\begin{rem}
One can even prove such a weaker version of the conjecture via the Hanf number \cite{baldwin2009categoricity}[Thm. 4.18], but it is in general very hard to find the Hanf number of an abstract elementary category. For this reason, we will describe a totally different path.
\end{rem}

\begin{rem}
We should mention that several approximations of this results have appeared in the literature. Unfortunately, none of these proofs is deeply categorical, or appears natural to us. Our initial aim was not only to prove Shelah's categoricity conjecture, but also to accommodate it in a framework in which the statement and the proof could look natural.
\end{rem}

\begin{rem}[A possible strategy]
Given an accessible category $\ca$, a possible strategy to prove the conjecture could involve the full subcategory $\ca_{\geq \lambda}$ of objects of cardinality at least $\lambda$. In fact, if the inclusion $\ca_{\geq \lambda} \to \ca$ is a topological embedding (for some $\gamma$-Scott adjunction such that $\ca$ is $\gamma$-accessible) in the hypotheses of Thm \ref{kappasaturated}, then $\ca_{\geq \lambda}$ must be categorical in some presentability rank.
\end{rem}

\begin{prop}[Very cheap version of Shelah's categoricity conjecture]
Let $\ca$ be an $\mu$-accessible category and let $\ca_{\geq \lambda}$ be its full subcategory $i: \ca_{\geq \lambda} \hookrightarrow \ca$ of objects of presentability rank at least $\lambda$. If the following holds:
\begin{enumerate}
	\item $i$ presents $\ca_{\geq \lambda}$ as a category of $\mu$-saturated objects;
	\item $\ca_\mu$ has the joint embedding property;
	\item $\eta_{\ca_{\geq \lambda}}$\footnote{The unit of the $\mu$-Scott adjunction.} is iso-full, faithful and surjective on objects.

\end{enumerate}
Then $\ca$ must be categorical in some presentability rank higher than $\lambda$.
\end{prop}
\begin{proof}
Apply \ref{kappasaturated}.
\end{proof}

\begin{rem}[A comment on the strategy above]
Somehow the missing brick in our proposition above is to show that when $\ca$ is categorical in a presentability rank $\kappa$, then there must exists a cardinal $\lambda > \kappa$ such that $i: \ca_{\geq \lambda} \hookrightarrow \ca$ is a category of $\mu$-saturated objects. Indeed many model theorists will recognize this as a relatively convincing statement. Of course, this strategy shouldn't be taken too strictly; if $\ca_{\geq \lambda}$ is not a category of $\mu$-saturated objects, one can try and change it with a sharper subcategory.
\end{rem}

\begin{rem}
The previous remark shows one of the most interesting open problems of the thesis, which is to provide a refinement of Thm. \ref{reflective} in the case in which the inclusion $i: \ca \to \cb$ is not weakly algebraically reflective.
\end{rem}

\subsection{Cosimplicial sets and Indiscernibles}

Indiscernibles are a classical tool in classical model theory and were introduced by Morley while working on what now is knows under the name of Morley's categoricity theorem. Makkai rephrased this result in \citep{Makkaipare} in the following way.

\begin{thm} Let $\mathsf{pt}(\ce)$ be a large category of points of some topos $\ce$. Then there is a faithful functor $\mathsf{U} :\text{Lin} \to \mathsf{pt}(\ce)$, where $\text{Lin}$ is the category of linear orders and monotone maps.
\end{thm}

A categorical \textit{understanding} of this result was firstly attempted in an unpublished work by Beke and Rosický. We wish we could fit this topic in our framework. The starting point of this process is the following observation.

\begin{thm}
$\mathsf{S}(\text{Lin})$ is $\Set^{\Delta}$.
\end{thm}
\begin{proof}
There is not much to prove, $\text{Lin}$ is finitely accessible and $\Delta$ (the simplex category) coincides with its full subcategory of finitely presentable objects, thus the theorem is a consequence of Rem. \ref{trivial}.
\end{proof}

\begin{rem}
By the Scott adjunction, any functor of the form $\text{Lin} \to \mathsf{pt}(\ce)$ corresponds to a geometric morphism $\Set^{\Delta} \to \ce$.  
\end{rem}

\begin{question} 
\begin{enumerate}
	\item[]
	\item Is it any easier to show the existence of $\mathsf{U}$ by providing a geometric morphism  $\Set^{\Delta} \to \ce$? 
	\item Is it true that every such $\mathsf{U}$ is of the form $\mathsf{pt}(f)$ for a localic geometric morphism $f: \Set^{\Delta} \to \ce$.
	\end{enumerate}
\end{question}

\section{Category Theory}

\subsection{The $\ce$-Scott adjunction}\label{relative}

It might be possible to develop $\ce$-relative a version of the Scott adjunction, where $\ce$ is a Grothendieck topos. The relevant technical setting should be contained in   \cite{borceux1996enriched} and \cite{borceux1998theory}.

\begin{conj}There is a $2$-adjunction

\[\mathsf{S}_{\ce}: \ce\text{-Acc}_{\omega} \leftrightarrows \text{Topoi}_{/\ce}: \mathsf{pt}_\ce\]

Moreover, if $\ce$ is a $\kappa$-topos, the relative version of $\lambda$-Scott adjunction holds for every $\omega \leq \lambda \leq \kappa$.
\end{conj}

\begin{rem}
Let us clarify what is understood and what is to be understood of the previous conjecture.
\begin{enumerate}
	\item By  $\text{BTopoi}_{/\ce}$ we mean the full $2$-subcategory of $\text{Topoi}_{/\ce}$ consisting of those topoi bounded over $\ce$. Recall that $\text{Topoi}$ coincides with $\text{BTopoi}_{/\Set}$.
	\item By $\mathsf{pt}_\ce$ we mean the relative points functor $\text{Topoi}_{/\ce}(\ce, -).$
	\item It is a bit unclear how $\text{Topoi}_{/\ce}(\ce, -)$ should be enriched over $\ce$, maybe it is fibered over it?
\end{enumerate}
\end{rem}

\section{Sparse questions}

\begin{question}
\begin{enumerate}
	\item[]
\item If $f$ is faithful, is $\mathsf{S}(f)$ localic?
\item Is $\eta$ always a topological embedding?
\end{enumerate}
\end{question}

\bibliography{thebib}
\bibliographystyle{alpha}

\end{document}